\documentclass[10pt,a4paper]{article}
\usepackage{amsmath,amsfonts,amssymb,amsthm}
\usepackage{paralist}
\usepackage{titlesec}

\newenvironment{mabstract}
{\begin{quote}\small {\bfseries Abstract.}}{\end{quote}\par}

\newenvironment{mkeywords}
{\begin{quote}\small {\bfseries Keywords.}}{\end{quote}\par}

\newenvironment{msubjclass}
{\begin{quote}\small {\bfseries AMS Subject Classification.}}{\end{quote}\par}

\newtheorem{theorem}{Theorem}[section]

\newtheorem{proposition}{Proposition}[section]

\newcounter{appendixcounter}

\newtheorem{lemma}{Lemma}[section]

\newtheorem{lemmaappendix}{Lemma}[appendixcounter]

\newtheorem{assumption}{Assumption}[section]

\newtheorem{assumptionappendix}{Assumption}[appendixcounter]

\newtheorem{corollary}{Corollary}[section]

\newtheorem*{remark}{Remark}

\newtheorem*{rproof}{\bfseries Proof}

\titleformat{\section}[hang]{\bf\Large}{\thesection.}{0.4em}{}
\titleformat{\subsection}[hang]{\bf\large}{\thesubsection.}{0.4em}{}

\begin{document}

\title{Stability of Optimal Filter Higher-Order Derivatives}

\author{Vladislav Z. B. Tadi\'{c}
\thanks{School of Mathematics, University of Bristol,
Bristol, United Kingdom
(email: v.b.tadic@bristol.ac.uk). }
\and
Arnaud Doucet\thanks{Department of Statistics,
University of Oxford, Oxford, United Kingdom
(doucet@stats.ox.ac.uk). } }

\date{}

\maketitle

\begin{mabstract}
In many scenarios, a state-space model depends on a parameter which needs to be inferred
from data.
Using stochastic gradient search and the optimal filter first-order derivatives,
the parameter can be estimated online.
To analyze the asymptotic behavior of such methods,
it is necessary to establish results on the existence and stability
of the optimal filter higher-order derivatives.
These properties are studied here.
Under regularity conditions, we show that the optimal filter higher-order derivatives exist and forget initial conditions exponentially fast.
We also show that the same derivatives are geometrically ergodic.
\end{mabstract}

\begin{mkeywords}
State-Space Models, Optimal Filter, Optimal Filter Higher-Order Derivatives,
Forgetting of Initial Conditions, Geometric Ergodicity, Log-Likelihood.
\end{mkeywords}

\begin{msubjclass}
Primary 60G35;
Secondary 62M20, 93E11.
\end{msubjclass}

\section{Introduction}
State-space models, also known as continuous-state hidden Markov models, are a powerful and versatile tool for statistical modeling of
complex time-series data and stochastic dynamic systems.
These models can be viewed as a discrete-time Markov process
which are observed only through noisy measurements of their states.
In this context, one of the most important problems is
the optimal estimation of the current state given the noisy measurements
of the current and previous states.
This problem is known as optimal filtering.
Optimal filtering has been studied in a number of papers and books; see, e.g., \cite{cappe&moulines&ryden}, \cite{crisan&rozovskii}, \cite{douc&moulines&stoffer} and
references therein.

In many applications, a state-space model depends on a parameter whose value
needs to be inferred from data.
When the number of data points is large,
it is desirable, for the sake of computational efficiency,
to infer the parameter recursively (i.e., online).
In the maximum likelihood approach,
recursive parameter estimation can be performed using stochastic gradient search,
where the underlying gradient estimation is based on the optimal filter
and its first-order derivatives; see, e.g.,  \cite{legland&mevel2}, \cite{poyiadjis&doucet&singh}, \cite{tadic2}.
In \cite{tadic2}, it has been shown that the asymptotic behavior of recursive maximum likelihood estimation
in finite-state hidden Markov models is closely related to the analytical properties, higher-order differentiability and analyticity, of the underlying log-likelihood rate.
In view of the recent results on stochastic gradient search \cite{tadic&doucet3},
a similar relationship is likely to hold for state-space models.
However, to apply the results of \cite{tadic&doucet3} to
recursive maximum likelihood estimation in state-space models,
it is necessary to establish results on the higher-order differentiability
of the log-likelihood rate for these models.
Since the log-likelihood rate for state-space models is a functional
of the optimal filter, the analytical properties of this rate
are tightly connected to the existence and stability of the optimal filter higher-order
derivatives. Hence, one of the first steps to carry out asymptotic analysis of recursive maximum likelihood
estimation in state-space models is to establish results on the existence and stability of these derivatives.
To the best of our knowledge, this problem has never been addressed before and the results presented here fill this gap in the literature on optimal filtering.


In this paper, the optimal filter higher-order derivatives
and their existence and stability properties are studied.
Under standard stability and regularity conditions,
we show that these derivatives exist and forget initial conditions exponentially fast.
We also show that the optimal filter higher-order derivatives are geometrically ergodic.
The obtained results cover state-space models met in practice and are one of the first stepping stones
to analyze the asymptotic behavior of recursive maximum likelihood estimation in
non-linear state-space modes \cite{tadic&doucet2}.

The paper is organized as follows.
In Section \ref{section1},
the existence and stability of the optimal filter higher-order derivatives
are studied and the main results are presented.
In Section \ref{section2}, the main results are used to study
the analytical properties of log-likelihood for state-space models.
An example illustrating the main results is provided in Section \ref{section3}.
In Sections \ref{section1.1*} -- \ref{section3*}, the main results and their corollaries are proved.

\section{Main Results} \label{section1}

\subsection{State-Space Models and Optimal Filter}\label{subsection1.1}

To specify state-space models and to formulate the problem of optimal filtering,
we use the following notation.
For a set ${\cal Z}$ in a metric space,
${\cal B}({\cal Z})$ denotes the collection of Borel subsets of ${\cal Z}$.
$d_{x}\geq 1$ and $d_{y}\geq 1$ are integers, while
${\cal X}\in{\cal B}(\mathbb{R}^{d_{x} })$ and
${\cal Y}\in{\cal B}(\mathbb{R}^{d_{y} })$.
$P(x,dx')$ is a transition kernel on ${\cal X}$, while
$Q(x,dy)$ is a conditional probability measure on ${\cal Y}$ given $x\in{\cal X}$.
$(\Omega,{\cal F}, P)$ is a probability space.
Then, a state-space model can be defined as
an ${\cal X}\times{\cal Y}$-valued stochastic process
$\{ (X_{n}, Y_{n} ) \}_{n\geq 0}$ on $(\Omega,{\cal F}, P)$ which satisfies
\begin{align*}
	P\left( (X_{n+1}, Y_{n+1} )\in B
	|X_{0:n}, Y_{0:n} \right)
	=
	\int I_{B}(x,y) Q(x,dy) P(X_{n}, dx )
\end{align*}
almost surely for any $B\in{\cal B}({\cal X}\times{\cal Y})$ and $n\geq 0$.
$\{X_{n} \}_{n\geq 0}$ are the unobservable states,
while $\{Y_{n} \}_{n\geq 0}$ are the observations.
One of the most important problems related to state-space models
is the estimation of the current state $X_{n}$
given the state-observations $Y_{1:n}$.
This problem is known as filtering.
In the Bayesian approach, the optimal estimation of $X_{n}$ given $Y_{1:n}$
is based on the (optimal) filtering distribution $P(X_{n}\in dx_{n}|Y_{1:n} )$.
As $P(x,dx')$ and $Q(x,dy)$ are rarely available in practice,
the filtering distribution is usually computed using
some approximate models.

In this paper, we assume that the model $\{ (X_{n}, Y_{n} ) \}_{n\geq 0}$ can be accurately approximated
by a parametric family of state-space models.
To define such a family, we rely on the following notation.
Let $d\geq 1$ be an integer, while
$\Theta\subset\mathbb{R}^{d}$ is a bounded open set.
${\cal P}({\cal X} )$ is the set of probability measures on ${\cal X}$, while
$\mu(dx)$ and $\nu(dy)$ are measures on ${\cal X}$ and ${\cal Y}$ (respectively).
$p_{\theta}(x'|x)$ and $q_{\theta}(y|x)$
are functions which map
$\theta\in\Theta$, $x,x'\in{\cal X}$, $y\in{\cal Y}$
to $[0,\infty )$ and satisfy
\begin{align*}
	\int p_{\theta}(x'|x) \mu(dx')
	=
	\int q_{\theta}(y|x) \nu(dy)
	=
	1
\end{align*}
for all $\theta\in\Theta$, $x\in{\cal X}$.
With this notation, a parametric family of state-space models can be defined as
an ${\cal X}\times{\cal Y}$-valued stochastic process
$\left\{ (X_{n}^{\theta,\lambda}, Y_{n}^{\theta,\lambda} ) \right\}_{n\geq 0}$
on $(\Omega, {\cal F}, P)$
which is parameterized by $\theta\in\Theta$, $\lambda\in{\cal P}({\cal X} )$
and satisfies
\begin{align*}
	&
	P\left( (X_{0}^{\theta,\lambda}, Y_{0}^{\theta,\lambda} ) \in B \right)
	=
	\int\int I_{B}(x,y) q_{\theta}(y|x) \lambda(dx)\nu(dy),
	\\
	&
	P\left(\left. (X_{n+1}^{\theta,\lambda}, Y_{n+1}^{\theta,\lambda} ) \in B\right|
	X_{0:n}^{\theta,\lambda}, Y_{0:n}^{\theta,\lambda} \right)
	=
	\int\int I_{B}(x,y)
	q_{\theta}(y|x) p_{\theta}(x|X_{n}^{\theta,\lambda} ) \mu(dx) \nu(dy),
\end{align*}
almost surely for any $B\in{\cal B}({\cal X}\times{\cal Y})$ and $n\geq 0$.\footnote
{To evaluate the values of $\theta$ for which
$\big\{ (X_{n}^{\theta,\lambda}, Y_{n}^{\theta,\lambda} ) \big\}_{n\geq 0}$
provides the best approximation to
$\left\{ (X_{n}, Y_{n} ) \right\}_{n\geq 0}$,
we usually rely on the maximum likelihood principle.
For further details on maximum likelihood estimation in state-space and hidden Markov models,
see e.g., \cite{cappe&moulines&ryden}, \cite{douc&moulines&stoffer} and references cited therein.
}

To show how the filtering distribution is computed using approximate model
$\left\{ (X_{n}^{\theta,\lambda}, Y_{n}^{\theta,\lambda} ) \right\}_{n\geq 0}$,
we use the following notation.
$r_{\theta}(y,x'|x)$ is the function defined by
\begin{align}\label{1.901}
	r_{\theta}(y,x'|x) = q_{\theta}(y|x') p_{\theta}(x'|x)
\end{align}
for $\theta\in\Theta$, $x,x'\in{\cal X}$, $y\in{\cal Y}$, 
while $r_{\theta,\boldsymbol y}^{m:n}(x'|x)$ is the function recursively defined by
\begin{align}\label{1.501}
	r_{\theta,\boldsymbol y}^{m:m+1}(x'|x) = r_{\theta}(y_{m+1}, x'|x),
	\;\;\;\;\;
	r_{\theta,\boldsymbol y}^{m:n+1}(x'|x)
	=
	\int r_{\theta}(y_{n+1}, x'|x'') r_{\theta,\boldsymbol y}^{m:n}(x''|x) \mu(dx'')
\end{align}
for $n>m\geq 0$ and any sequence $\boldsymbol y = \{y_{n} \}_{n\geq 1}$ in ${\cal Y}$.
$p_{\theta,\boldsymbol y}^{m:n}(x|\lambda )$ and $P_{\theta,\boldsymbol y}^{m:n}(dx|\lambda )$
are the function and the probability measure defined by
\begin{align}\label{1.903}
	p_{\theta,\boldsymbol y}^{m:n}(x|\lambda )
	=
	\frac{\int r_{\theta,\boldsymbol y}^{m:n}(x|x') \lambda(dx') }
	{\int\int r_{\theta,\boldsymbol y}^{m:n}(x''|x') \mu(dx'') \lambda(dx') },
	\;\;\;\;\;
	P_{\theta,\boldsymbol y}^{m:n}(B|\lambda )
	=
	\int_{B} p_{\theta,\boldsymbol y}^{m:n}(x'|\lambda ) \mu(dx')
\end{align}
for $B\in{\cal B}({\cal X} )$,
$\lambda\in{\cal P}({\cal X} )$,
while $P_{\theta,\boldsymbol y}^{m:n}(\lambda )$ is a
`short-hand' notation for $P_{\theta,\boldsymbol y}^{m:n}(dx|\lambda )$.
Then, it can easily be shown that $P_{\theta,\boldsymbol y}^{m:n}(\lambda )$ is the filtering distribution, i.e.,
\begin{align*}
	P_{\theta,\boldsymbol y}^{0:n}(B|\lambda )
	=
	P\left(\left. X_{n}^{\theta,\lambda}\in B
	\right| Y_{1:n}^{\theta,\lambda} = y_{1:n}
	\right)
\end{align*}
for each $\theta\in\Theta$, $B\in{\cal B}({\cal X} )$,
$\lambda\in{\cal P}({\cal X} )$, $n\geq 1$ and any sequence $\boldsymbol y = \{y_{n} \}_{n\geq 1}$
in ${\cal Y}$.
In this context, $\lambda$ can be interpreted as the initial condition of the filtering distribution
$P_{\theta,\boldsymbol y}^{m:n}(\lambda )$.

\subsection{Optimal Filter Higher-Order Derivatives}

Let $p\geq 1$ be an integer. Throughout the paper, we assume that $p_{\theta}(x'|x)$ and $q_{\theta}(y|x)$
are $p$-times differentiable in $\theta$ for each
$\theta\in\Theta$, $x,x'\in{\cal X}$, $y\in{\cal Y}$.

To define the higher-order derivatives of the optimal filter, we use the following notation.
$\mathbb{N}_{0}$ is the set of non-negative integers.
$\boldsymbol 0$ is the element of $\mathbb{N}_{0}^{d}$
whose all components are zero.
For $\boldsymbol\alpha = (\alpha_{1},\dots,\alpha_{d} ) \in\mathbb{N}_{0}^{d}$,
$\boldsymbol\beta = (\beta_{1},\dots,\beta_{d} ) \in\mathbb{N}_{0}^{d}$,
relation $\boldsymbol\beta\leq\boldsymbol\alpha$ is taken component-wise, i.e.,
$\boldsymbol\beta\leq\boldsymbol\alpha$
if and only if $\alpha_{i}\leq\beta_{i}$ for each $1\leq i\leq d$.
For the same $\boldsymbol\alpha$, $\boldsymbol\beta$
satisfying $\boldsymbol\beta\leq\boldsymbol\alpha$,
$\left(\boldsymbol\alpha \atop \boldsymbol\beta \right)$ denotes
the multinomial coefficient
\begin{align*}
	\left( \boldsymbol\alpha \atop \boldsymbol\beta \right)
	=
	\left(\alpha_{1} \atop \beta_{1} \right)
	\cdots
	\left(\alpha_{d} \atop \beta_{d} \right).
\end{align*}
For $\boldsymbol\alpha=(\alpha_{1},\dots,\alpha_{d} )\in\mathbb{N}_{0}^{d}$,
$\theta=(\theta_{1},\dots,\theta_{d} )\in\Theta$,
notation $|\boldsymbol \alpha|$ and $\partial_{\theta}^{\boldsymbol \alpha}$ stand for
\begin{align*}
	|\boldsymbol\alpha| = \alpha_{1} + \cdots + \alpha_{d},
	\;\;\;\;\;
	\partial_{\theta}^{\boldsymbol\alpha}
	=
	\frac{\partial^{|\boldsymbol\alpha|} }{\partial\theta_{1}^{\alpha_{1} } \cdots \partial\theta_{d}^{\alpha_{d} } }.
\end{align*}
$d(p)$ is the number elements in set
$\{\boldsymbol\alpha: \boldsymbol\alpha\in\mathbb{N}_{0}^{d},|\boldsymbol\alpha|\leq p\}$, i.e.,
\begin{align*}
	d(p)
	=
	\sum_{k=0}^{p} \left( d+k-1 \atop k \right).
\end{align*}
${\cal M}_{s}({\cal X} )$ is the set of finite signed measures on ${\cal X}$.
${\cal L}({\cal X})$ is the set of $d(p)$-dimensional finite signed vector measures on ${\cal X}$.
The components of an element of ${\cal L}({\cal X})$ are indexed by multi-indices in $\mathbb{N}_{0}^{d}$
and ordered lexicographically.
More specifically, an element $\Lambda$ of ${\cal L}({\cal X})$ can be denoted by
\begin{align}\label{1.301}
	\Lambda = \left\{ \lambda_{\boldsymbol \alpha}:
	\boldsymbol \alpha \in \mathbb{N}_{0}^{d}, |\boldsymbol \alpha |\leq p \right\},
\end{align}
where $\lambda_{\boldsymbol \alpha}\in{\cal M}_{s}({\cal X} )$
is referred to as the component $\boldsymbol \alpha$ of $\Lambda$.
The components of $\Lambda$ follow lexicographical order, i.e.,
$\lambda_{\boldsymbol \alpha}$ precedes $\lambda_{\boldsymbol \beta}$
if and only if $\alpha_{i}<\beta_{i}$, $\alpha_{j}=\beta_{j}$ for some $i$ and each $j$
satisfying $1\leq i\leq d$, $1\leq j< i$,
where $\boldsymbol \alpha = (\alpha_{1},\dots,\alpha_{d} )$,
$\boldsymbol \beta = (\beta_{1},\dots,\beta_{d} )$.
For $\lambda\in{\cal M}_{s}({\cal X} )$, $\|\lambda\|$ denotes the total variation norm of $\lambda$.
For $\Lambda\in{\cal L}({\cal X} )$, $\|\Lambda\|$ denotes the total variation norm of $\Lambda$
induced by the $l_{\infty}$ vector norm, i.e.,
\begin{align*}
	\|\Lambda\|
	=
	\max\left\{ \|\lambda_{\boldsymbol \alpha} \|:
	\boldsymbol \alpha \in\mathbb{N}_{0}^{d}, |\boldsymbol \alpha|\leq p \right\}
\end{align*}
for $\Lambda$ specified in (\ref{1.301}).
${\cal L}_{0}({\cal X} )$ is the set of $d(p)$-dimensional finite vector measures
whose component $\boldsymbol 0$ is a probability measure
(i.e., $\Lambda$ specified in (\ref{1.301}) belongs to ${\cal L}_{0}({\cal X} )$
if and only if $\lambda_{\boldsymbol 0}\in{\cal P}({\cal X} )$).

We need a few additional notation:
$r_{\theta,y}^{\boldsymbol\alpha}(x|\lambda )$ and $s_{\theta,y}^{\boldsymbol\alpha}(x|\Lambda )$
are the functions defined by
\begin{align}\label{1.703}
	r_{\theta,y}^{\boldsymbol\alpha}(x|\lambda )
	=&
	\int
	\partial_{\theta}^{\boldsymbol\alpha}
	r_{\theta}(y,x|x') \lambda(dx'),
	\;\;\;\;\;
	s_{\theta,y}^{\boldsymbol\alpha}(x|\Lambda )
	=
	\sum_{\stackrel{\scriptstyle\boldsymbol\beta\in\mathbb{N}_{0}^{d} }
	{\boldsymbol\beta\leq\boldsymbol\alpha} }
	\left( \boldsymbol\alpha \atop \boldsymbol \beta \right)
	\frac{r_{\theta,y}^{\boldsymbol\alpha-\boldsymbol\beta}(x|\lambda_{\boldsymbol\beta} ) }
	{ \int r_{\theta,y}^{\boldsymbol 0}(x'|\lambda_{\boldsymbol 0} ) \mu(dx') }
\end{align}
for $\theta\in\Theta$, $x\in{\cal X}$, $y\in{\cal Y}$, $\lambda\in{\cal M}_{s}({\cal X} )$,
$\Lambda=\left\{\lambda_{\boldsymbol\beta}:
\boldsymbol\beta\in\mathbb{N}_{0}^{d}, |\boldsymbol\beta|\leq p \right\} \in{\cal L}_{0}({\cal X} )$,
$\boldsymbol\alpha\in\mathbb{N}_{0}^{d}$,
$|\boldsymbol\alpha|\leq p$.
$f_{\theta,y}^{\boldsymbol\alpha}(x|\Lambda )$ is the function recursively defined by
\begin{align}\label{1.1}
	f_{\theta,y}^{\boldsymbol 0}(x|\Lambda ) = s_{\theta,y}^{\boldsymbol 0}(x|\Lambda ),
	\;\;\;\;\;
	f_{\theta,y}^{\boldsymbol\alpha}(x|\Lambda )
	=
	s_{\theta,y}^{\boldsymbol\alpha}(x|\Lambda )
	-
	\sum_{\stackrel{\scriptstyle\boldsymbol\beta\in\mathbb{N}_{0}^{d}\setminus\{\boldsymbol\alpha \} }
	{\boldsymbol\beta\leq\boldsymbol\alpha} }
	\left( \boldsymbol\alpha \atop \boldsymbol \beta \right)
	f_{\theta,y}^{\boldsymbol\beta}(x|\Lambda)
	\int s_{\theta,y}^{\boldsymbol\alpha-\boldsymbol\beta}(x'|\Lambda) \mu(dx'),
\end{align}
where the recursion is in $|\boldsymbol\alpha|$.\footnote
{In (\ref{1.1}), $f_{\theta,y}^{\boldsymbol 0}(x|\Lambda )$ is the initial condition.
At iteration $k$ of (\ref{1.1}) ($1\leq k\leq p$),
function $f_{\theta,y}^{\boldsymbol\alpha}(x|\Lambda )$
is computed for multi-indices
$\boldsymbol\alpha\in\mathbb{N}_{0}^{d}$, $|\boldsymbol\alpha|=k$
using the results obtained at the previous iterations. }
$R_{\theta,y}^{\boldsymbol\alpha}(dx|\lambda )$,
$S_{\theta,y}^{\boldsymbol\alpha}(dx|\Lambda )$ and
$F_{\theta,y}^{\boldsymbol\alpha}(dx|\Lambda )$ are the elements of ${\cal M}_{s}({\cal X} )$ defined by
\begin{align}\label{1.905}
	R_{\theta,y}^{\boldsymbol\alpha}(B|\lambda )
	=
	\int_{B} r_{\theta,y}^{\boldsymbol\alpha}(x|\lambda ) \mu(dx),
	\;\;\;\:\:
	S_{\theta,y}^{\boldsymbol\alpha}(B|\Lambda )
	=
	\int_{B} s_{\theta,y}^{\boldsymbol\alpha}(x|\Lambda ) \mu(dx),
	\;\;\;\:\:
	F_{\theta,y}^{\boldsymbol\alpha}(B|\Lambda )
	=
	\int_{B} f_{\theta,y}^{\boldsymbol\alpha}(x|\Lambda ) \mu(dx)
\end{align}
for $B\in{\cal B}({\cal X} )$,
while $R_{\theta,y}^{\boldsymbol\alpha}(\lambda )$,
$S_{\theta,y}^{\boldsymbol\alpha}(\Lambda )$,
$F_{\theta,y}^{\boldsymbol\alpha}(\Lambda )$ are a `short-hand' notation for
$R_{\theta,y}^{\boldsymbol\alpha}(dx|\lambda )$,
$S_{\theta,y}^{\boldsymbol\alpha}(dx|\Lambda )$,
$F_{\theta,y}^{\boldsymbol\alpha}(dx|\Lambda )$ (respectively).
$\big\langle R_{\theta,y}^{\boldsymbol\alpha}(\lambda ) \big\rangle$,
$\big\langle S_{\theta,y}^{\boldsymbol\alpha}(\Lambda ) \big\rangle $ and
$\big\langle F_{\theta,y}^{\boldsymbol\alpha}(\Lambda ) \big\rangle$ are the quantities defined by
\begin{align}\label{1.907}
	\left\langle R_{\theta,y}^{\boldsymbol\alpha}(\lambda ) \right\rangle
	=
	R_{\theta,y}^{\boldsymbol\alpha}({\cal X}|\lambda ),
	\;\;\;\;\;
	\left\langle S_{\theta,y}^{\boldsymbol\alpha}(\Lambda ) \right\rangle
	=
	S_{\theta,y}^{\boldsymbol\alpha}({\cal X}|\Lambda ),
	\;\;\;\;\;
	\left\langle F_{\theta,y}^{\boldsymbol\alpha}(\Lambda ) \right\rangle
	=
	F_{\theta,y}^{\boldsymbol\alpha}({\cal X}|\Lambda ).
\end{align}
$F_{\theta,y}(\Lambda)$ is the element of ${\cal L}_{0}({\cal X} )$ defined by
\begin{align}\label{1.909}
	F_{\theta,y}(\Lambda)
	=
	\left\{F_{\theta,y}^{\boldsymbol\alpha}(\Lambda):
	\boldsymbol\alpha\in\mathbb{N}_{0}^{d}, |\boldsymbol\alpha|\leq p \right\},
\end{align}
where $F_{\theta,y}^{\boldsymbol\alpha}(\Lambda)$ is
the component $\boldsymbol\alpha$ of $F_{\theta,y}(\Lambda)$.
$F_{\theta,\boldsymbol y}^{m:n}(\Lambda)$
is the element of ${\cal L}_{0}({\cal X} )$ recursively defined by
\begin{align}\label{1.705}
	F_{\theta,\boldsymbol y}^{m:m}(\Lambda)
	=
	\Lambda,
	\;\;\;\;\;
	F_{\theta,\boldsymbol y}^{m:n+1}(\Lambda)
	=
	F_{\theta,y_{n+1} }\left(F_{\theta,\boldsymbol y}^{m:n}(\Lambda) \right)
\end{align}
for $n\geq m\geq 0$
and a sequence $\boldsymbol y = \{y_{n} \}_{n\geq 1}$ in ${\cal Y}$.
$f_{\theta,\boldsymbol y}^{\boldsymbol\alpha,m:n}(x|\Lambda)$
is the function defined for $n>m\geq 0$ by
\begin{align}\label{1.921}
	f_{\theta,\boldsymbol y}^{\boldsymbol\alpha,m:n}(x|\Lambda )
	=
	f_{\theta,y_{n} }^{\boldsymbol\alpha}(x|F_{\theta,\boldsymbol y}^{m:n-1}(\Lambda) ).
\end{align}
${\cal E}_{\lambda} = \left\{ {\cal E}_{\lambda}^{\boldsymbol\alpha}: \boldsymbol\alpha\in\mathbb{N}_{0}^{d},
|\boldsymbol\alpha|\leq p \right\}$ is the element of ${\cal L}_{0}({\cal X} )$ defined by
\begin{align}\label{1.911}
	{\cal E}_{\lambda}^{\boldsymbol 0}(B) = \lambda(B),
	\;\;\;\;\;
	{\cal E}_{\lambda}^{\boldsymbol\alpha}(B) = 0
\end{align}
for $B\in{\cal B}({\cal X} )$, $\lambda\in{\cal P}({\cal X} )$,
$\boldsymbol\alpha\in\mathbb{N}_{0}^{d}$, $1\leq|\boldsymbol\alpha|\leq p$,
where ${\cal E}_{\lambda}^{\boldsymbol 0}$ and
${\cal E}_{\lambda}^{\boldsymbol\alpha}$ are (respectively) the component $\boldsymbol 0$ and
the component $\boldsymbol\alpha$ of ${\cal E}_{\lambda}$.

We will show in Theorem \ref{theorem1.1} that
$F_{\theta,\boldsymbol y}^{m:n}(\Lambda)$ is the vector of the optimal filter derivatives
of the order up to $p$.
More specifically, we will demonstrate 
\begin{align*}
	F_{\theta,\boldsymbol y}^{\boldsymbol \alpha, m:n}(B|{\cal E}_{\lambda} )
	=
	\partial_{\theta}^{\boldsymbol \alpha} \:
	P\left(\left. X_{n}^{\theta,\lambda}\in B
	\right| Y_{1:n}^{\theta,\lambda} = y_{1:n}
	\right)
\end{align*}
for each $\theta\in\Theta$, $B\in{\cal B}({\cal X} )$,
$\lambda\in{\cal P}({\cal X} )$, $\boldsymbol\alpha\in\mathbb{N}_{0}^{d}$,
$|\boldsymbol\alpha|\leq p$,
$n\geq 1$ and any sequence $\boldsymbol y = \{y_{n} \}_{n\geq 1}$
in ${\cal Y}$.

\subsection{Existence and Stability Results}

We analyze here the existence and stability of the optimal filter higher-order derivatives.
The analysis is carried out under the following assumptions.

\begin{assumption}\label{a1.1}
There exists a real number $\varepsilon\in(0,1)$
and
for each $\theta\in\Theta$, $y\in{\cal Y}$, there exists a measure
$\mu_{\theta}(dx|y)$ on ${\cal X}$ such that
$0<\mu_{\theta}({\cal X}|y)<\infty$
and
\begin{align*}
	\varepsilon\mu_{\theta}(B|y)
	\leq
	\int_{B} r_{\theta}(y,x'|x) \mu(dx')
	\leq
	\frac{\mu_{\theta}(B|y)}{\varepsilon}
\end{align*}
for all $x\in{\cal X}$, $B\in{\cal B}({\cal X} )$.
\end{assumption}

\begin{assumption}\label{a1.2}
There exists a function $\psi: {\cal Y}\rightarrow [1,\infty )$
such that
\begin{align}\label{a1.2.1*}
	\left|\partial_{\theta}^{\boldsymbol\alpha} r_{\theta}(y,x'|x) \right|
	\leq
	\big(\psi(y) \big)^{|\boldsymbol\alpha|}
	r_{\theta}(y,x'|x)
\end{align}
for all $\theta\in\Theta$, $x,x'\in{\cal X}$, $y\in{\cal Y}$
and any multi-index $\boldsymbol\alpha\in\mathbb{N}_{0}^{d}\setminus\{\boldsymbol 0\}$,
$|\boldsymbol\alpha|\leq p$.
\end{assumption}

\begin{assumption}\label{a1.2'}
There exists a function $\phi:{\cal Y}\times{\cal X}\rightarrow[1,\infty)$
such that
\begin{align*}
	r_{\theta}(y,x'|x) \leq\phi(y,x'),
	\;\;\;\;\;
	\int \phi(y,x'')\mu(dx'')<\infty
\end{align*}
for all $\theta\in\Theta$, $x,x'\in{\cal X}$, $y\in{\cal Y}$.
\end{assumption}

Assumption \ref{a1.1} is a standard strong mixing condition
for state-space models.
It ensures that the optimal filter
$P_{\theta,\boldsymbol y}^{m:n}(\lambda)$
forgets its initial condition $\lambda$ exponentially fast
(see Proposition \ref{proposition1.2}).
In this or a similar form, Assumption \ref{a1.1} is a crucial ingredient
in many results on optimal filtering and statistical inference in
state-space and hidden Markov models
(see e.g., \cite{atar&zeitouni}, \cite{bickel&ritov&ryden},
\cite{delmoral&guionnet}, \cite{delmoral&doucet&singh}, \cite{douc&moulines&ryden},
\cite{legland&mevel2} -- \cite{legland&oudjane},
\cite{ryden}, \cite{tadic2};
see also \cite{cappe&moulines&ryden}, \cite{crisan&rozovskii}, \cite{douc&moulines&stoffer}
and references cited therein).
Assumption \ref{a1.2} can be considered as an extension of \cite[Assumption B]{legland&mevel}
and \cite[Assumption 3.2]{tadic&doucet1}
to the optimal filter higher-order derivatives.
It ensures that the higher-order score functions
\begin{align*}
	\frac{\partial_{\theta}^{\boldsymbol\alpha} r_{\theta}(y,x'|x) }{r_{\theta}(y,x'|x) }
\end{align*}
are well-defined and uniformly bounded in $\theta$, $x,x'$.
Together with Assumptions \ref{a1.1} and \ref{a1.2'},
Assumption \ref{a1.2} guarantees that the higher-order derivatives of
the optimal filter $P_{\theta,\boldsymbol y}^{m:n}(\lambda)$
exist and forget their initial conditions exponentially fast
(see Theorems \ref{theorem1.1} and \ref{theorem1.2}).
Assumptions \ref{a1.1} -- \ref{a1.2'}
hold if ${\cal X}$ is a compact set
and $q_{\theta}(y|x)$ is a mixture of Gaussian densities
(see the example studied in Section \ref{section3}).

Our results on the existence and stability of the optimal filter higher-order derivatives
are presented in the next two theorems.

\begin{theorem}[Higher-Order Differentiability]\label{theorem1.1}
Let Assumptions \ref{a1.1} -- \ref{a1.2'} hold.
Then, $p_{\theta, \boldsymbol y}^{m:n}(x|\lambda )$ and $P_{\theta, \boldsymbol y}^{m:n}(B|\lambda )$
are $p$-times differentiable in $\theta$
for each $\theta\in\Theta$, $x\in{\cal X}$, $B\in{\cal B}({\cal X} )$,
$\lambda\in{\cal P}({\cal X} )$, $n>m\geq 0$
and any sequence $\boldsymbol y = \{y_{n} \}_{n\geq 1}$ in ${\cal Y}$.
Moreover, we have
\begin{align}\label{t1.1.1*}
	\partial_{\theta}^{\boldsymbol\alpha} p_{\theta, \boldsymbol y}^{m:n}(x|\lambda )
	=
	f_{\theta, \boldsymbol y}^{\boldsymbol\alpha, m:n}(x|{\cal E}_{\lambda} ),
	\;\;\;\;\;
	\partial_{\theta}^{\boldsymbol\alpha} P_{\theta,\boldsymbol y}^{m:n}(B|\lambda)
	=
	F_{\theta,\boldsymbol y}^{\boldsymbol\alpha, m:n}(B|{\cal E}_{\lambda} )
\end{align}
for any multi-index $\boldsymbol\alpha\in\mathbb{N}_{0}^{d}$, $|\boldsymbol\alpha|\leq p$.
\end{theorem}

\begin{theorem}[Forgetting]\label{theorem1.2}
Let Assumptions \ref{a1.1} and \ref{a1.2} hold.
Then, there exist real numbers $\tau\in(0,1)$, $K\in[1,\infty)$
(depending only on $p$, $\varepsilon$) such that
\begin{align}
	&\label{t1.2.1*}
	\|F_{\theta,\boldsymbol y}^{m:n}(\Lambda)\|
	\leq
	K \|\Lambda \|^{p}
	\left(\sum_{k=m+1}^{n} \psi(y_{k} ) \right)^{p},
	\\
	&\label{t1.2.3*}
	\|F_{\theta,\boldsymbol y}^{m:n}(\Lambda )
	-
	F_{\theta,\boldsymbol y}^{m:n}(\Lambda' )\|
	\leq
	K \tau^{n-m}
	\|\Lambda-\Lambda' \| (\|\Lambda \| + \|\Lambda' \| )^{p}
	\left(\sum_{k=m+1}^{n} \psi(y_{k} ) \right)^{p}
\end{align}
for all $\theta\in\Theta$,
$\Lambda, \Lambda'\in{\cal L}_{0}({\cal X} )$, $n\geq m\geq 0$
and any sequence $\boldsymbol y = \{y_{n} \}_{n\geq 1}$ in ${\cal Y}$.
\end{theorem}

Theorems \ref{theorem1.1} and \ref{theorem1.2} are proved in Sections \ref{section2*}
and \ref{section1.1*}, respectively.
According to Theorem \ref{theorem1.1},
the filtering density
$p_{\theta,\boldsymbol y}^{m:n}(x|\lambda )$ and the filtering distribution
$P_{\theta,\boldsymbol y}^{m:n}(dx|\lambda)$
are $p$-times differentiable in $\theta$.
The same theorem also shows how their higher-order derivatives
can be computed recursively using
mappings $f_{\theta,y}^{\boldsymbol\alpha}(x|\Lambda)$,
$F_{\theta,y}^{\boldsymbol\alpha}(\Lambda )$.
According to Theorem \ref{theorem1.2},
the filtering distribution
and its higher-order derivatives $F_{\theta,\boldsymbol y}^{m:n}(\Lambda)$
forget their initial conditions exponentially fast.

In the rest of the section, we study the ergodicity properties of the optimal filter
higher-order derivatives.
To do so, we use the following notation.
${\cal Z}$ is the set defined by
${\cal Z} = {\cal X}\times{\cal Y}\times{\cal L}_{0}({\cal X} )$.
$\Phi_{\theta}(x,y,\Lambda)$ is a function which maps
$\theta\in\Theta$, $x\in{\cal X}$, $y\in{\cal Y}$, $\Lambda\in{\cal L}_{0}({\cal X} )$
to $\mathbb{R}$.
$\Phi_{\theta}(z)$ is another notation for $\Phi_{\theta}(x,y,\Lambda)$,
i.e., $\Phi_{\theta}(z)=\Phi_{\theta}(x,y,\Lambda)$ for
$z=(x,y,\Lambda)$.
$\big\{ Z_{n}^{\theta,\Lambda} \big\}_{n\geq 0}$ and
$\big\{ \tilde{Z}_{n}^{\theta,\Lambda} \big\}_{n\geq 0}$ are stochastic processes defined by
\begin{align*}
	Z_{n}^{\theta,\Lambda }
	=
	\left(X_{n}, Y_{n}, F_{\theta, \boldsymbol Y}^{0:n}(\Lambda) \right),
	\;\;\;\;\;
	\tilde{Z}_{n}^{\theta,\Lambda }
	=
	\left(X_{n+1}, Y_{n+1}, F_{\theta, \boldsymbol Y}^{0:n}(\Lambda) \right).
\end{align*}
for $n\geq 0$,
where $\boldsymbol Y= \{Y_{n} \}_{n\geq 1}$.
$\Pi_{\theta}(z,dz')$ and $\tilde{\Pi}_{\theta}(z,dz')$ are the kernels on ${\cal Z}$
defined by
\begin{align*}
	\Pi_{\theta}(z,B)
	=&
	\int\int I_{B}(x',y',F_{\theta,y'}(\Lambda) )
	Q(x',dy') P(x,dx'),
	\\
	\tilde{\Pi}_{\theta}(z,B)
	=&
	\int\int I_{B}(x',y',F_{\theta,y}(\Lambda) )
	Q(x',dy') P(x,dx')
\end{align*}
for $B\in{\cal B}({\cal Z} )$ and $z=(x,y,\Lambda)$.
Then, it is easy to show that
$\{ Z_{n}^{\theta,\Lambda } \}_{n\geq 0}$ and
$\{ \tilde{Z}_{n}^{\theta,\Lambda } \}_{n\geq 0}$
are homogeneous Markov processes
whose transition kernels are
$\Pi_{\theta}(z,dz')$ and $\tilde{\Pi}_{\theta}(z,dz')$, 
respectively.

To analyze the ergodicity properties of
$\{ Z_{n}^{\theta,\Lambda } \}_{n\geq 0}$ and
$\{ \tilde{Z}_{n}^{\theta,\Lambda } \}_{n\geq 0}$,
we rely on following assumptions.

\begin{assumption} \label{a1.3}
There exist a probability measure $\pi(dx)$ on ${\cal X}$
and real numbers $\delta\in(0,1)$, $K_{0}\in[1,\infty )$
such that
\begin{align*}
	|P^{n}(x,B) - \pi(B) |
	\leq
	K_{0}\delta^{n}
\end{align*}
for all $x\in{\cal X}$, $B\in{\cal B}({\cal X} )$, $n\geq 0$.
\end{assumption}

\begin{assumption}\label{a1.4}
There exit a function $\varphi:{\cal X} \times {\cal Y} \rightarrow [1,\infty )$
and a real number $q\in[0,\infty )$ such that
\begin{align*}
	&
	|\Phi_{\theta}(x,y,\Lambda ) |
	\leq
	\varphi(x,y) \|\Lambda \|^{q},
	\\
	&
	|\Phi_{\theta}(x,y,\Lambda) - \Phi_{\theta}(x,y,\Lambda') |
	\leq
	\varphi(x,y) \|\Lambda - \Lambda' \| (\|\Lambda \| + \|\Lambda' \| )^{q}
\end{align*}
for all $\theta\in\Theta$,
$x\in{\cal X}$, $y\in{\cal Y}$,
$\Lambda,\Lambda'\in{\cal L}_{0}({\cal X} )$.
\end{assumption}

\begin{assumption}\label{a1.5}
There exists a real number $L_{0}\in[1,\infty )$ such that
\begin{align}\label{a1.4.1*}
	\int \varphi(x,y) \psi^{r}(y) Q(x,dy) \leq L_{0}
\end{align}
for all $x\in{\cal X}$, where $r=p(p+q+1)$.
\end{assumption}

Assumption \ref{a1.3} ensures that the Markov process $\{ (X_{n}, Y_{n} ) \}_{n\geq 0}$
is geometrically ergodic (for further details, see e.g., \cite{meyn&tweedie}).
Assumption \ref{a1.4} is related to function $\Phi_{\theta}(x,y,\Lambda)$
and its analytical properties.
It requires $\Phi_{\theta}(x,y,\Lambda)$ to be locally Lipschitz continuous in $\Lambda$
and to grow at most polynomially in the same argument.
Assumption \ref{a1.5} corresponds to the conditional mean of
$\varphi(X_{n}, Y_{n} ) \psi^{r}(Y_{n} )$ given $X_{n}=x$.\footnote
{Assumption \ref{a1.5} holds under the following conditions:
(i) ${\cal X}$ is compact, (ii) $\varphi(x,y)$ is continuous in $(x,y)$
and polynomial in $y$,
(iii) $\psi(y)$ is polynomial and
(iv) $q_{\theta}(y|x)$ is Gaussian in $y$ and
continuous in $(\theta,x,y)$.
}
In this or a similar form,
Assumptions \ref{a1.3} -- \ref{a1.5}
are involved in many results on the stability of the optimal filter
and the asymptotic properties of maximum likelihood estimation
in state-space and hidden Markov models
(see e.g. \cite{atar&zeitouni}, \cite{delmoral&guionnet},
\cite{legland&mevel}, \cite{legland&oudjane}, \cite{tadic2}, \cite{tadic&doucet2};
see also \cite{cappe&moulines&ryden}, \cite{crisan&rozovskii}
and references cited therein).

Our results on the ergodicity of $\{ Z_{n}^{\theta,\Lambda } \}_{n\geq 0}$ and
$\{ \tilde{Z}_{n}^{\theta,\Lambda } \}_{n\geq 0}$
are presented in the next theorem.

\begin{theorem}[Ergodicity]\label{theorem1.3}
Let Assumptions \ref{a1.1} -- \ref{a1.5} hold.
Moreover, let $s=p(q+1)$.
Then, there exist functions $\phi_{\theta}$, $\tilde{\phi}_{\theta}$
mapping $\theta\in\Theta$ to $\mathbb{R}$
such that
\begin{align*}
	\phi_{\theta}
	=
	\lim_{n\rightarrow\infty }
	(\Pi^{n}\Phi)_{\theta}(z),
	\;\;\;\;\;
	\tilde{\phi}_{\theta}
	=
	\lim_{n\rightarrow\infty }
	(\tilde{\Pi}^{n}\Phi)_{\theta}(z)
\end{align*}
for all $\theta\in\Theta$, $z\in{\cal Z}$.
There also exist real numbers $\rho\in(0,1)$, $L\in[1,\infty )$
(depending only on $\varepsilon$, $\delta$, $p$, $q$, $K_{0}$, $L_{0}$)
such that
\begin{align*}
	&
	|(\Pi^{n}\Phi)_{\theta}(z) - \phi_{\theta} |
	\leq
	L\rho^{n} \|\Lambda \|^{s},
	\;\;\;\;\;
	|(\tilde{\Pi}^{n}\Phi)_{\theta}(z) - \tilde{\phi}_{\theta} |
	\leq
	L\rho^{n} \psi^{r}(y) \|\Lambda \|^{s}
\end{align*}
for all $\theta\in\Theta$, $x\in{\cal X}$, $y\in{\cal Y}$,
$\Lambda\in{\cal L}_{0}({\cal X})$, $n\geq 1$
and $z=(x,y,\Lambda)$.
Here $(\Pi^{n}\Phi)_{\theta}(z)$ and $(\tilde{\Pi}^{n}\Phi)_{\theta}(z)$ are the functions defined by
\begin{align*}
	(\Pi^{n}\Phi)_{\theta}(z) = \int \Phi_{\theta}(z') \Pi_{\theta}^{n}(z,dz'),
	\;\;\;\;\;
	(\tilde{\Pi}^{n}\Phi)_{\theta}(z) = \int \Phi_{\theta}(z') \tilde{\Pi}_{\theta}^{n}(z,dz').
\end{align*}
\end{theorem}

Theorem \ref{theorem1.3} is proved in Section \ref{section1.2*}.
According to this theorem,
Markov processes $\big\{ Z_{n}^{\theta,\Lambda} \big\}_{n\geq 0}$ and
$\big\{ \tilde{Z}_{n}^{\theta,\Lambda} \big\}_{n\geq 0}$
are geometrically ergodic.
As $F_{\theta,Y}^{0:n}(\Lambda)$ is a component of $Z_{n}^{\theta,\Lambda}$ and
$\tilde{Z}_{n}^{\theta,\Lambda}$,
the optimal filter and its higher-order derivatives
are geometrically ergodic, too.

The optimal filter and its properties have extensively been studied
in the literature.
However, to the best of our knowledge,
the existing results do not provide any information about
the existence and stability of
the optimal filter higher-order derivatives.
Theorems \ref{theorem1.1} -- \ref{theorem1.3} fill this gap
in the literature on optimal filtering.
More specifically, these theorems extend the existing results on
the optimal filter first-order derivatives
(in particular those of \cite{douc&matias}, \cite{legland&mevel} and \cite{tadic&doucet1})
to the higher-order derivatives.
In Section \ref{section2},
we use Theorems \ref{theorem1.1} -- \ref{theorem1.3} to study
the analytical properties of the log-likelihood rate for state-space models.
Moreover, in \cite{tadic&doucet2}, we use the same theorems
to analyze the asymptotic behavior of recursive maximum likelihood estimation in
state-space models.

\section{Analytical Properties of Log-Likelihood Rate} \label{section2}

In this section, the results presented in Section \ref{section1} are used to study
the higher-order differentiability of the log-likelihood rate for state-space models.
In addition to the notation specified in Section \ref{section1},
the following notation is used here, too.
Let $q_{\theta}^{n}(y_{1:n}|\lambda)$ be the function defined by
\begin{align}\label{1.915}
	q_{\theta}^{n}(y_{1:n}|\lambda )
	=
	\int\cdots\int\int
	\left(\prod_{k=1}^{n} r_{\theta}(y_{k},x_{k}|x_{k-1} ) \right)
	\mu(dx_{n} ) \cdots \mu(dx_{1} ) \lambda(dx_{0} )
\end{align}
for $\theta\in\Theta$, $y_{1},\dots,y_{n}\in{\cal Y}$, $\lambda\in{\cal P}({\cal X} )$, $n\geq 1$.
Then, the average log-likelihood for state-space model
$\{ (X_{n}, Y_{n} ) \}_{n\geq 0}$ is defined as
\begin{align*}
	l_{n}(\theta,\lambda )
	=
	E\left(\frac{1}{n} \log q_{\theta}^{n}(Y_{1:n}|\lambda ) \right),
\end{align*}
while the corresponding likelihood rate is the limit
$\lim_{n\rightarrow\infty}l_{n}(\theta,\lambda)$.
To analyze the analytical and asymptotic properties of $l_{n}(\theta,\lambda )$,
we rely on the following assumptions.

\begin{assumption} \label{a2.1}
There exists a function $\varphi:{\cal Y}\rightarrow [1,\infty )$ such that
\begin{align*}
	\left|
	\log\mu_{\theta}({\cal X}|y)
	\right|
	\leq
	\varphi(y)
\end{align*}
for all $\theta\in\Theta$, $y\in{\cal Y}$, 
where $\mu_{\theta}(dx|y)$ is specified in Assumption \ref{a1.1}.
\end{assumption}

\begin{assumption}\label{a2.2}
There exists a real number $M_{0}\in[1,\infty )$ such that
\begin{align*}
	\int \varphi(y) \psi^{u}(y) Q(x,dy)
	\leq M_{0},
	\;\;\;\;\;
	\int \psi^{v}(y) Q(x,dy)
	\leq
	M_{0}
\end{align*}
for all $x\in{\cal X}$, where $u=p(p+1)$, $v=2p(p+1)$ and $\psi(y)$ is specified in Assumption \ref{a1.2}.
\end{assumption}

Assumptions \ref{a2.1} and \ref{a2.2} are related to the conditional measure $\mu_{\theta}(dx|y)$
and its properties.
In this or similar form, these assumptions are involved in
a number of result on the asymptotic properties of maximum likelihood estimation
in state-space and hidden Markov models
(see \cite{bickel&ritov&ryden}, \cite{douc&matias}, \cite{douc&moulines&ryden},
\cite{ryden}, \cite{tadic2};
see also \cite{cappe&moulines&ryden} and references cited therein).

Our results on the higher-order differentiability of log-likelihood rate for
state-space models are provided in the next theorem.

\begin{theorem}\label{theorem2.1}
Let Assumptions \ref{a1.1} -- \ref{a1.3}, \ref{a2.1} and \ref{a2.2} hold.
Then, there exists a function $l:\Theta\rightarrow\mathbb{R}$
which is $p$-times differentiable on $\Theta$ and satisfies
$l(\theta )	= \lim_{n\rightarrow\infty }l_{n}(\theta, \lambda )$
for all $\theta\in\Theta$, $\lambda\in{\cal P}({\cal X} )$.
\end{theorem}

Theorem \ref{theorem2.1} is proved in Section \ref{section2*}.
The theorem claims that the log-likelihood rate
$\lim_{n\rightarrow\infty }l_{n}(\theta, \lambda )$ is well-defined for
each $\theta\in\Theta$, $\lambda\in{\cal P}({\cal X} )$.
It also claims that this rate is independent of $\lambda$
and $p$-times differentiable in $\theta$.

In the context of statistical inference,
the properties of
log-likelihood rate for state-space and hidden Markov models
have been studied in a number of papers
(see \cite{bickel&ritov&ryden}, \cite{douc&matias}, \cite{douc&moulines&ryden},
\cite{ryden}, \cite{tadic2};
see also \cite{cappe&moulines&ryden} and references cited therein).
However, the existing results do not address
the higher-order differentiability of this rate.
Theorem \ref{theorem2.1} fills this gap in the literature.
Theorem \ref{theorem2.1} is also relevant for
asymptotic properties of maximum likelihood estimation in state-space models \cite{tadic&doucet2}.
The same theorem can also be used to study the higher-order statistical asymptotics for
the maximum likelihood estimation in time-series models
(for further details on such asymptotics,
see e.g. \cite{mccullagh}, \cite{taniguchi&kakizawa}).

\section{Example}\label{section3}

To illustrate the main results,
we use them to study optimal filtering in non-linear state-space models.
Let $\Theta$ and $d$ have the same meaning as in Section \ref{section1},
while $\tilde{\Theta}\subseteq\mathbb{R}^{d}$ is an open set satisfying
$\text{cl}\Theta\subset\tilde{\Theta}$.
We consider the following state-space model:
\begin{align}\label{3.1}
	X_{n+1}^{\theta,\lambda}
	=
	A_{\theta}(X_{n}^{\theta,\lambda} ) + B_{\theta}(X_{n}^{\theta,\lambda} ) U_{n},
	\;\;\;\;\;
	Y_{n}^{\theta,\lambda}
	=
	C_{\theta}(X_{n}^{\theta,\lambda} ) + D_{\theta}(X_{n}^{\theta,\lambda} ) V_{n},
	\;\;\;\;\;
	n\geq 0.
\end{align}
Here, $\theta\in\tilde{\Theta}$, $\lambda\in{\cal P}({\cal X} )$
are the parameters indexing the model (\ref{3.1}).
$A_{\theta}(x)$ and $B_{\theta}(x)$ are functions mapping
$\theta\in\tilde{\Theta}$, $x\in\mathbb{R}^{d_{x} }$ (respectively) to
$\mathbb{R}^{d_{x} }$ and $\mathbb{R}^{d_{x}\times d_{x} }$
($d_{x}$ has the same meaning as in Section \ref{section1}).
$C_{\theta}(x)$ and $D_{\theta}(x)$ are functions mapping
$\theta\in\tilde{\Theta}$, $x\in\mathbb{R}^{d_{x} }$ (respectively) to
$\mathbb{R}^{d_{y} }$ and $\mathbb{R}^{d_{y}\times d_{y} }$
($d_{y}$ has the same meaning as in Section \ref{section1}).
$X_{0}^{\theta,\lambda}$ is an $\mathbb{R}^{d_{x}}$-valued random variable
defined on a probability space $(\Omega,{\cal F}, P)$ and distributed according to $\lambda$.
$\{ U_{n} \}_{n\geq 0}$ are $\mathbb{R}^{d_{x}}$-valued i.i.d. random variables which are defined
on $(\Omega, {\cal F}, P)$ and have marginal density $r(u)$ with respect to Lebesgue measure.
$\{ V_{n} \}_{n\geq 0}$ are $\mathbb{R}^{d_{y}}$-valued i.i.d. random variables which are defined
on $(\Omega, {\cal F}, P)$ and have marginal density $s(v)$ with respect to Lebesgue measure.
We also assume that $X_{0}^{\theta,\lambda}$,
$\{U_{n}\}_{n\geq 0}$ and $\{V_{n}\}_{n\geq 0}$ are (jointly) independent.

In addition to the previously introduced notation, the following notation is used here, too.
$\tilde{p}_{\theta}(x'|x)$ and $\tilde{q}_{\theta}(y|x)$ are the functions defined by
\begin{align*}
	\tilde{p}_{\theta}(x'|x)
	=
	\frac{r\left(B_{\theta}^{-1}(x) (x'-A_{\theta}(x) ) \right)}
	{|\text{det} B_{\theta}(x) |},
	\;\;\;\;\;
	\tilde{q}_{\theta}(y|x)
	=
	\frac{s\left(D_{\theta}^{-1}(x) (y-C_{\theta}(x) ) \right)}
	{|\text{det} D_{\theta}(x) |}
\end{align*}
for $\theta\in\tilde{\Theta}$, $x,x'\in\mathbb{R}^{d_{x} }$, $y\in\mathbb{R}^{d_{y} }$
(provided $B_{\theta}(x)$ and $D_{\theta}(x)$ are invertible).
$p_{\theta}(x'|x)$ and $q_{\theta}(y|x)$ are the functions defined by
\begin{align}\label{3.5}
	p_{\theta}(x'|x)
	=
	\frac{r\left(B_{\theta}^{-1}(x) (x'-A_{\theta}(x) ) \right) 1_{\cal X}(x') }
	{\int_{\cal X} r\left(B_{\theta}^{-1}(x) (x''-A_{\theta}(x) ) \right) dx'' },
	\;\;\;\;\;
	q_{\theta}(y|x)
	=
	\frac{s\left(D_{\theta}^{-1}(x) (y-C_{\theta}(x) ) \right) 1_{\cal Y}(y) }
	{\int_{\cal Y} s\left(D_{\theta}^{-1}(x) (y'-C_{\theta}(x) ) \right) dy' }
\end{align}
(${\cal X}$, ${\cal Y}$ have the same meaning as in Section \ref{section1}).
It is easy to conclude that
$\tilde{p}_{\theta}(x'|x)$ and $\tilde{q}_{\theta}(y|x)$ are
the conditional densities of $X_{n+1}^{\theta,\lambda}$
and $Y_{n}^{\theta,\lambda}$ (respectively) given $X_{n}^{\theta,\lambda}=x$.
It is also easy to deduce that $p_{\theta}(x'|x)$ and $q_{\theta}(y|x)$
accurately approximate $\tilde{p}_{\theta}(x'|x)$ and $\tilde{q}_{\theta}(y|x)$
when ${\cal X}$ and ${\cal Y}$ are sufficiently large
(i.e., when balls of a sufficiently large radius can be inscribed in
${\cal X}$, ${\cal Y}$).
$p_{\theta}(x'|x)$ and $q_{\theta}(y|x)$ can be interpreted as
truncations of $\tilde{p}_{\theta}(x'|x)$ and $\tilde{q}_{\theta}(y|x)$
to sets ${\cal X}$ and ${\cal Y}$
(i.e., model specified in (\ref{3.5}) can be considered as
a truncation of model (\ref{3.1}) to ${\cal X}$, ${\cal Y}$).
This or similar truncation is involved (implicitly or explicitly)
in the implementation of any numerical approximation to the optimal filter for the model (\ref{3.1}).

The optimal filter based on the truncated model (\ref{3.5})
is studied under the following assumptions.

\begin{assumption}\label{a3.1}
$r(x)>0$ and $s(y)>0$ for all $x\in\mathbb{R}^{d_{x} }$,
$y\in\mathbb{R}^{d_{y} }$.
Moreover, $B_{\theta}(x)$ and $D_{\theta}(x)$ are invertible
for each $\theta\in\tilde{\Theta}$, $x\in\mathbb{R}^{d_{x} }$.
\end{assumption}

\begin{assumption}\label{a3.2}
$r(x)$ and $s(y)$ are $p$-times differentiable for all $x\in\mathbb{R}^{d_{x} }$,
$y\in\mathbb{R}^{d_{y} }$, where $p\geq 1$.
Moreover, $A_{\theta}(x)$, $B_{\theta}(x)$, $C_{\theta}(x)$ and $D_{\theta}(x)$
are $p$-times differentiable in $\theta$
for each $\theta\in\tilde{\Theta}$, $x\in\mathbb{R}^{d_{x} }$.
\end{assumption}

\begin{assumption}\label{a3.3}
$\partial^{\boldsymbol\alpha} r(x)$ and $\partial^{\boldsymbol\alpha} s(y)$
are continuous for each $x\in\mathbb{R}^{d_{x} }$, $y\in\mathbb{R}^{d_{y} }$ and
any multi-index $\boldsymbol\alpha\in\mathbb{N}_{0}^{d}$, $|\boldsymbol\alpha|\leq p$.
Moreover,
$\partial_{\theta}^{\boldsymbol\alpha} A_{\theta}(x)$,
$\partial_{\theta}^{\boldsymbol\alpha} B_{\theta}(x)$,
$\partial_{\theta}^{\boldsymbol\alpha} C_{\theta}(x)$ and
$\partial_{\theta}^{\boldsymbol\alpha} D_{\theta}(x)$
are continuous in $(\theta,x)$ for each $\theta\in\tilde{\Theta}$,
$x\in\mathbb{R}^{d_{x} }$, $y\in\mathbb{R}^{d_{y} }$ and
any multi-index $\boldsymbol\alpha\in\mathbb{N}_{0}^{d}$, $|\boldsymbol\alpha|\leq p$.
\end{assumption}

\begin{assumption}\label{a3.4}
${\cal X}$ and ${\cal Y}$ are compact sets with non-empty interiors.
\end{assumption}

\begin{assumption}\label{a3.5}
${\cal X}$ is a compact set with a non-empty interior,
while ${\cal Y}=\mathbb{R}^{d_{y} }$.
Moreover, there exists a real number $K_{0}\in[1,\infty)$ such that
\begin{align*}
	s(y)\leq K_{0},
	\;\;\;\;\;
	\left|\partial^{\boldsymbol\alpha}s(y) \right|
	\leq
	K_{0}s(y)(1+\|y\| )^{|\boldsymbol\alpha|}
\end{align*}
for all $y\in\mathbb{R}^{d_{y} }$ and any multi-index
$\boldsymbol\alpha\in\mathbb{N}_{0}^{d}$, $|\boldsymbol\alpha|\leq p$.
\end{assumption}

\begin{assumption}\label{a3.6}
There exists a real number $L_{0}\in[1,\infty)$ such that
\begin{align*}
	\left|\log s(y) \right|
	\leq
	L_{0}(1 + \|y\|)^{2}
\end{align*}
for all $y\in{\cal Y}$.
\end{assumption}

Assumptions \ref{a3.1} -- \ref{a3.6} cover several classes of
non-linear state-space models met in practice --- 
e.g. they hold for a class of stochastic volatility and dynamic probit models.
Moreover, these assumptions include non-linear state-space models in which the observation
noise $\{V_{n} \}_{n\geq 0}$ is a mixture of Gaussian distributions.
Other models satisfying assumptions \ref{a3.1} -- \ref{a3.6} can be found
in \cite{cappe&moulines&ryden},
\cite{douc&moulines&stoffer} (see also references cited therein).

Our results on the optimal filter for model (\ref{3.5}) and its higher-order derivatives
read as follows.

\begin{corollary}\label{corollary3.1}
(i) Let Assumptions \ref{a3.1} -- \ref{a3.4} hold.
Then, all conclusions of Theorems \ref{theorem1.1} and \ref{theorem1.2} are true.

(ii) Let Assumptions \ref{a1.3}, \ref{a1.4} and \ref{a3.1} -- \ref{a3.4} hold.
Moreover, assume
\begin{align}\label{c3.1.1*}
	\sup_{x\in{\cal X}} \int \varphi(x,y)Q(x,dy)<\infty,
\end{align}
where $\varphi(x,y)$ is specified in Assumption \ref{a1.4}.
Then, all conclusions of Theorem \ref{theorem1.3} are true.

(iii) Let Assumptions \ref{a1.3} and \ref{a3.1} -- \ref{a3.4} hold.
Then, all conclusions of Theorem \ref{theorem2.1} are true.
\end{corollary}

\begin{corollary}\label{corollary3.2}
(i) Let Assumptions \ref{a3.1} -- \ref{a3.3} and \ref{a3.5} hold.
Then, all conclusions of Theorems \ref{theorem1.1} and \ref{theorem1.2} are true.

(ii) Let Assumptions \ref{a1.3}, \ref{a1.4}, \ref{a3.1} -- \ref{a3.3}
and \ref{a3.5} hold.
Moreover, assume
\begin{align}\label{c3.2.1*}
	\sup_{x\in{\cal X}} \int \varphi(x,y)(1+\|y\|)^{2r}Q(x,dy)<\infty,
\end{align}
where $r$ and $\varphi(x,y)$ are specified in Assumptions \ref{a1.4} and \ref{a1.5}.
Then, all conclusions of Theorem \ref{theorem1.3} are true.

(iii) Let Assumptions \ref{a1.3}, \ref{a3.1} -- \ref{a3.3}, \ref{a3.5} and \ref{a3.6} hold.
Moreover, assume
\begin{align}\label{c3.2.3*}
	\sup_{x\in{\cal X}} \int (1+\|y\|)^{2v}Q(x,dy)<\infty,
\end{align}
where $v$ is specified in Assumption \ref{a2.2}.
Then, all conclusions of Theorem \ref{theorem2.1} are true.
\end{corollary}

Corollaries \ref{corollary3.1} and \ref{corollary3.2} are proved in Section \ref{section3*}.

\section{Proof of Theorem \ref{theorem1.2} } \label{section1.1*}

In this section, we use the following notation.
$\tau$ is the real number defined as
$\tau = (1-\varepsilon^{2} )^{1/2}$.
$G_{\theta,y}(\lambda,\tilde{\lambda} )$ is the element of ${\cal M}_{s}({\cal X} )$ defined by
\begin{align}\label{4.101}
	G_{\theta,y}(\lambda,\tilde{\lambda} )
	=
	\frac{R_{\theta,y}^{\boldsymbol 0}(\tilde{\lambda} ) }
	{\big\langle R_{\theta,y}^{\boldsymbol 0}(\lambda) \big\rangle}
	-
	\frac{R_{\theta,y}^{\boldsymbol 0}(\lambda)
	\big\langle R_{\theta,y}^{\boldsymbol 0}(\tilde{\lambda} ) \big\rangle}
	{\big\langle R_{\theta,y}^{\boldsymbol 0}(\lambda) \big\rangle^{2} }
\end{align}
for $\theta\in\Theta$, $y\in{\cal Y}$, $\lambda\in{\cal P}({\cal X} )$, $\tilde{\lambda}\in{\cal M}_{s}({\cal X} )$.
$T_{\theta,y}^{\boldsymbol\alpha,\boldsymbol\beta}(\Lambda )$ is the element of ${\cal M}_{s}({\cal X} )$ defined by
\begin{align}\label{4.103}
	&
	T_{\theta,y}^{\boldsymbol\alpha,\boldsymbol\beta}(\Lambda )
	=
	\frac{R_{\theta,y}^{\boldsymbol\alpha-\boldsymbol\beta}(\lambda_{\boldsymbol\beta} ) }
	{\big\langle R_{\theta,y}^{\boldsymbol 0}(\lambda_{\boldsymbol 0} ) \big\rangle }
	-
	F_{\theta,y}^{\boldsymbol 0}(\Lambda)
	\frac{\big\langle R_{\theta,y}^{\boldsymbol\alpha-\boldsymbol\beta}(\lambda_{\boldsymbol\beta} ) \big\rangle }
	{\big\langle R_{\theta,y}^{\boldsymbol 0}(\lambda_{\boldsymbol 0} ) \big\rangle }
\end{align}
for $\Lambda=\left\{\lambda_{\boldsymbol\gamma}:
\boldsymbol\gamma\in\mathbb{N}_{0}^{d}, |\boldsymbol\gamma|\leq p \right\} \in {\cal L}_{0}({\cal X} )$,
$\boldsymbol\alpha, \boldsymbol\beta \in \mathbb{N}_{0}^{d}$,
$\boldsymbol\beta\leq\boldsymbol\alpha$, $|\boldsymbol\alpha|\leq p$.
$G_{\theta,y}^{\boldsymbol\alpha}(\Lambda)$ and
$H_{\theta,y}^{\boldsymbol\alpha}(\Lambda)$ are the elements of ${\cal M}_{s}({\cal X} )$ defined by
\begin{align}\label{4.105}
	G_{\theta,y}^{\boldsymbol\alpha}(\Lambda)
	=
	G_{\theta,y}(\lambda_{\boldsymbol 0}, \lambda_{\boldsymbol\alpha} ),
	\;\;\;
	H_{\theta,y}^{\boldsymbol\alpha}(\Lambda)
	=
	\sum_{\stackrel{\scriptstyle
	\boldsymbol\beta\in\mathbb{N}_{0}^{d}\setminus\{\boldsymbol\alpha\} }
	{\boldsymbol\beta\leq\boldsymbol\alpha} }
	\left(\boldsymbol\alpha \atop \boldsymbol\beta \right)
	T_{\theta,y}^{\boldsymbol\alpha,\boldsymbol\beta}(\Lambda )
	-
	\sum_{\stackrel{\scriptstyle \boldsymbol\beta\in\mathbb{N}_{0}^{d}
	\setminus\{\boldsymbol 0,\boldsymbol\alpha\} }
	{\boldsymbol\beta\leq\boldsymbol\alpha} }
	\left(\boldsymbol\alpha \atop \boldsymbol\beta \right)
	F_{\theta,y}^{\boldsymbol\beta}(\Lambda )
	\big\langle S_{\theta,y}^{\boldsymbol\alpha-\boldsymbol\beta}(\Lambda ) \big\rangle.
\end{align}
Here and throughout the paper, we rely on the convention that
$\sum_{\boldsymbol\beta\in{\cal B} }$ is zero whenever ${\cal B}=\emptyset$.
Then, using (\ref{1.703}) -- (\ref{1.907}), it is straightforward to verify
\begin{align}\label{4.151}
	&
	S_{\theta,y}^{\boldsymbol\alpha}(\Lambda)
	=
	\sum_{\stackrel{\scriptstyle\boldsymbol\beta\in\mathbb{N}_{0}^{d} }
	{\boldsymbol\beta\leq\boldsymbol\alpha} }
	\left( \boldsymbol\alpha \atop \boldsymbol \beta \right)
	\frac{R_{\theta,y}^{\boldsymbol\alpha-\boldsymbol\beta}(\lambda_{\boldsymbol\beta} ) }
	{ \big\langle R_{\theta,y}^{\boldsymbol 0}(\lambda_{\boldsymbol 0} ) \big\rangle },
	\;\;\;\;\;
	F_{\theta,y}^{\boldsymbol\alpha}(\Lambda)
	=
	S_{\theta,y}^{\boldsymbol\alpha}(\Lambda)
	-
	\sum_{\stackrel{\scriptstyle\boldsymbol\beta\in\mathbb{N}_{0}^{d}\setminus\{\boldsymbol\alpha\} }
	{\boldsymbol\beta\leq\boldsymbol\alpha} }
	F_{\theta,y}^{\boldsymbol\beta}(\Lambda)
	\big\langle S_{\theta,y}^{\boldsymbol\alpha-\boldsymbol\beta}(\Lambda) \big\rangle.
\end{align}
Hence, we get
$
	F_{\theta,y}^{\boldsymbol 0}(\Lambda)
	=
	S_{\theta,y}^{\boldsymbol 0}(\Lambda)
	=
	R_{\theta,y}^{\boldsymbol 0}(\lambda_{\boldsymbol 0} ) /\big\langle R_{\theta,y}^{\boldsymbol 0}(\lambda_{\boldsymbol 0} ) \big\rangle
$ and
\begin{align*}
	T_{\theta,y}^{\boldsymbol\alpha,\boldsymbol\alpha}(\Lambda)
	=
	\frac{R_{\theta,y}^{\boldsymbol 0}(\lambda_{\boldsymbol\alpha} ) }
	{ \big\langle R_{\theta,y}^{\boldsymbol 0}(\lambda_{\boldsymbol 0} ) \big\rangle }
	-
	F_{\theta,y}^{\boldsymbol 0}(\Lambda)
	\frac{ \big\langle R_{\theta,y}^{\boldsymbol 0}(\lambda_{\boldsymbol\alpha} ) \big\rangle }
	{ \big\langle R_{\theta,y}^{\boldsymbol 0}(\lambda_{\boldsymbol 0} ) \big\rangle }
	=
	\frac{R_{\theta,y}^{\boldsymbol 0}(\lambda_{\boldsymbol\alpha} ) }
	{ \big\langle R_{\theta,y}^{\boldsymbol 0}(\lambda_{\boldsymbol 0} ) \big\rangle }
	-
	\frac{R_{\theta,y}^{\boldsymbol 0}(\lambda_{\boldsymbol 0} )
	\big\langle R_{\theta,y}^{\boldsymbol 0}(\lambda_{\boldsymbol\alpha} ) \big\rangle }
	{ \big\langle R_{\theta,y}^{\boldsymbol 0}(\lambda_{\boldsymbol 0} ) \big\rangle^{2} }
	=
	G_{\theta,y}^{\boldsymbol\alpha}(\Lambda).
\end{align*}
Consequently, (\ref{4.103}) -- (\ref{4.105}) imply
\begin{align}\label{4.153}
	S_{\theta,y}^{\boldsymbol\alpha}(\Lambda)
	-
	F_{\theta,y}^{\boldsymbol 0}(\Lambda)
	\big\langle S_{\theta,y}^{\boldsymbol\alpha}(\Lambda) \big\rangle
	=&
	\sum_{\stackrel{\scriptstyle\boldsymbol\beta\in\mathbb{N}_{0}^{d} }
	{\boldsymbol\beta\leq\boldsymbol\alpha} }
	\left( \boldsymbol\alpha \atop \boldsymbol \beta \right)
	\left(
	\frac{R_{\theta,y}^{\boldsymbol\alpha-\boldsymbol\beta}(\lambda_{\boldsymbol\beta} ) }
	{ \big\langle R_{\theta,y}^{\boldsymbol 0}(\lambda_{\boldsymbol 0} ) \big\rangle }
	-
	F_{\theta,y}^{\boldsymbol 0}(\Lambda)
	\frac{ \big\langle R_{\theta,y}^{\boldsymbol\alpha-\boldsymbol\beta}(\lambda_{\boldsymbol\beta} ) \big\rangle }
	{ \big\langle R_{\theta,y}^{\boldsymbol 0}(\lambda_{\boldsymbol 0} ) \big\rangle }
	\right)
	\nonumber\\
	=&
	\sum_{\stackrel{\scriptstyle\boldsymbol\beta\in\mathbb{N}_{0}^{d} }
	{\boldsymbol\beta\leq\boldsymbol\alpha} }
	\left( \boldsymbol\alpha \atop \boldsymbol \beta \right)
	T_{\theta,y}^{\boldsymbol\alpha,\boldsymbol\beta}(\Lambda)
	=
	G_{\theta,y}^{\boldsymbol\alpha}(\Lambda)
	+
	\sum_{\stackrel{\scriptstyle\boldsymbol\beta\in\mathbb{N}_{0}^{d}\setminus\{\boldsymbol\alpha\} }
	{\boldsymbol\beta\leq\boldsymbol\alpha} }
	\left( \boldsymbol\alpha \atop \boldsymbol \beta \right)
	T_{\theta,y}^{\boldsymbol\alpha,\boldsymbol\beta}(\Lambda).
\end{align}
Then, (\ref{4.105}), (\ref{4.151}) 
yield
\begin{align}\label{4.1}
	F_{\theta,y}^{\boldsymbol\alpha}(\Lambda )
	=
	S_{\theta,y}^{\boldsymbol\alpha}(\Lambda)
	-
	F_{\theta,y}^{\boldsymbol 0}(\Lambda)
	\big\langle S_{\theta,y}^{\boldsymbol\alpha}(\Lambda) \big\rangle
	-
	\sum_{\stackrel{\scriptstyle\boldsymbol\beta\in\mathbb{N}_{0}^{d}\setminus\{\boldsymbol 0,\boldsymbol\alpha\} }
	{\boldsymbol\beta\leq\boldsymbol\alpha} }
	F_{\theta,y}^{\boldsymbol\beta}(\Lambda)
	\big\langle S_{\theta,y}^{\boldsymbol\alpha-\boldsymbol\beta}(\Lambda) \big\rangle
	=
	G_{\theta,y}^{\boldsymbol\alpha}(\Lambda )
	+
	H_{\theta,y}^{\boldsymbol\alpha}(\Lambda ).
\end{align}

In addition to the previously introduced notation,
the following notation is used here, too.
$G_{\theta,\boldsymbol y}^{m:n}(\lambda,\!\tilde{\lambda} )$
is the element of ${\cal M}_{s}({\cal X} )$ recursively defined by
\begin{align}\label{4.107}
	G_{\theta,\boldsymbol y}^{m:m}(\lambda,\tilde{\lambda} )=\tilde{\lambda},
	\;\;\;\;\;
	G_{\theta,\boldsymbol y}^{m:n}(\lambda, \tilde{\lambda} )
	=
	G_{\theta,y_{n} }
	\left(P_{\theta,\boldsymbol y}^{m:n-1}(\lambda),
	G_{\theta,\boldsymbol y}^{m:n-1}(\lambda, \tilde{\lambda} )
	\right)
\end{align}
for $\theta\in\Theta$, $\lambda\in{\cal P}({\cal X} )$, $\tilde{\lambda}\in{\cal M}_{s}({\cal X} )$,
$n>m\geq 0$ and a sequence $\boldsymbol y = \{y_{n} \}_{n\geq 1}$ in ${\cal Y}$.
$V_{\theta,\boldsymbol y}^{\boldsymbol\alpha, m:n}(\Lambda )$
and
$W_{\theta,\boldsymbol y}^{\boldsymbol\alpha, m:n}(\Lambda )$
are the elements of ${\cal M}_{s}({\cal X} )$ defined by
\begin{align}\label{4.121}
	V_{\theta,\boldsymbol y}^{\boldsymbol\alpha, m:n}(\Lambda )
	=
	G_{\theta,\boldsymbol y}^{m:n}(\lambda_{\boldsymbol 0}, \lambda_{\boldsymbol\alpha } ),
	\;\;\;\;\;
	W_{\theta,\boldsymbol y}^{\boldsymbol\alpha, m:n}(\Lambda )
	=
	H_{\theta,y_{n} }^{\boldsymbol\alpha}(F_{\theta,\boldsymbol y}^{m:n-1}(\Lambda ) )
\end{align}
for $\Lambda=\left\{\lambda_{\boldsymbol\beta}:
\boldsymbol\beta\in\mathbb{N}_{0}^{d}, |\boldsymbol\beta|\leq p \right\} \in {\cal L}_{0}({\cal X} )$,
$\boldsymbol\alpha\in \mathbb{N}_{0}^{d}$,
$|\boldsymbol\alpha|\leq p$.
$\Phi_{\boldsymbol y}^{m:n}$ and
$\Psi_{\boldsymbol y}^{m:n}$ are the quantities
defined by
\begin{align*}
	\Phi_{\boldsymbol y}^{m:m} = 1,
	\;\;\;\;\;
	\Psi_{\boldsymbol y}^{m:m} = 1,
	\;\;\;\;\;
	\Phi_{\boldsymbol y}^{m:n} = (n-m) \sum_{k=m+1}^{n} \psi(y_{k} ),
	\;\;\;\;\;
	\Psi_{\boldsymbol y}^{m:n} = \sum_{k=m+1}^{n} \psi(y_{k} ).
\end{align*}
$M_{\boldsymbol\alpha}(\Lambda)$ is the function defined by
\begin{align*}
	M_{\boldsymbol\alpha}(\Lambda)
	=
	\max\left\{\|\lambda_{\boldsymbol\beta}\|: \boldsymbol\beta\in\mathbb{N}_{0}^{d},
	\boldsymbol\beta\leq\boldsymbol\alpha \right\}.
\end{align*}
$K_{\boldsymbol\alpha}(\Lambda,\Lambda')$ and $L_{\boldsymbol\alpha}(\Lambda,\Lambda')$
are the functions defined by
\begin{align}\label{4.109}
	K_{\boldsymbol\alpha}(\Lambda,\Lambda')
	=
	\min\{1, M_{\boldsymbol\alpha}(\Lambda-\Lambda') \},
	\;\;\;\;\;
	L_{\boldsymbol\alpha}(\Lambda,\Lambda')
	=
	M_{\boldsymbol\alpha}(\Lambda) + M_{\boldsymbol\alpha}(\Lambda')
\end{align}
for $\Lambda,\Lambda'\in{\cal L}({\cal X} )$.
$L_{\boldsymbol\alpha, \boldsymbol y}^{m:n}(\Lambda,\Lambda')$ and $M_{\boldsymbol\alpha, \boldsymbol y}^{m:n}(\Lambda)$
are the functions defined by
\begin{align}\label{4.5}
	L_{\boldsymbol\alpha, \boldsymbol y}^{m:n}(\Lambda,\Lambda')
	=
	\left( L_{\boldsymbol\alpha}(\Lambda,\Lambda') \Phi_{\boldsymbol y}^{m:n} \right)^{|\boldsymbol\alpha|},
	\;\;\;\;\;
	M_{\boldsymbol\alpha, \boldsymbol y}^{m:n}(\Lambda)
	=
	\left( M_{\boldsymbol\alpha}(\Lambda) \Psi_{\boldsymbol y}^{m:n} \right)^{|\boldsymbol\alpha|}
\end{align}
for $n\geq m\geq 0$.

\begin{remark}
Throughout this and subsequent sections, the following convention is applied.
Diacritic $\tilde{}$ is used to denote a locally defined quantity,
i.e., a quantity whose definition holds only within the proof where
the quantity appears.
\end{remark}

\begin{lemma}\label{lemma1.2}
Let Assumptions \ref{a1.1} and \ref{a1.2} hold.
Then, there exists a real number $C_{1}\in[1,\infty )$
(depending only on $\varepsilon$)
such that
\begin{align}
	&\label{l1.2.5*}
	\left\|
	\frac{R_{\theta, y}^{\boldsymbol\alpha}(\tilde{\lambda} )}
	{\big\langle R_{\theta, y}^{\boldsymbol 0}(\lambda ) \big\rangle }
	\right\|
	\leq
	C_{1} \left(\psi(y) \right)^{|\boldsymbol\alpha|}
	\|\tilde{\lambda} \|,
	\\
	&\label{l1.2.7*}
	\left\|
	\frac{R_{\theta, y}^{\boldsymbol\alpha}(\tilde{\lambda} )}
	{\big\langle R_{\theta, y}^{\boldsymbol 0}(\lambda) \big\rangle }
	-
	\frac{R_{\theta, y}^{\boldsymbol\alpha}(\tilde{\lambda}' )}
	{\big\langle R_{\theta, y}^{\boldsymbol 0}(\lambda' ) \big\rangle }
	\right\|
	\leq
	C_{1} \left(\psi(y) \right)^{|\boldsymbol\alpha|}
	\left(\|\tilde{\lambda} - \tilde{\lambda}' \|
	+ \|\lambda - \lambda' \| \|\tilde{\lambda}' \| \right)
\end{align}
for all $\theta\in\Theta$, $y\in{\cal Y}$, $\lambda, \lambda'\in{\cal P}({\cal X} )$,
$\tilde{\lambda},\tilde{\lambda}'\in{\cal M}_{s}({\cal X} )$
and any multi-index $\boldsymbol\alpha\in\mathbb{N}_{0}^{d}$, $|\boldsymbol\alpha|\leq p$.
\end{lemma}

\begin{proof}
Throughout the proof, we rely on the following notation.
$C_{1}$ is the real number defined by $C_{1}=\varepsilon^{-4}$
($\varepsilon$ is specified in Assumption \ref{a1.1}).
$\theta$, $y$ are any elements in $\Theta$, ${\cal Y}$ (respectively).
$\lambda,\lambda'$ are any elements of ${\cal P}({\cal X} )$,
while $\tilde{\lambda},\tilde{\lambda}'$ are any elements in ${\cal M}_{s}({\cal X} )$.
$\boldsymbol\alpha$ is any element of $\mathbb{N}_{0}^{d}$
satisfying $|\boldsymbol\alpha|\leq p$.

Owing to Assumption \ref{a1.1}, we have
\begin{align}\label{l1.2.201}
	\big\langle R_{\theta,y}^{\boldsymbol 0}(\lambda ) \big\rangle
	=
	\int\int r_{\theta}(y,x'|x) \mu(dx') \lambda(dx)
	\geq
	\varepsilon \mu_{\theta}({\cal X}|y).
\end{align}
Moreover, due to Assumptions \ref{a1.1}, \ref{a1.2}, we have
\begin{align}\label{l1.2.203}
	\left\|
	R_{\theta,y}^{\boldsymbol\alpha}(\tilde{\lambda} )
	\right\|
	\leq 
	\int\int \left|\partial_{\theta}^{\boldsymbol\alpha} r_{\theta}(y,x'|x) \right|
	\mu(dx') \: |\tilde{\lambda}|(dx)
	\leq &
	\left(\psi(y) \right)^{|\boldsymbol\alpha|}
	\int\int r_{\theta}(y,x'|x)
	\mu(dx') \: |\tilde{\lambda}|(dx)
	\nonumber\\
	\leq &
	\varepsilon^{-1}
	\left(\psi(y) \right)^{|\boldsymbol\alpha|} \:
	\|\tilde{\lambda} \| \mu_{\theta}({\cal X}|y).
\end{align}
Combining (\ref{l1.2.201}), (\ref{l1.2.203}), we get
\begin{align}\label{l1.2.1}
	\left\|
	\frac{R_{\theta, y}^{\boldsymbol\alpha}(\tilde{\lambda} )}
	{\big\langle R_{\theta, y}^{\boldsymbol 0}(\lambda ) \big\rangle }
	\right\|
	\leq
	\varepsilon^{-2} \left(\psi(y) \right)^{|\boldsymbol\alpha|}
	\|\tilde{\lambda} \|.
\end{align}
Consequently, we have
\begin{align}\label{l1.2.3}
	\left\|
	\frac{R_{\theta, y}^{\boldsymbol\alpha}(\tilde{\lambda} )}
	{\big\langle R_{\theta, y}^{\boldsymbol 0}(\lambda ) \big\rangle }
	-
	\frac{R_{\theta, y}^{\boldsymbol\alpha}(\tilde{\lambda}' )}
	{\big\langle R_{\theta, y}^{\boldsymbol 0}(\lambda' ) \big\rangle }
	\right\|
	\leq &
	\frac{\left\|R_{\theta, y}^{\boldsymbol\alpha}(\tilde{\lambda} )
	- R_{\theta, y}^{\boldsymbol\alpha}(\tilde{\lambda}' )\right\| }
	{\big\langle R_{\theta, y}^{\boldsymbol 0}(\lambda ) \big\rangle }
	+
	\frac{\left\|R_{\theta, y}^{\boldsymbol\alpha}(\tilde{\lambda}' )\right\|
	\left|\big\langle R_{\theta, y}^{\boldsymbol 0}(\lambda ) \big\rangle
	- \big\langle R_{\theta, y}^{\boldsymbol 0}(\lambda' ) \big\rangle\right|}
	{\big\langle R_{\theta, y}^{\boldsymbol 0}(\lambda ) \big\rangle
	\big\langle R_{\theta, y}^{\boldsymbol 0}(\lambda' ) \big\rangle }
	\nonumber\\
	\leq &
	\frac{\left\|R_{\theta, y}^{\boldsymbol\alpha}(\tilde{\lambda} - \tilde{\lambda}' ) \right\| }
	{\big\langle R_{\theta, y}^{\boldsymbol 0}(\lambda' ) \big\rangle }
	+
	\frac{\left\|R_{\theta, y}^{\boldsymbol\alpha}(\tilde{\lambda}' )\right\|
	\left\|R_{\theta, y}^{\boldsymbol 0}(\lambda - \lambda' ) \right\|}
	{\big\langle R_{\theta, y}^{\boldsymbol 0}(\lambda ) \big\rangle
	\big\langle R_{\theta, y}^{\boldsymbol 0}(\lambda' ) \big\rangle }
	\nonumber\\
	\leq &
	\varepsilon^{-4}
	\left(\psi(y) \right)^{|\boldsymbol\alpha|}
	\left(\|\tilde{\lambda} - \tilde{\lambda}' \|
	+ \|\lambda - \lambda' \| \|\tilde{\lambda}' \| \right).
\end{align}
Then, (\ref{l1.2.5*}), (\ref{l1.2.7*}) directly follow from (\ref{l1.2.1}), (\ref{l1.2.3}).
\end{proof}

\begin{lemma}\label{lemma1.4}
Let Assumptions \ref{a1.1} and \ref{a1.2} hold.
Then, there exists a real number $C_{2}\in[1,\infty )$
(depending only on $p$, $\varepsilon$)
such that
\begin{align}
	&\label{l1.4.1*}
	\left\|
	T_{\theta, y}^{\boldsymbol\alpha,\boldsymbol\beta}(\Lambda )
	\right\|
	\leq
	C_{2}
	\left(\psi(y) \right)^{|\boldsymbol\alpha-\boldsymbol\beta|}
	\|\lambda_{\boldsymbol\beta} \|,
	\\
	&\label{l1.4.3*}
	\left\|
	S_{\theta, y}^{\boldsymbol\alpha}(\Lambda )
	\right\|
	\leq
	C_{2}
	\sum_{\stackrel{\scriptstyle \boldsymbol\gamma\in\mathbb{N}_{0}^{d} }
	{\boldsymbol\gamma\leq\boldsymbol\alpha} }
	\left(\psi(y) \right)^{|\boldsymbol\alpha-\boldsymbol\gamma|}
	\|\lambda_{\boldsymbol\gamma} \|,
	\\
	&\label{l1.3.5*}
	\left\|
	T_{\theta, y}^{\boldsymbol\alpha,\boldsymbol\beta}(\Lambda )
	-
	T_{\theta, y}^{\boldsymbol\alpha,\boldsymbol\beta}(\Lambda' )
	\right\|
	\leq
	C_{2}
	\left(\psi(y) \right)^{|\boldsymbol\alpha-\boldsymbol\beta|}
	\left(
	\|\lambda_{\boldsymbol\beta} - \lambda'_{\boldsymbol\beta} \|
	+
	\|\lambda_{\boldsymbol 0} - \lambda'_{\boldsymbol 0} \|
	\|\lambda'_{\boldsymbol\beta} \|
	\right),
	\\
	&\label{l1.3.7*}
	\left\|
	S_{\theta, y}^{\boldsymbol\alpha}(\Lambda )
	-
	S_{\theta, y}^{\boldsymbol\alpha}(\Lambda' )
	\right\|
	\leq
	C_{2}
	\sum_{\stackrel{\scriptstyle \boldsymbol\gamma\in\mathbb{N}_{0}^{d} }
	{\boldsymbol\gamma\leq\boldsymbol\alpha} }
	\left(\psi(y) \right)^{|\boldsymbol\alpha-\boldsymbol\gamma|}
	\left(
	\|\lambda_{\boldsymbol\gamma} - \lambda'_{\boldsymbol\gamma} \|
	+
	\|\lambda_{\boldsymbol 0} - \lambda'_{\boldsymbol 0} \|
	\|\lambda''_{\boldsymbol\gamma} \|
	\right)
\end{align}
for all $\theta\in\Theta$, $y\in{\cal Y}$,
$\Lambda=\left\{\lambda_{\boldsymbol\gamma}:\boldsymbol\gamma\in\mathbb{N}_{0}^{d},
|\boldsymbol\gamma|\leq p \right\}\in{\cal L}_{0}({\cal X} )$,
$\Lambda'=\left\{\lambda'_{\boldsymbol\gamma}:\boldsymbol\gamma\in\mathbb{N}_{0}^{d},
|\boldsymbol\gamma|\leq p \right\}\in{\cal L}_{0}({\cal X} )$
and any multi-indices
$\boldsymbol\alpha,\boldsymbol\beta\in\mathbb{N}_{0}^{d}$,
$\boldsymbol\beta\leq\boldsymbol\alpha$, $|\boldsymbol\alpha|\leq p$.
\end{lemma}

\begin{proof}
Throughout the proof, we rely on the following notation.
$C_{2}$ is the real number defined by $C_{2}=2^{p}C_{1}$
($C_{1}$ is specified in Lemma \ref{lemma1.2}).
$\theta$, $y$ are any elements in $\Theta$, ${\cal Y}$ (respectively),
while $\Lambda\!=\!\left\{\lambda_{\boldsymbol\gamma}:\boldsymbol\gamma\in\mathbb{N}_{0}^{d},
|\boldsymbol\gamma|\leq p \right\}$,
$\Lambda'=\left\{\lambda'_{\boldsymbol\gamma}:\boldsymbol\gamma\in\mathbb{N}_{0}^{d},
|\boldsymbol\gamma|\leq p \right\}$.
$\boldsymbol\alpha,\boldsymbol\beta$ are any elements of $\mathbb{N}_{0}^{d}$
satisfying
$\boldsymbol\beta\leq\boldsymbol\alpha$, $|\boldsymbol\alpha|\leq p$.

Since
$\sum_{\stackrel{\scriptstyle \boldsymbol\gamma\in\mathbb{N}_{0}^{d} }
{\boldsymbol\gamma\leq\boldsymbol\alpha } }\left(\boldsymbol \alpha \atop \boldsymbol\gamma \right)
=2^{|\boldsymbol\alpha|}$,
Lemma \ref{lemma1.2} and (\ref{4.151}) imply
\begin{align}\label{l1.3.1}
	\left\|S_{\theta, y }^{\boldsymbol\alpha }(\Lambda ) \right\|
	\leq 
	\sum_{\stackrel{\scriptstyle \boldsymbol\gamma\in\mathbb{N}_{0}^{d} }
	{\boldsymbol\gamma\leq\boldsymbol\alpha } }
	\left(\boldsymbol \alpha \atop \boldsymbol\gamma \right)
	\left\|
	\frac{R_{\theta,y}^{\boldsymbol\alpha-\boldsymbol\gamma}(\lambda_{\boldsymbol\gamma} ) }
	{\big\langle R_{\theta,y}^{\boldsymbol 0}(\lambda_{\boldsymbol 0} ) \big\rangle }
	\right\|
	\leq &
	2^{|\boldsymbol\alpha|}C_{1}
	\sum_{\stackrel{\scriptstyle \boldsymbol\gamma\in\mathbb{N}_{0}^{d} }
	{\boldsymbol\gamma\leq\boldsymbol\alpha } }
	\left(\psi(y) \right)^{|\boldsymbol\alpha-\boldsymbol\gamma|}
	\|\lambda_{\boldsymbol\gamma} \|.
\end{align}
As $F_{\theta,y}^{\boldsymbol 0}(\Lambda )=
R_{\theta,y}^{\boldsymbol 0}(\lambda_{\boldsymbol 0} ) /
\big\langle R_{\theta,y}^{\boldsymbol 0}(\lambda_{\boldsymbol 0} ) \big\rangle
\in{\cal P}({\cal X} )$,
the same arguments and (\ref{4.103}) yield
\begin{align}\label{l1.3.3}
	\left\|
	T_{\theta, y }^{\boldsymbol\alpha,\boldsymbol\beta}(\Lambda )
	\right\|
	\leq &
	\left\|
	\frac{R_{\theta,y}^{\boldsymbol\alpha-\boldsymbol\beta}(\lambda_{\boldsymbol\beta} ) }
	{\big\langle R_{\theta,y}^{\boldsymbol 0}(\lambda_{\boldsymbol 0} ) \big\rangle }
	\right\|
	+
	\left\|F_{\theta,y}^{\boldsymbol 0}(\Lambda ) \right\|
	\left|
	\frac{\big\langle R_{\theta,y}^{\boldsymbol\alpha-\boldsymbol\beta}(\lambda_{\boldsymbol\beta} ) \big\rangle }
	{\big\langle R_{\theta,y}^{\boldsymbol 0}(\lambda_{\boldsymbol 0} ) \big\rangle }
	\right|
	\leq 
	2
	\left\|
	\frac{R_{\theta,y}^{\boldsymbol\alpha-\boldsymbol\beta}(\lambda_{\boldsymbol\beta} ) }
	{\big\langle R_{\theta,y}^{\boldsymbol 0}(\lambda_{\boldsymbol 0} ) \big\rangle }
	\right\|
	\leq 
	2 C_{1}
	\left(\psi(y) \right)^{|\boldsymbol\alpha-\boldsymbol\beta|}
	\|\lambda_{\boldsymbol\beta}\|.
\end{align}
Then, (\ref{l1.4.1*}), (\ref{l1.4.3*})
directly follow from (\ref{l1.3.1}), (\ref{l1.3.3}).

Using Lemma \ref{lemma1.2} and (\ref{4.151}), we conclude
\begin{align}\label{l1.3.5}
	\left\|
	S_{\theta, y }^{\boldsymbol\alpha }(\Lambda )
	-
	S_{\theta, y }^{\boldsymbol\alpha }(\Lambda' )
	\right\|
	\leq &
	\sum_{\stackrel{\scriptstyle \boldsymbol\gamma\in\mathbb{N}_{0}^{d} }
	{\boldsymbol\gamma\leq\boldsymbol\alpha } }
	\left(\boldsymbol \alpha \atop \boldsymbol\gamma \right)
	\left\|
	\frac{R_{\theta,y}^{\boldsymbol\alpha-\boldsymbol\gamma}(\lambda_{\boldsymbol\gamma} ) }
	{\big\langle R_{\theta,y}^{\boldsymbol 0}(\lambda_{\boldsymbol 0} ) \big\rangle }
	-
	\frac{R_{\theta,y}^{\boldsymbol\alpha-\boldsymbol\gamma}(\lambda'_{\boldsymbol\gamma} ) }
	{\big\langle R_{\theta,y}^{\boldsymbol 0}(\lambda'_{\boldsymbol 0} ) \big\rangle }
	\right\|
	\nonumber\\
	\leq &
	2^{p}C_{1}
	\sum_{\stackrel{\scriptstyle \boldsymbol\gamma\in\mathbb{N}_{0}^{d} }
	{\boldsymbol\gamma\leq\boldsymbol\alpha } }
	\left(\psi(y) \right)^{|\boldsymbol\alpha-\boldsymbol\gamma|}
	\left(
	\|\lambda_{\boldsymbol\gamma} - \lambda'_{\boldsymbol\gamma} \|
	+
	\|\lambda_{\boldsymbol 0} - \lambda'_{\boldsymbol 0} \| \|\lambda'_{\boldsymbol\gamma} \|
	\right).
\end{align}
Relying on the same arguments and (\ref{4.103}), we deduce
\begin{align}\label{l1.3.7}
	\left\|
	T_{\theta, y }^{\boldsymbol\alpha,\boldsymbol\beta}(\Lambda )
	-
	T_{\theta, y }^{\boldsymbol\alpha,\boldsymbol\beta}(\Lambda' )
	\right\|
	\leq &
	\left\|
	\frac{R_{\theta,y}^{\boldsymbol\alpha-\boldsymbol\beta}(\lambda_{\boldsymbol\beta} ) }
	{\big\langle R_{\theta,y}^{\boldsymbol 0}(\lambda_{\boldsymbol 0} ) \big\rangle }
	-
	\frac{R_{\theta,y}^{\boldsymbol\alpha-\boldsymbol\beta}(\lambda'_{\boldsymbol\beta} ) }
	{\big\langle R_{\theta,y}^{\boldsymbol 0}(\lambda'_{\boldsymbol 0} ) \big\rangle }
	\right\|
	+
	\left\|
	F_{\theta,y}^{\boldsymbol 0}(\Lambda )
	-
	F_{\theta,y}^{\boldsymbol 0}(\Lambda' )
	\right\|
	\left|
	\frac
	{\big\langle R_{\theta,y}^{\boldsymbol\alpha-\boldsymbol\beta}(\lambda'_{\boldsymbol\beta} ) \big\rangle }
	{\big\langle R_{\theta,y}^{\boldsymbol 0}(\lambda'_{\boldsymbol 0} ) \big\rangle }
	\right|
	\nonumber\\
	&+
	\left\|
	F_{\theta,y}^{\boldsymbol 0}(\Lambda )
	\right\|
	\left|
	\frac
	{\big\langle R_{\theta,y}^{\boldsymbol\alpha-\boldsymbol\beta}(\lambda_{\boldsymbol\beta} ) \big\rangle }
	{\big\langle R_{\theta,y}^{\boldsymbol 0}(\lambda_{\boldsymbol 0} ) \big\rangle }
	-
	\frac
	{\big\langle R_{\theta,y}^{\boldsymbol\alpha-\boldsymbol\beta}(\lambda'_{\boldsymbol\beta} ) \big\rangle}
	{\big\langle R_{\theta,y}^{\boldsymbol 0}(\lambda'_{\boldsymbol 0} ) \big\rangle }
	\right|
	\nonumber\\
	\leq &
	2
	\left\|
	\frac{R_{\theta,y}^{\boldsymbol\alpha-\boldsymbol\beta}(\lambda_{\boldsymbol\beta} ) }
	{\big\langle R_{\theta,y}^{\boldsymbol 0}(\lambda_{\boldsymbol 0} ) \big\rangle }
	-
	\frac{R_{\theta,y}^{\boldsymbol\alpha-\boldsymbol\beta}(\lambda'_{\boldsymbol\beta} ) }
	{\big\langle R_{\theta,y}^{\boldsymbol 0}(\lambda'_{\boldsymbol 0} ) \big\rangle }
	\right\|
	+
	\left\|
	\frac{R_{\theta,y}^{\boldsymbol 0}(\lambda_{\boldsymbol 0} ) }
	{\big\langle R_{\theta,y}^{\boldsymbol 0}(\lambda_{\boldsymbol 0} ) \big\rangle }
	-
	\frac{R_{\theta,y}^{\boldsymbol 0}(\lambda'_{\boldsymbol 0} ) }
	{\big\langle R_{\theta,y}^{\boldsymbol 0}(\lambda'_{\boldsymbol 0} ) \big\rangle }
	\right\|
	\left\|
	\frac
	{R_{\theta,y}^{\boldsymbol\alpha-\boldsymbol\beta}(\lambda'_{\boldsymbol\beta} ) }
	{\big\langle R_{\theta,y}^{\boldsymbol 0}(\lambda'_{\boldsymbol 0} ) \big\rangle }
	\right\|
	\nonumber\\
	\leq &
	2 C_{1}
	\left(\psi(y) \right)^{|\boldsymbol\alpha-\boldsymbol\beta|}
	\left(
	\|\lambda_{\boldsymbol\beta} - \lambda'_{\boldsymbol\beta} \|
	+
	\|\lambda_{\boldsymbol 0} - \lambda'_{\boldsymbol 0} \| \|\lambda'_{\boldsymbol\beta} \|
	\right).
\end{align}
Then, (\ref{l1.3.5*}), (\ref{l1.3.7*}) directly follow from (\ref{l1.3.5}), (\ref{l1.3.7}).
\end{proof}

\begin{proposition}\label{proposition1.1}
Let Assumption \ref{a1.1} hold.
Then, there exists a real number $C_{3}\in[1,\infty )$
(depending only on $\varepsilon$) such that
\begin{align*}
	&
	\left\|G_{\theta,\boldsymbol y}^{m:n}(\lambda,\tilde{\lambda} ) \right\|
	\leq
	C_{3} \tau^{2(n-m)} \big\|\tilde{\lambda} \big\|,
	\\
	&
	\left\|G_{\theta,\boldsymbol y}^{m:n}(\lambda, \tilde{\lambda} )
	-
	G_{\theta,\boldsymbol y}^{m:n}(\lambda', \tilde{\lambda}' ) \right\|
	\leq
	C_{3} \tau^{2(n-m)}
	\left(
	\big\|\tilde{\lambda} - \tilde{\lambda}' \big\|
	+
	\big\|\lambda - \lambda' \big\| \: \big\|\tilde{\lambda}' \big\|
	\right)
\end{align*}
for all $\theta\in\Theta$, $\lambda,\lambda'\in{\cal P}({\cal X} )$,
$\tilde{\lambda},\tilde{\lambda}'\in{\cal M}_{s}({\cal X} )$,
$n\geq m\geq 0$
and any sequence $\boldsymbol y = \{y_{n}\}_{n\geq 1}$
in ${\cal Y}$
($\tau$ is defined at the beginning of Section \ref{section1.1*}).
\end{proposition}

\begin{proof}
See \cite[Lemmas 6.6, 6.7]{tadic&doucet1}.
\end{proof}

\begin{proposition}\label{proposition1.2}
Let Assumptions \ref{a1.1} and \ref{a1.2} hold.
Then, there exists a real number $C_{4}\in[1,\infty )$
(depending only on $\varepsilon$) such that
\begin{align*}
\left\|F_{\theta, \boldsymbol y}^{\boldsymbol 0, m:n}(\Lambda )
-
F_{\theta, \boldsymbol y}^{\boldsymbol 0, m:n}(\Lambda' ) \right\|
\leq
C_{4}
\tau^{2(n-m)}
K_{\boldsymbol 0}(\Lambda,\Lambda')
\end{align*}
for all $\theta\in\Theta$, $\Lambda,\Lambda'\in {\cal L}_{0}({\cal X} )$, $n\geq m\geq 0$
and any sequence $\boldsymbol y = \{y_{n} \}_{n\geq 1}$ in ${\cal Y}$
($\tau$ is defined at the beginning of Section \ref{section1.1*}).
\end{proposition}

\begin{proof}
Let $\theta$ be any element of $\Theta$, while
$\boldsymbol y = \{y_{n} \}_{n\geq 1}$ is any sequence in ${\cal Y}$.
Moreover, let
$\Lambda = \big\{\lambda_{\boldsymbol\beta}:
\boldsymbol\beta\in\mathbb{N}_{0}^{d}, |\boldsymbol\beta|\leq p \big\}$,
$\Lambda' = \big\{\lambda'_{\boldsymbol\beta}:
\boldsymbol\beta\in\mathbb{N}_{0}^{d}, |\boldsymbol\beta|\leq p \big\}$
be any elements of ${\cal L}_{0}({\cal X} )$,
while $n,m$ are any integers satisfying $n\geq m\geq 0$.

Using (\ref{1.903}), (\ref{1.905}), we conclude
$P_{\theta, \boldsymbol y}^{m:m+1}(\lambda_{\boldsymbol 0} ) =
F_{\theta, y_{m+1} }^{\boldsymbol 0}(\lambda_{\boldsymbol 0} )$,
$P_{\theta, \boldsymbol y}^{m:m+1}(\lambda'_{\boldsymbol 0} ) =
F_{\theta, y_{m+1} }^{\boldsymbol 0}(\lambda'_{\boldsymbol 0} )$
and
\begin{align*}
	P_{\theta, \boldsymbol y}^{m:n+1}(\lambda_{\boldsymbol 0} )
	=
	F_{\theta, y_{n+1} }^{\boldsymbol 0}\left(P_{\theta, \boldsymbol y}^{m:n}(\lambda_{\boldsymbol 0} ) \right),
	\;\;\;\;\;
	P_{\theta, \boldsymbol y}^{m:n+1}(\lambda'_{\boldsymbol 0} )
	=
	F_{\theta, y_{n+1} }^{\boldsymbol 0}\left(P_{\theta, \boldsymbol y}^{m:n}(\lambda'_{\boldsymbol 0} ) \right).
\end{align*}
Comparing this with (\ref{1.705}), we get
\begin{align*}
	F_{\theta, \boldsymbol y}^{\boldsymbol 0, m:n}(\Lambda )
	=
	P_{\theta, \boldsymbol y}^{m:n}(\lambda_{\boldsymbol 0} ),
	\;\;\;\;\;
	F_{\theta, \boldsymbol y}^{\boldsymbol 0, m:n}(\Lambda' )
	=
	P_{\theta, \boldsymbol y}^{m:n}(\lambda'_{\boldsymbol 0} )
\end{align*}
(i.e., $F_{\theta, \boldsymbol y}^{\boldsymbol 0, m:n}(\Lambda )$,
$F_{\theta, \boldsymbol y}^{\boldsymbol 0, m:n}(\Lambda' )$
are the filtering distributions initialized by $\lambda_{\boldsymbol 0}$, $\lambda'_{\boldsymbol 0}$).
Consequently, \cite[Theorem 3.1]{tadic&doucet1} implies that
there exists a real number $C_{4}\in[1,\infty )$
(depending only on $\varepsilon$) such that
\begin{align}
	&\label{p1.2.3}
	\left\|F_{\theta,\boldsymbol y}^{\boldsymbol 0, m:n}(\Lambda)
	-
	F_{\theta,\boldsymbol y}^{\boldsymbol 0, m:n}(\Lambda')
	\right\|
	\leq
	C_{4} \tau^{2(n-m)} \|\lambda_{\boldsymbol 0} - \lambda'_{\boldsymbol 0} \|
	=
	C_{4} \tau^{2(n-m)} K_{\boldsymbol 0}(\Lambda,\Lambda' ).
\end{align}
\end{proof}

\begin{lemma}\label{lemma1.12}
Let Assumptions \ref{a1.1} and \ref{a1.2} hold.
Then, we have
\begin{align}\label{l1.12.1*}
	F_{\theta,\boldsymbol y}^{\boldsymbol\alpha, m:n }(\Lambda )
	=
	V_{\theta,\boldsymbol y}^{\boldsymbol\alpha, m:n }(\Lambda )
	+
	\sum_{k=m+1}^{n}
	G_{\theta,\boldsymbol y}^{k:n}\left(
	F_{\theta,\boldsymbol y}^{\boldsymbol 0, m:k}(\Lambda ),
	W_{\theta,\boldsymbol y}^{\boldsymbol\alpha, m:k }(\Lambda )
	\right)
\end{align}
for all $\theta\in\Theta$, $\Lambda\in{\cal L}_{0}({\cal X} )$, $n\geq m\geq 0$,
any multi-index $\boldsymbol\alpha\in\mathbb{N}_{0}^{d}$, $|\boldsymbol\alpha|\leq p$
and any sequence $\boldsymbol y = \{y_{n} \}_{n\geq 1}$ in ${\cal Y}$.
Here and throughout the paper, we rely on the convention that $\sum_{k=i}^{j}$ is zero
whenever $j<i$.
\end{lemma}

\begin{proof}
Throughout the proof, the following notation is used.
$\theta$ is any element of $\Theta$,
while
$\Lambda = \big\{\lambda_{\boldsymbol\beta}:
\boldsymbol\beta \in \mathbb{N}_{0}^{d}, |\boldsymbol\beta|\leq p \big\}$
is any element of ${\cal L}_{0}({\cal X} )$.
$m$ is any non-negative integer, while
$\boldsymbol\alpha$ is any element of $\mathbb{N}_{0}^{d}$
satisfying $|\boldsymbol\alpha|\leq p$.
$\boldsymbol y = \{y_{n} \}_{n\geq 1}$ is any sequence in ${\cal Y}$.

We prove (\ref{l1.12.1*}) by induction in $n$.
Owing to (\ref{1.705}), (\ref{4.107}), (\ref{4.121}), we have
\begin{align*}
	F_{\theta,\boldsymbol y}^{\boldsymbol\alpha, m:m}(\Lambda )
	=
	\lambda_{\boldsymbol\alpha},
	\;\;\;\;\;
	V_{\theta,\boldsymbol y}^{\boldsymbol\alpha, m:m}(\Lambda )
	=
	G_{\theta,\boldsymbol y}^{m:m}(\lambda_{\boldsymbol 0}, \lambda_{\boldsymbol\alpha} )
	=
	\lambda_{\boldsymbol\alpha}.
\end{align*}
Hence, (\ref{l1.12.1*}) is true when $n=m$.
Now, suppose that (\ref{l1.12.1*}) holds for some integer $n$ satisfying $n\geq m$.
As $G_{\theta,y}(\lambda,\tilde{\lambda} )$ is linear in $\tilde{\lambda}$,
we then get
\begin{align}\label{l1.12.1}
	G_{\theta,y_{n+1} }\left(
	F_{\theta,\boldsymbol y}^{\boldsymbol 0, m:n}(\Lambda ),
	F_{\theta,\boldsymbol y}^{\boldsymbol\alpha, m:n}(\Lambda )
	\right)
	=&
	G_{\theta,y_{n+1} }\left(
	F_{\theta,\boldsymbol y}^{\boldsymbol 0, m:n}(\Lambda ),
	V_{\theta,\boldsymbol y}^{\boldsymbol\alpha, m:n}(\Lambda )
	\right)
	\nonumber\\
	&
	+
	\sum_{k=m+1}^{n}
	G_{\theta,y_{n+1} }\left(
	F_{\theta,\boldsymbol y}^{\boldsymbol 0, m:n}(\Lambda ),
	G_{\theta,\boldsymbol y}^{k:n}\left(
	F_{\theta,\boldsymbol y}^{\boldsymbol 0, m:k}(\Lambda ),
	W_{\theta,\boldsymbol y}^{\boldsymbol\alpha, m:k}(\Lambda ) \right)
	\right).
\end{align}
Since
$P_{\theta,\boldsymbol y}^{m:n}(\lambda_{\boldsymbol 0} ) =
F_{\theta,\boldsymbol y}^{\boldsymbol 0, m:n}(\Lambda )$
(for further details, see the proof of Proposition \ref{proposition1.2}),
(\ref{4.107}), (\ref{4.121}) imply
\begin{align}\label{l1.12.3}
	G_{\theta,y_{n+1} }\left(
	F_{\theta,\boldsymbol y}^{\boldsymbol 0, m:n}(\Lambda ),
	V_{\theta,\boldsymbol y}^{\boldsymbol\alpha, m:n}(\Lambda )
	\right)
	=\!
	G_{\theta,y_{n+1} }\left(
	P_{\theta,\boldsymbol y}^{m:n}(\lambda_{\boldsymbol 0} ),
	G_{\theta,\boldsymbol y}^{m:n}(\lambda_{\boldsymbol 0}, \lambda_{\boldsymbol\alpha} )
	\right)
	=\!
	G_{\theta,\boldsymbol y}^{m:n+1}(\lambda_{\boldsymbol 0}, \lambda_{\boldsymbol\alpha} )
	=\!
	V_{\theta,\boldsymbol y}^{\boldsymbol\alpha, m:n+1}(\Lambda ).
\end{align}
Moreover, due to (\ref{1.705}), we have
\begin{align*}
	F_{\theta,\boldsymbol y}^{\boldsymbol 0,m:n}(\Lambda )
	=
	F_{\theta,\boldsymbol y}^{\boldsymbol 0,k:n}(F_{\theta,\boldsymbol y}^{\boldsymbol 0,m:k}(\Lambda ) )
	=
	P_{\theta,\boldsymbol y}^{k:n}(F_{\theta,\boldsymbol y}^{\boldsymbol 0,m:k}(\Lambda ) )
\end{align*}
(for further details, see again the proof of Proposition \ref{proposition1.2}).
Consequently, (\ref{4.107}), (\ref{4.121}) yield
\begin{align}\label{l1.12.5}
	&
	G_{\theta,y_{n+1} }\left(
	F_{\theta,\boldsymbol y}^{\boldsymbol 0, m:n}(\Lambda ),
	G_{\theta,\boldsymbol y}^{k:n}\left(
	F_{\theta,\boldsymbol y}^{\boldsymbol 0, m:k}(\Lambda ),
	W_{\theta,\boldsymbol y}^{\boldsymbol\alpha, m:k}(\Lambda ) \right)
	\right)
	\nonumber\\
	&
	=
	G_{\theta,y_{n+1} }\left(
	P_{\theta,\boldsymbol y}^{k:n}\left(
	F_{\theta,\boldsymbol y}^{\boldsymbol 0, m:k}(\Lambda )
	\right),
	G_{\theta,\boldsymbol y}^{k:n}\left(
	F_{\theta,\boldsymbol y}^{\boldsymbol 0, m:k}(\Lambda ),
	W_{\theta,\boldsymbol y}^{\boldsymbol\alpha, m:k}(\Lambda ) \right)
	\right)
	\nonumber\\
	&
	=
	G_{\theta,\boldsymbol y}^{k:n+1}\left(
	F_{\theta,\boldsymbol y}^{\boldsymbol 0, m:k}(\Lambda ),
	W_{\theta,\boldsymbol y}^{\boldsymbol\alpha, m:k}(\Lambda ) \right)
\end{align}
for $n\geq k> m$.
Similarly, (\ref{4.107}) implies
\begin{align}\label{l1.12.7}
	W_{\theta,\boldsymbol y}^{\boldsymbol\alpha, m:n+1}(\Lambda )
	=
	G_{\theta,\boldsymbol y}^{n+1:n+1}\left(
	F_{\theta,\boldsymbol y}^{\boldsymbol 0, m:n+1}(\Lambda ),
	W_{\theta,\boldsymbol y}^{\boldsymbol\alpha, m:n+1}(\Lambda ) \right).
\end{align}
Combining (\ref{l1.12.1}) -- (\ref{l1.12.5}), we get
\begin{align*}
	G_{\theta,y_{n+1} }\left(
	F_{\theta,\boldsymbol y}^{\boldsymbol 0, m:n}(\Lambda ),
	F_{\theta,\boldsymbol y}^{\boldsymbol\alpha, m:n}(\Lambda )
	\right)
	=&
	V_{\theta,\boldsymbol y}^{\boldsymbol\alpha, m:n+1}(\Lambda )
	+
	\sum_{k=m+1}^{n}
	G_{\theta,\boldsymbol y}^{k:n+1}\left(
	F_{\theta,\boldsymbol y}^{\boldsymbol 0, m:k}(\Lambda ),
	W_{\theta,\boldsymbol y}^{\boldsymbol\alpha, m:k}(\Lambda ) \right).
\end{align*}
Consequently,
(\ref{4.105}), (\ref{4.1}), (\ref{l1.12.7}) imply
\begin{align*}
	F_{\theta,\boldsymbol y}^{\boldsymbol\alpha, m:n+1}(\Lambda )
	=
	F_{\theta,y_{n+1} }^{\boldsymbol\alpha}\left(
	F_{\theta,\boldsymbol y}^{m:n}(\Lambda ) \right)
	=&
	G_{\theta,y_{n+1} }^{\boldsymbol\alpha}\left(
	F_{\theta,\boldsymbol y}^{m:n}(\Lambda ) \right)
	+
	H_{\theta,y_{n+1} }^{\boldsymbol\alpha}\left(
	F_{\theta,\boldsymbol y}^{m:n}(\Lambda ) \right)
	\\
	=&
	G_{\theta,y_{n+1} }\left(
	F_{\theta,\boldsymbol y}^{\boldsymbol 0, m:n}(\Lambda ),
	F_{\theta,\boldsymbol y}^{\boldsymbol\alpha, m:n}(\Lambda )
	\right)
	+
	W_{\theta,\boldsymbol y}^{\boldsymbol\alpha, m:n+1}(\Lambda )
	\\
	=&
	V_{\theta,\boldsymbol y}^{\boldsymbol\alpha, m:n+1}(\Lambda )
	+
	\sum_{k=m+1}^{n+1}
	G_{\theta,\boldsymbol y}^{k:n+1}\left(
	F_{\theta,\boldsymbol y}^{\boldsymbol 0, m:k}(\Lambda ),
	W_{\theta,\boldsymbol y}^{\boldsymbol\alpha, m:k}(\Lambda ) \right).
\end{align*}
Hence, (\ref{l1.12.1*}) is true for $n+1$.
Then, the lemma directly follows by the principle of mathematical induction.
\end{proof}

\begin{proposition}\label{proposition1.3}
Let Assumptions \ref{a1.1} and \ref{a1.2} hold.
Then, for each multi-index $\boldsymbol\alpha\in\mathbb{N}_{0}^{d}$, $|\boldsymbol\alpha|\leq p$,
there exists a real numbers $A_{\boldsymbol\alpha}\in [1,\infty )$
(depending only on $p$, $\varepsilon$)
such that
\begin{align}
	&\label{p1.3.1*}
	\left\|F_{\theta, \boldsymbol y}^{\boldsymbol\alpha, m:n}(\Lambda ) \right\|
	\leq
	A_{\boldsymbol\alpha}
	M_{\boldsymbol\alpha, \boldsymbol y}^{m:n}(\Lambda),
	\\
	&\label{p1.3.3*}
	\left\|F_{\theta, \boldsymbol y}^{\boldsymbol\alpha, m:n}(\Lambda )
	-
	F_{\theta, \boldsymbol y}^{\boldsymbol\alpha, m:n}(\Lambda' ) \right\|
	\leq
	\tau^{2(n-m)}
	A_{\boldsymbol\alpha} K_{\boldsymbol\alpha}(\Lambda,\Lambda')
	L_{\boldsymbol\alpha, \boldsymbol y}^{m:n}(\Lambda,\Lambda')
\end{align}
for all $\theta\in\Theta$, $\Lambda,\Lambda'\in {\cal L}_{0}({\cal X} )$, $n\geq m\geq 0$
and any sequence $\boldsymbol y = \{y_{n} \}_{n\geq 1}$ in ${\cal Y}$
($\tau$ is defined at the beginning of Section \ref{section1.1*}).
\end{proposition}

\begin{proof}
Throughout the proof, the following notation is used.
$\theta$ is any element of $\Theta$, while
$\boldsymbol y = \{y_{n} \}_{n\geq 1}$ is any sequence in ${\cal Y}$.
$\tilde{C}_{1}$, $\tilde{C}_{2}$, $\tilde{C}_{3}$ are the real numbers defined by
\begin{align*}
	\tilde{C}_{1}
	=
	\frac{4^{p}C_{2}C_{3}C_{4} }
	{\tau^{2}(1-\tau^{2} ) },
	\;\;\;\;\;
	\tilde{C}_{2}
	=
	\frac{\tilde{C}_{1} }{4^{p} },
	\;\;\;\;\;
	\tilde{C}_{3}
	=
	\frac{\tilde{C}_{2} }{C_{3}C_{4} }
\end{align*}
($C_{2}$, $C_{3}$, $C_{4}$ are specified in Lemma \ref{lemma1.4} and Propositions \ref{proposition1.1}, \ref{proposition1.2}).
$A_{\boldsymbol\alpha}$ is the real number defined by
$A_{\boldsymbol\alpha} = \exp\big(8\tilde{C}_{1}^{2}(|\boldsymbol\alpha|^{2} + 1 ) \big)$
for $\boldsymbol\alpha\in\mathbb{N}_{0}^{d}$.
Then, it easy to show
\begin{align}\label{p1.3.701}
	A_{\boldsymbol\beta}
	\leq
	\frac{A_{\boldsymbol\alpha} }{\exp(8\tilde{C}_{1}^{2} ) }
	\leq
	\frac{A_{\boldsymbol\alpha} }{8\tilde{C}_{1}^{2} },
	\;\;\;\;\;
	A_{\boldsymbol\gamma} A_{\boldsymbol\alpha-\boldsymbol\gamma}
	\leq
	\frac{A_{\boldsymbol\alpha} }{\exp(8\tilde{C}_{1}^{2} ) }
	\leq
	\frac{A_{\boldsymbol\alpha} }{8\tilde{C}_{1}^{2} }
\end{align}
for
$\boldsymbol\beta\in\mathbb{N}_{0}^{d}\setminus\{\boldsymbol\alpha\}$,
$\boldsymbol\gamma\in\mathbb{N}_{0}^{d}\setminus\{\boldsymbol 0,\boldsymbol\alpha\}$,
$\boldsymbol\beta\leq\boldsymbol\alpha$, $\boldsymbol\gamma\leq\boldsymbol\alpha$.

Since $F_{\theta,\boldsymbol y}^{m:m}(\Lambda)=\Lambda$
(due to (\ref{1.705})), (\ref{p1.3.1*}), (\ref{p1.3.3*}) are trivially satisfied
when $n=m\geq 0$.
For $n>m\geq 0$, we prove (\ref{p1.3.1*}), (\ref{p1.3.3*})
by the mathematical induction in $|\boldsymbol\alpha |$.
As $F_{\theta,\boldsymbol y}^{\boldsymbol 0,m:n}(\Lambda)\in{\cal P}({\cal X})$,
Proposition \ref{proposition1.2} implies
that when $|\boldsymbol\alpha|=0$ (i.e., $\boldsymbol\alpha=\boldsymbol 0$),
(\ref{p1.3.1*}), (\ref{p1.3.3*}) are true
for all $\Lambda, \Lambda'\in{\cal L}_{0}({\cal X} )$, $n,m\in\mathbb{N}_{0}$
fulfilling $n>m\geq 0$.
Now, the induction hypothesis is formulated:
Suppose that (\ref{p1.3.1*}), (\ref{p1.3.3*}) hold
for some $l\in\mathbb{N}_{0}$ and
all $\Lambda, \Lambda'\in{\cal L}_{0}({\cal X} )$, $n,m\in\mathbb{N}_{0}$,
$\boldsymbol\alpha\in\mathbb{N}_{0}^{d}$
satisfying $0\leq l< p$, $n>m\geq 0$, $|\boldsymbol\alpha |\leq l$.
Then, to prove (\ref{p1.3.1*}), (\ref{p1.3.3*}),
it is sufficient to show (\ref{p1.3.1*}), (\ref{p1.3.3*})
for any $\Lambda, \Lambda'\in{\cal L}_{0}({\cal X} )$, $n,m\in\mathbb{N}_{0}$,
$\boldsymbol\alpha \in \mathbb{N}_{0}^{d}$
fulfilling $n>m\geq 0$, $|\boldsymbol\alpha | = l+1$.
In what follows in the proof,
$\Lambda = \left\{\lambda_{\boldsymbol\alpha}:
\boldsymbol\alpha\in\mathbb{N}_{0}^{d}, |\boldsymbol\alpha|\leq p \right\}$,
$\Lambda' = \left\{\lambda'_{\boldsymbol\alpha}:
\boldsymbol\alpha\in\mathbb{N}_{0}^{d}, |\boldsymbol\alpha|\leq p \right\}$
are any elements of ${\cal L}_{0}({\cal X} )$.
$\boldsymbol\delta$ is any element of $\mathbb{N}_{0}^{d}$,
while $\boldsymbol\alpha$ is any element of $\mathbb{N}_{0}^{d}$
satisfying $|\boldsymbol\alpha | = l+1$.
$\boldsymbol\beta,\boldsymbol\gamma$ are any elements of $\mathbb{N}_{0}^{d}\setminus\{\boldsymbol\alpha\}$
fulfilling $\boldsymbol\beta\leq\boldsymbol\alpha$, $\boldsymbol\gamma\leq\boldsymbol\alpha$.
$n,m$ are any integers satisfying $n>m\geq 0$.

Since $\boldsymbol\beta\leq\boldsymbol\alpha$,
$\boldsymbol\beta\neq\boldsymbol\alpha$, we have $|\boldsymbol\beta|\leq |\boldsymbol\alpha|-1 = l$.
As (\ref{p1.3.1*}), (\ref{p1.3.3*}) are trivially satisfied for $n=m$,
the induction hypothesis imply
\begin{align}
	&\label{p1.3.21}
	\max\left\{
	\frac{
	\big\| F_{\theta, \boldsymbol y}^{\boldsymbol\gamma, m:k}(\Lambda) \big\| }
	{M_{\boldsymbol\gamma, \boldsymbol y}^{m:k}(\Lambda) },
	\frac{
	\big\| F_{\theta, \boldsymbol y}^{\boldsymbol\gamma, m:k}(\Lambda') \big\| }
	{M_{\boldsymbol\gamma, \boldsymbol y}^{m:k}(\Lambda') }
	\right\}
	\leq
	A_{\boldsymbol\gamma},
	\\
	&\label{p1.3.43}
	\frac{
	\big\|F_{\theta, \boldsymbol y}^{\boldsymbol\gamma, m:k}(\Lambda )
	-
	F_{\theta, \boldsymbol y}^{\boldsymbol\gamma, m:k}(\Lambda' ) \big\| }
	{L_{\boldsymbol\gamma, \boldsymbol y}^{m:k}(\Lambda,\Lambda') }
	\leq
	\tau^{2(k-m) }
	A_{\boldsymbol\gamma} K_{\boldsymbol\gamma}(\Lambda,\Lambda')
\end{align}
for $k\geq m\geq 0$.
Moreover, since
$|\boldsymbol\gamma+\boldsymbol\delta|=|\boldsymbol\gamma|+|\boldsymbol\delta|$ and
$M_{\boldsymbol\gamma}(\Lambda)\geq 1$,
(\ref{4.5}) yields
\begin{align}
	&\label{p1.3.705}
	M_{\boldsymbol\gamma, \boldsymbol y}^{m:n-1}(\Lambda)
	\leq
	M_{\boldsymbol\gamma, \boldsymbol y}^{m:n}(\Lambda)
	\leq
	\frac{M_{\boldsymbol\gamma+\boldsymbol\delta, \boldsymbol y}^{m:n}(\Lambda) }
	{\left(\psi(y_{n} ) \right)^{|\boldsymbol\delta|} },
	\;\;\;\;\;
	M_{\boldsymbol\gamma, \boldsymbol y}^{m:n}(\Lambda)
	M_{\boldsymbol\delta, \boldsymbol y}^{m:n}(\Lambda)
	\leq
	M_{\boldsymbol\gamma+\boldsymbol\delta,\boldsymbol y}^{m:n}(\Lambda ).
\end{align}
Similarly, (\ref{4.5}) leads to
\begin{align}
	&\label{p1.3.709}
	L_{\boldsymbol\gamma,\boldsymbol y}^{m:n-1}(\Lambda,\Lambda')
	\leq
	L_{\boldsymbol\gamma,\boldsymbol y}^{m:n}(\Lambda,\Lambda')
	\leq
	\frac{L_{\boldsymbol\gamma+\boldsymbol\delta,\boldsymbol y}^{m:n}(\Lambda,\Lambda') }
	{\left( \psi(y_{n} ) (n-m) \right)^{|\boldsymbol\delta|} },
	\;\;\;\;\;
	L_{\boldsymbol\gamma, \boldsymbol y}^{m:n}(\Lambda,\Lambda')
	L_{\boldsymbol\delta, \boldsymbol y}^{m:n}(\Lambda,\Lambda')
	\leq
	L_{\boldsymbol\gamma+\boldsymbol\delta,\boldsymbol y}^{m:n}(\Lambda,\Lambda').
\end{align}
The same arguments also imply
\begin{align}\label{p1.3.721}
	M_{\boldsymbol\delta, \boldsymbol y}^{m:n}(\Lambda)
	+
	M_{\boldsymbol\delta, \boldsymbol y}^{m:n}(\Lambda')
	\leq
	\frac{L_{\boldsymbol\delta, \boldsymbol y}^{m:n}(\Lambda,\Lambda')}
	{(n-m)^{|\boldsymbol\delta|} }.
\end{align}

Using (\ref{p1.3.21}), (\ref{p1.3.705}), we conclude
\begin{align}\label{p1.3.301}
	\left\|F_{\theta, \boldsymbol y}^{\boldsymbol\gamma, m:n-1}(\Lambda ) \right\|
	\leq
	A_{\boldsymbol\gamma} M_{\boldsymbol\gamma,\boldsymbol y}^{m:n}(\Lambda)
	\leq
	\frac{A_{\boldsymbol\gamma } M_{\boldsymbol\gamma+\boldsymbol\delta,\boldsymbol y}^{m:n}(\Lambda) }
	{\left(\psi(y_{n} ) \right)^{|\boldsymbol\delta| } }.
\end{align}
Then, Lemma \ref{lemma1.4} and (\ref{p1.3.701}), (\ref{p1.3.301}) imply
\begin{align}\label{p1.3.23}
	\left\|T_{\theta, y_{n} }^{\boldsymbol\alpha, \boldsymbol\beta }
	(F_{\theta, \boldsymbol y}^{m:n-1}(\Lambda ) ) \right\|
	\leq 
	C_{2}
	\left(\psi(y_{n} ) \right)^{|\boldsymbol\alpha-\boldsymbol\beta | }
	\left\|F_{\theta, \boldsymbol y}^{\boldsymbol\beta, m:n-1}(\Lambda ) \right\|
	\leq &
	C_{2} A_{\boldsymbol\beta } M_{\boldsymbol\alpha,\boldsymbol y}^{m:n}(\Lambda)
	\leq 
	\frac{A_{\boldsymbol\alpha } M_{\boldsymbol\alpha,\boldsymbol y}^{m:n}(\Lambda) }
	{2\tilde{C}_{1} }
\end{align}
(as $C_{2}\leq 4\tilde{C}_{1}$).
The same lemma and (\ref{p1.3.301}) yield
\begin{align}\label{p1.3.25}
	\left\|S_{\theta, y_{n} }^{\boldsymbol\gamma }
	(F_{\theta, \boldsymbol y}^{m:n-1}(\Lambda ) ) \right\|
	\leq 
	C_{2}
	\sum_{\stackrel{\scriptstyle \boldsymbol\delta \in \mathbb{N}_{0}^{d} }
	{\boldsymbol\delta \leq \boldsymbol\gamma} }
	\left(\psi(y_{n} ) \right)^{|\boldsymbol\gamma-\boldsymbol\delta | }
	\left\|F_{\theta, \boldsymbol y}^{\boldsymbol\delta, m:n-1}(\Lambda ) \right\|
	\leq 
	2^{|\boldsymbol\gamma|} C_{2} A_{\boldsymbol\gamma} M_{\boldsymbol\gamma,\boldsymbol y}^{m:n}(\Lambda)
	\leq 
	\tilde{C}_{1} A_{\boldsymbol\gamma} M_{\boldsymbol\gamma,\boldsymbol y}^{m:n}(\Lambda)
\end{align}
(since $2^{|\boldsymbol\gamma|}C_{2}\leq 2^{p}C_{2}\leq\tilde{C}_{1}/2$).
If $\boldsymbol\beta\neq\boldsymbol 0$,
(\ref{p1.3.701}), (\ref{p1.3.21}), (\ref{p1.3.705}), (\ref{p1.3.25}) lead to
\begin{align}\label{p1.3.29}
	\left\|F_{\theta, \boldsymbol y}^{\boldsymbol\beta, m:n}(\Lambda ) \:
	\big\langle S_{\theta, y_{n} }^{\boldsymbol\alpha - \boldsymbol\beta }
	(F_{\theta, \boldsymbol y}^{m:n-1}(\Lambda ) )
	\big\rangle \right\|
	\leq &
	\left\|F_{\theta, \boldsymbol y}^{\boldsymbol\beta, m:n}(\Lambda ) \right\|
	\left\|S_{\theta, y_{n} }^{\boldsymbol\alpha - \boldsymbol\beta }
	(F_{\theta, \boldsymbol y}^{m:n-1}(\Lambda ) ) \right\|
	\nonumber\\
	\leq &
	\tilde{C}_{1} A_{\boldsymbol\beta} A_{\boldsymbol\alpha-\boldsymbol\beta}
	M_{\boldsymbol\beta,\boldsymbol y}^{m:n}(\Lambda)
	M_{\boldsymbol\alpha-\boldsymbol\beta,\boldsymbol y}^{m:n}(\Lambda)
	\nonumber\\
	\leq &
	\frac{A_{\boldsymbol\alpha} M_{\boldsymbol\alpha,\boldsymbol y}^{m:n}(\Lambda) }
	{2\tilde{C}_{1} }.
\end{align}
Consequently, (\ref{1.705}), (\ref{4.105})
(\ref{4.121}), (\ref{p1.3.23}) imply
\begin{align}\label{p1.3.33}
	\left\|W_{\theta, \boldsymbol y }^{\boldsymbol\alpha, m:n}(\Lambda ) \right\|
	\leq &
	\sum_{\stackrel{\scriptstyle \boldsymbol\beta \in \mathbb{N}_{0}^{d}
	\setminus\{\boldsymbol\alpha \} }
	{\boldsymbol\beta \leq \boldsymbol\alpha } }
	\left(\boldsymbol\alpha \atop \boldsymbol\beta \right)
	\left\|T_{\theta, y_{n} }^{\boldsymbol\alpha, \boldsymbol\beta }
	(F_{\theta, \boldsymbol y}^{m:n-1}(\Lambda ) ) \right\|
	+\!\!\!
	\sum_{\stackrel{\scriptstyle \boldsymbol\beta \in \mathbb{N}_{0}^{d}
	\setminus\{\boldsymbol 0, \boldsymbol\alpha \} }
	{\boldsymbol\beta \leq \boldsymbol\alpha } }
	\left(\boldsymbol\alpha \atop \boldsymbol\beta \right)
	\left\|F_{\theta, \boldsymbol y}^{\boldsymbol\beta, m:n}(\Lambda ) \:
	\big\langle S_{\theta, y_{n} }^{\boldsymbol\alpha - \boldsymbol\beta }
	(F_{\theta, \boldsymbol y}^{m:n-1}(\Lambda ) )
	\big\rangle \right\|
	\nonumber\\
	\leq &
	\frac{2^{|\boldsymbol\alpha|} A_{\boldsymbol\alpha } M_{\boldsymbol\alpha,\boldsymbol y}^{m:n}(\Lambda) }
	{\tilde{C}_{1} }
	\nonumber\\
	\leq &
	\frac{A_{\boldsymbol\alpha } M_{\boldsymbol\alpha,\boldsymbol y}^{m:n}(\Lambda) }
	{\tilde{C}_{2} }
\end{align}
(as $\tilde{C}_{1}/2^{|\boldsymbol\alpha|}\geq\tilde{C}_{1}/2^{p}\geq\tilde{C}_{2}$).
Then, owing to Proposition \ref{proposition1.1}, we have
\begin{align}\label{p1.3.1}
	\left\|
	G_{\theta, \boldsymbol y}^{k:n}
	\left(F_{\theta, \boldsymbol y }^{\boldsymbol 0, m:k}(\Lambda ),
	W_{\theta, \boldsymbol y }^{\boldsymbol\alpha, m:k}(\Lambda ) \right)
	\right\|
	\leq 
	C_{3} \tau^{2(n-k)}
	\left\|
	W_{\theta, \boldsymbol y }^{\boldsymbol\alpha, m:k}(\Lambda )
	\right\|
	\leq &
	\frac{C_{3} \tau^{2(n-k)} A_{\boldsymbol\alpha } M_{\boldsymbol\alpha,\boldsymbol y}^{m:k}(\Lambda) }
	{\tilde{C}_{2} }
	\nonumber\\
	\leq &
	\frac{\tau^{2(n-k)} A_{\boldsymbol\alpha } M_{\boldsymbol\alpha,\boldsymbol y}^{m:n}(\Lambda) }
	{\tilde{C}_{3} }
\end{align}
for $n\geq k>m$
(since $C_{3}/\tilde{C}_{2}\leq 1/\tilde{C}_{3}$).
Due to the same proposition and (\ref{4.121}), we have
\begin{align}\label{p1.3.3}
	\left\|V_{\theta, \boldsymbol y}^{\boldsymbol\alpha, m:n}(\Lambda ) \right\|
	\leq
	C_{3} \tau^{2(n-m)} \|\lambda_{\boldsymbol\alpha } \|
	\leq
	C_{3} \tau^{2(n-m)} M_{\boldsymbol\alpha}(\Lambda)
	\leq
	\frac{\tau^{2(n-m)} A_{\boldsymbol\alpha } M_{\boldsymbol\alpha,\boldsymbol y}^{m:n}(\Lambda) }
	{\tilde{C}_{3} }
\end{align}
(as $A_{\boldsymbol\alpha }\geq\tilde{C}_{1}^{2}\geq C_{3}\tilde{C}_{3}$).
Combining Lemma \ref{lemma1.12} and (\ref{p1.3.1}), (\ref{p1.3.3}), we get
\begin{align}\label{p1.3.35}
	\left\|F_{\theta, \boldsymbol y }^{\boldsymbol\alpha, m:n}(\Lambda ) \right\|
	\leq 
	\left\|V_{\theta, \boldsymbol y}^{m:n}(\Lambda ) \right\|
	+
	\sum_{k=m+1}^{n}
	\left\|
	G_{\theta, \boldsymbol y}^{k:n}
	\left(F_{\theta, \boldsymbol y }^{\boldsymbol 0, m:k}(\Lambda ),
	W_{\theta, \boldsymbol y }^{\boldsymbol\alpha, m:k}(\Lambda ) \right)
	\right\|
	\leq &
	\frac{A_{\boldsymbol\alpha} M_{\boldsymbol\alpha,\boldsymbol y}^{m:n}(\Lambda) }
	{\tilde{C}_{3} }
	\sum_{k=m}^{n} \tau^{2(n-k)}
	\nonumber\\
	\leq &
	\frac{A_{\boldsymbol\alpha} M_{\boldsymbol\alpha,\boldsymbol y}^{m:n}(\Lambda) }
	{\tilde{C}_{3} (1-\tau^{2} ) }
	\nonumber\\
	\leq &
	A_{\boldsymbol\alpha }
	M_{\boldsymbol\alpha,\boldsymbol y}^{m:n}(\Lambda)
\end{align}
(since $\tilde{C}_{3}(1-\tau^{2} )\geq 1$).
Hence, (\ref{p1.3.1*}) holds for
$\boldsymbol\alpha\in\mathbb{N}_{0}^{d}$, $|\boldsymbol\alpha|=l+1$.

Now, (\ref{p1.3.3*}) is proved.
Relying on (\ref{p1.3.43}), (\ref{p1.3.709}), we deduce
\begin{align}\label{p1.3.303}
	\left\|
	F_{\theta, \boldsymbol y}^{\boldsymbol\gamma, m:n-1}(\Lambda )
	-
	F_{\theta, \boldsymbol y}^{\boldsymbol\gamma, m:n-1}(\Lambda' )
	\right\|
	\leq &
	\tau^{2(n-m-1)} A_{\boldsymbol\gamma} K_{\boldsymbol\gamma}(\Lambda,\Lambda')
	L_{\boldsymbol\gamma,\boldsymbol y}^{m:n}(\Lambda,\Lambda')
	\nonumber\\
	\leq &
	\frac{\tau^{2(n-m-1)} A_{\boldsymbol\gamma} K_{\boldsymbol\gamma+\boldsymbol\delta}(\Lambda,\Lambda')
	L_{\boldsymbol\gamma+\boldsymbol\delta,\boldsymbol y}^{m:n}(\Lambda,\Lambda') }
	{\left(\psi(y_{n} ) (n-m) \right)^{|\boldsymbol\delta| } }.
\end{align}
Similarly, using Proposition \ref{proposition1.2} and
(\ref{p1.3.709}), (\ref{p1.3.301}), we conclude
\begin{align}\label{p1.3.305}
	\left\|
	F_{\theta, \boldsymbol y}^{\boldsymbol 0, m:n-1}(\Lambda )
	-
	F_{\theta, \boldsymbol y}^{\boldsymbol 0, m:n-1}(\Lambda' )
	\right\|
	\left\|
	F_{\theta, \boldsymbol y}^{\boldsymbol\gamma, m:n-1}(\Lambda' )
	\right\|
	\leq &
	C_{4} \tau^{2(n-m-1)} A_{\boldsymbol\gamma}
	K_{\boldsymbol 0}(\Lambda,\Lambda')
	M_{\boldsymbol\gamma,\boldsymbol y}^{m:n}(\Lambda')
	\nonumber\\
	\leq &
	\frac{C_{4} \tau^{2(n-m-1)} A_{\boldsymbol\gamma}
	K_{\boldsymbol\gamma+\boldsymbol\delta}(\Lambda,\Lambda')
	L_{\boldsymbol\gamma+\boldsymbol\delta,\boldsymbol y}^{m:n}(\Lambda,\Lambda') }
	{\left(\psi(y_{n} ) (n-m) \right)^{|\boldsymbol\delta|} }
\end{align}
(since
$M_{\boldsymbol\gamma,\boldsymbol y}^{m:n}(\Lambda')\leq
L_{\boldsymbol\gamma,\boldsymbol y}^{m:n}(\Lambda,\Lambda')$).
Then, Lemma \ref{lemma1.4} and (\ref{p1.3.701}) imply
\begin{align}\label{p1.3.45}
	&
	\begin{aligned}
	&
	\left\|
	T_{\theta, y_{n} }^{\boldsymbol\alpha, \boldsymbol\beta }
	(F_{\theta, \boldsymbol y}^{m:n-1}(\Lambda ) )
	-
	T_{\theta, y_{n} }^{\boldsymbol\alpha, \boldsymbol\beta }
	(F_{\theta, \boldsymbol y}^{m:n-1}(\Lambda' ) )
	\right\|
	\leq 
	C_{2}
	\left(\psi(y_{n} ) \right)^{|\boldsymbol\alpha-\boldsymbol\beta | }
	\left\|
	F_{\theta, \boldsymbol y}^{\boldsymbol\beta, m:n-1}(\Lambda )
	-
	F_{\theta, \boldsymbol y}^{\boldsymbol\beta, m:n-1}(\Lambda' )
	\right\|&
	\\
	&\hspace{1em}
	+
	C_{2}
	\left(\psi(y_{n} ) \right)^{|\boldsymbol\alpha-\boldsymbol\beta | }
	\left\|
	F_{\theta, \boldsymbol y}^{\boldsymbol 0, m:n-1}(\Lambda )
	-
	F_{\theta, \boldsymbol y}^{\boldsymbol 0, m:n-1}(\Lambda' )
	\right\|
	\left\|
	F_{\theta, \boldsymbol y}^{\boldsymbol\beta, m:n-1}(\Lambda' )
	\right\|
	\end{aligned}
	\nonumber\\
	&\leq
	\frac{2C_{2} C_{4} \tau^{2(n-m-1)}
	A_{\boldsymbol\beta}
	K_{\boldsymbol\alpha}(\Lambda,\Lambda' )
	L_{\boldsymbol\alpha,\boldsymbol y}^{m:n}(\Lambda,\Lambda') }
	{(n-m)^{|\boldsymbol\alpha-\boldsymbol\beta|} }
	\nonumber\\
	&\leq
	\frac{\tau^{2(n-m)} A_{\boldsymbol\alpha}
	K_{\boldsymbol\alpha}(\Lambda,\Lambda')
	L_{\boldsymbol\alpha,\boldsymbol y}^{m:n}(\Lambda,\Lambda') }
	{4\tilde{C}_{1} (n-m) }
\end{align}
(as $|\boldsymbol\alpha-\boldsymbol\beta|\geq 1$, $C_{2}C_{4}\leq\tilde{C}_{1}\tau^{2}$).
The same lemma and (\ref{p1.3.303}), (\ref{p1.3.305}) yield
\begin{align}\label{p1.3.47}
	&
	\begin{aligned}
	&
	\left\|
	S_{\theta, y_{n} }^{\boldsymbol\gamma }
	(F_{\theta, \boldsymbol y}^{m:n-1}(\Lambda ) )
	-
	S_{\theta, y_{n} }^{\boldsymbol\gamma }
	(F_{\theta, \boldsymbol y}^{m:n-1}(\Lambda' ) )
	\right\|
	\leq 
	C_{2} 
	\sum_{\stackrel{\scriptstyle \boldsymbol\delta \in \mathbb{N}_{0}^{d} }
	{\boldsymbol\delta \leq \boldsymbol\gamma} } \!
	\left(\psi(y_{n} ) \right)^{|\boldsymbol\gamma-\boldsymbol\delta | } \!
	\left\|
	F_{\theta, \boldsymbol y}^{\boldsymbol\delta, m:n-1}(\Lambda )
	-
	F_{\theta, \boldsymbol y}^{\boldsymbol\delta, m:n-1}(\Lambda' )
	\right\|
	\\
	&\hspace{1em}
	+
	C_{2} 
	\sum_{\stackrel{\scriptstyle \boldsymbol\delta \in \mathbb{N}_{0}^{d} }
	{\boldsymbol\delta \leq \boldsymbol\gamma} }
	\left(\psi(y_{n} ) \right)^{|\boldsymbol\gamma-\boldsymbol\delta | }
	\left\|
	F_{\theta, \boldsymbol y}^{\boldsymbol 0, m:n-1}(\Lambda )
	-
	F_{\theta, \boldsymbol y}^{\boldsymbol 0, m:n-1}(\Lambda' )
	\right\|
	\left\|
	F_{\theta, \boldsymbol y}^{\boldsymbol\delta, m:n-1}(\Lambda' )
	\right\|
	\end{aligned}
	\nonumber\\
	&
	\leq
	4^{|\boldsymbol\gamma|}C_{2}C_{4}\tau^{2(n-m-1)}
	A_{\boldsymbol\gamma}
	K_{\boldsymbol\gamma}(\Lambda,\Lambda' )
	L_{\boldsymbol\gamma,\boldsymbol y}^{m:n}(\Lambda,\Lambda')
	\nonumber\\
	&
	\leq
	\tilde{C}_{1}\tau^{2(n-m)}
	A_{\boldsymbol\gamma }
	K_{\boldsymbol\gamma}(\Lambda,\Lambda' )
	L_{\boldsymbol\gamma,\boldsymbol y}^{m:n}(\Lambda,\Lambda')
\end{align}
(since $4^{|\boldsymbol\gamma|}C_{2}C_{4}\leq\tilde{C}_{1}\tau^{2}$).
Then, (\ref{p1.3.21}), (\ref{p1.3.43}), (\ref{p1.3.25}) lead to
\begin{align*}
	&
	\left\|
	F_{\theta, \boldsymbol y}^{\boldsymbol\beta, m:n}(\Lambda )
	\:
	\big\langle S_{\theta, y_{n} }^{\boldsymbol\alpha - \boldsymbol\beta }
	(F_{\theta, \boldsymbol y}^{m:n-1}(\Lambda ) )
	\big\rangle
	-
	F_{\theta, \boldsymbol y}^{\boldsymbol\beta, m:n}(\Lambda' )
	\:
	\big\langle S_{\theta, y_{n} }^{\boldsymbol\alpha - \boldsymbol\beta }
	(F_{\theta, \boldsymbol y}^{m:n-1}(\Lambda' ) )
	\big\rangle
	\right\|
	\nonumber\\
	&
	\begin{aligned}[b]
	\leq &
	\left\|F_{\theta, \boldsymbol y}^{\boldsymbol\beta, m:n}(\Lambda ) \right\|
	\left\|
	S_{\theta, y_{n} }^{\boldsymbol\alpha - \boldsymbol\beta }
	(F_{\theta, \boldsymbol y}^{m:n-1}(\Lambda ) )
	-
	S_{\theta, y_{n} }^{\boldsymbol\alpha - \boldsymbol\beta }
	(F_{\theta, \boldsymbol y}^{m:n-1}(\Lambda' ) )
	\right\|
	\\
	&
	+
	\left\|
	F_{\theta, \boldsymbol y}^{\boldsymbol\beta, m:n}(\Lambda )
	-
	F_{\theta, \boldsymbol y}^{\boldsymbol\beta, m:n}(\Lambda' )
	\right\|
	\left\|
	S_{\theta, y_{n} }^{\boldsymbol\alpha - \boldsymbol\beta }
	(F_{\theta, \boldsymbol y}^{m:n-1}(\Lambda' ) )
	\right\|
	\end{aligned}
	\nonumber\\
	&
	\begin{aligned}[b]
	\leq &
	\tilde{C}_{1}\tau^{2(n-m)}
	A_{\boldsymbol\beta} A_{\boldsymbol\alpha-\boldsymbol\beta}
	K_{\boldsymbol\alpha-\boldsymbol\beta}(\Lambda,\Lambda')
	L_{\boldsymbol\alpha-\boldsymbol\beta,\boldsymbol y}^{m:n}(\Lambda,\Lambda')
	M_{\boldsymbol\beta,\boldsymbol y}^{m:n}(\Lambda)
	\\
	&+
	\tilde{C}_{1}\tau^{2(n-m)}
	A_{\boldsymbol\beta} A_{\boldsymbol\alpha-\boldsymbol\beta}
	K_{\boldsymbol\beta}(\Lambda,\Lambda')
	L_{\boldsymbol\beta,\boldsymbol y}^{m:n}(\Lambda,\Lambda')
	M_{\boldsymbol\alpha-\boldsymbol\beta,\boldsymbol y}^{m:n}(\Lambda').
	\end{aligned}
\end{align*}
Hence, if $\boldsymbol\beta\neq\boldsymbol 0$,
(\ref{p1.3.701}), (\ref{p1.3.709}), (\ref{p1.3.721}) imply
\begin{align}\label{p1.3.51}
	&
	\left\|
	F_{\theta, \boldsymbol y}^{\boldsymbol\beta, m:n}(\Lambda )
	\:
	\big\langle S_{\theta, y_{n} }^{\boldsymbol\alpha - \boldsymbol\beta }
	(F_{\theta, \boldsymbol y}^{m:n-1}(\Lambda ) )
	\big\rangle
	-
	F_{\theta, \boldsymbol y}^{\boldsymbol\beta, m:n}(\Lambda' )
	\:
	\big\langle S_{\theta, y_{n} }^{\boldsymbol\alpha - \boldsymbol\beta }
	(F_{\theta, \boldsymbol y}^{m:n-1}(\Lambda' ) )
	\big\rangle
	\right\|
	\nonumber\\
	&
	\leq
	\frac{\tau^{2(n-m)} A_{\boldsymbol\alpha}
	K_{\boldsymbol\alpha}(\Lambda,\Lambda' )
	L_{\boldsymbol\alpha,\boldsymbol y}^{m:n}(\Lambda,\Lambda') }
	{4\tilde{C}_{1} (n-m) }
\end{align}
(as $\boldsymbol\beta\neq\boldsymbol 0$, $\boldsymbol\alpha-\boldsymbol\beta\neq\boldsymbol 0$).
Consequently, (\ref{1.705}), (\ref{4.105}), (\ref{4.121}), (\ref{p1.3.45}) yield
\begin{align}\label{p1.3.55}
	&
	\begin{aligned}
	&
	\left\|
	W_{\theta, \boldsymbol y }^{\boldsymbol\alpha, m:n}(\Lambda )
	-
	W_{\theta, \boldsymbol y }^{\boldsymbol\alpha, m:n}(\Lambda' )
	\right\|
	\leq 
	\sum_{\stackrel{\scriptstyle \boldsymbol\beta \in \mathbb{N}_{0}^{d}
	\setminus\{\boldsymbol\alpha \} }
	{\boldsymbol\beta \leq \boldsymbol\alpha } }
	\left(\boldsymbol\alpha \atop \boldsymbol\beta \right)
	\left\|
	T_{\theta, y_{n} }^{\boldsymbol\alpha, \boldsymbol\beta }
	(F_{\theta, \boldsymbol y}^{m:n-1}(\Lambda ) )
	-
	T_{\theta, y_{n} }^{\boldsymbol\alpha, \boldsymbol\beta }
	(F_{\theta, \boldsymbol y}^{m:n-1}(\Lambda' ) )
	\right\|
	\\
	&\hspace{1em}
	+
	\sum_{\stackrel{\scriptstyle \boldsymbol\beta \in \mathbb{N}_{0}^{d}
	\setminus\{\boldsymbol 0, \boldsymbol\alpha \} }
	{\boldsymbol\beta \leq \boldsymbol\alpha } }
	\left(\boldsymbol\alpha \atop \boldsymbol\beta \right)
	\left\|
	F_{\theta, \boldsymbol y}^{\boldsymbol\beta, m:n}(\Lambda ) \:
	\big\langle S_{\theta, y_{n} }^{\boldsymbol\alpha - \boldsymbol\beta }
	(F_{\theta, \boldsymbol y}^{m:n-1}(\Lambda ) )
	\big\rangle
	-
	F_{\theta, \boldsymbol y}^{\boldsymbol\beta, m:n}(\Lambda' ) \:
	\big\langle S_{\theta, y_{n} }^{\boldsymbol\alpha - \boldsymbol\beta }
	(F_{\theta, \boldsymbol y}^{m:n-1}(\Lambda' ) )
	\big\rangle
	\right\|
	\end{aligned}
	\nonumber\\
	&
	\leq
	\frac{4^{|\boldsymbol\alpha|} \tau^{2(n-m)} A_{\boldsymbol\alpha }
	K_{\boldsymbol\alpha}(\Lambda,\Lambda' )
	L_{\boldsymbol\alpha,\boldsymbol y}^{m:n}(\Lambda,\Lambda') }
	{2\tilde{C}_{1} (n-m) }
	\nonumber\\
	&
	\leq
	\frac{\tau^{2(n-m)} A_{\boldsymbol\alpha }
	K_{\boldsymbol\alpha}(\Lambda,\Lambda' )
	L_{\boldsymbol\alpha,\boldsymbol y}^{m:n}(\Lambda,\Lambda') }
	{2\tilde{C}_{2} (n-m) }
\end{align}
(since $\tilde{C}_{1}/4^{|\boldsymbol\alpha|}\geq\tilde{C}_{1}/4^{p}=\tilde{C}_{2}$).
Then, owing to Propositions \ref{proposition1.1}, \ref{proposition1.2}
and (\ref{p1.3.721}), (\ref{p1.3.33}),
we have
\begin{align}\label{p1.3.5}
	&
	\left\|
	G_{\theta, \boldsymbol y}^{k:n}
	\left(F_{\theta, \boldsymbol y }^{\boldsymbol 0, m:k}(\Lambda ),
	W_{\theta, \boldsymbol y }^{\boldsymbol\alpha, m:k}(\Lambda ) \right)
	-
	G_{\theta, \boldsymbol y}^{k:n}
	\left(F_{\theta, \boldsymbol y }^{\boldsymbol 0, m:k}(\Lambda' ),
	W_{\theta, \boldsymbol y }^{\boldsymbol\alpha, m:k}(\Lambda' ) \right)
	\right\|
	\nonumber\\
	&
	\leq
	C_{3} \tau^{2(n-k)}
	\left\|
	W_{\theta, \boldsymbol y }^{\boldsymbol\alpha, m:k}(\Lambda )
	-
	W_{\theta, \boldsymbol y }^{\boldsymbol\alpha, m:k}(\Lambda' )
	\right\|
	+
	C_{3} \tau^{2(n-k)}
	\left\|
	F_{\theta, \boldsymbol y }^{\boldsymbol 0, m:k}(\Lambda )
	-
	F_{\theta, \boldsymbol y }^{\boldsymbol 0, m:k}(\Lambda' )
	\right\|
	\left\|
	W_{\theta, \boldsymbol y }^{\boldsymbol\alpha, m:k}(\Lambda' )
	\right\|
	\nonumber\\
	&
	\leq
	\frac{C_{3} \tau^{2(n-m)} A_{\boldsymbol\alpha}
	K_{\boldsymbol\alpha}(\Lambda,\Lambda')
	L_{\boldsymbol\alpha,\boldsymbol y}^{m:k}(\Lambda,\Lambda') }
	{2\tilde{C}_{2} (n-m) }
	+
	\frac{C_{3}C_{4} \tau^{2(n-m)} A_{\boldsymbol\alpha}
	K_{\boldsymbol 0}(\Lambda,\Lambda')
	M_{\boldsymbol\alpha,\boldsymbol y}^{m:k}(\Lambda') }
	{2\tilde{C}_{2} }
	\nonumber\\
	&\leq
	\frac{\tau^{2(n-m)} A_{\boldsymbol\alpha}
	K_{\boldsymbol\alpha}(\Lambda,\Lambda')
	L_{\boldsymbol\alpha,\boldsymbol y}^{m:n}(\Lambda,\Lambda') }
	{\tilde{C}_{3} (n-m) }
\end{align}
for $n\geq k> m$
(as $C_{3}/\tilde{C}_{2}\leq C_{3}C_{4}/\tilde{C}_{2} = 1/\tilde{C}_{3}$).
Due to the same propositions and (\ref{4.121}), we have
\begin{align}\label{p1.3.5'}
	\left\|
	V_{\theta, \boldsymbol y}^{\alpha,m:n}(\Lambda)
	-
	V_{\theta, \boldsymbol y}^{\alpha,m:n}(\Lambda')
	\right\|
	\leq &
	C_{3} \tau^{2(n-m)}
	\left(
	\|\lambda_{\boldsymbol\alpha } - \lambda'_{\boldsymbol\alpha } \|
	+
	\|\lambda_{\boldsymbol 0} - \lambda'_{\boldsymbol 0} \|
	\|\lambda'_{\boldsymbol\alpha } \|
	\right).
\end{align}
Moreover, we have
\begin{align*}
	&
	\|\lambda_{\boldsymbol\alpha } - \lambda'_{\boldsymbol\alpha } \|
	+
	\|\lambda_{\boldsymbol 0} - \lambda'_{\boldsymbol 0} \|
	\|\lambda'_{\boldsymbol\alpha } \|
	\leq
	M_{\boldsymbol\alpha}(\Lambda-\Lambda')
	+
	M_{\boldsymbol\alpha}(\Lambda-\Lambda') M_{\boldsymbol\alpha}(\Lambda')
	\leq
	2M_{\boldsymbol\alpha}(\Lambda-\Lambda') L_{\boldsymbol\alpha}(\Lambda,\Lambda'),
	\\
	&
	\|\lambda_{\boldsymbol\alpha } - \lambda'_{\boldsymbol\alpha } \|
	+
	\|\lambda_{\boldsymbol 0} - \lambda'_{\boldsymbol 0} \|
	\|\lambda'_{\boldsymbol\alpha } \|
	\leq
	\|\lambda_{\boldsymbol\alpha } \| + 2\|\lambda'_{\boldsymbol\alpha } \|
	\leq
	3L_{\boldsymbol\alpha}(\Lambda,\Lambda')
\end{align*}
(since $M_{\boldsymbol\alpha}(\Lambda')\geq \|\lambda'_{\boldsymbol 0} \| = 1$).
Hence, we get
\begin{align*}
	\|\lambda_{\boldsymbol\alpha } - \lambda'_{\boldsymbol\alpha } \|
	+
	\|\lambda_{\boldsymbol 0} - \lambda'_{\boldsymbol 0} \|
	\|\lambda'_{\boldsymbol\alpha } \|
	\leq
	3\min\left\{1, M_{\boldsymbol\alpha}(\Lambda-\Lambda') \right\}
	L_{\boldsymbol\alpha}(\Lambda,\Lambda')
	=
	3 K_{\boldsymbol\alpha}(\Lambda,\Lambda')
	L_{\boldsymbol\alpha}(\Lambda,\Lambda').
\end{align*}
Therefore, (\ref{p1.3.5'}) implies
\begin{align}\label{p1.3.7}
	\left\|
	V_{\theta, \boldsymbol y}^{\alpha,m:n}(\Lambda)
	-
	V_{\theta, \boldsymbol y}^{\alpha,m:n}(\Lambda')
	\right\|
	\leq&
	3C_{3}\tau^{2(n-m)} K_{\boldsymbol\alpha}(\Lambda,\Lambda')
	L_{\boldsymbol\alpha}(\Lambda,\Lambda')
	\nonumber\\
	\leq &
	\frac{\tau^{2(n-m)} A_{\boldsymbol\alpha} K_{\boldsymbol\alpha}(\Lambda,\Lambda')
	L_{\boldsymbol\alpha,\boldsymbol y}^{m:n}(\Lambda,\Lambda' ) }
	{\tilde{C}_{3} }
\end{align}
(as $A_{\boldsymbol\alpha }\geq\tilde{C}_{1}^{2}\geq 3C_{3}\tilde{C}_{3}$).
Combining Lemma \ref{lemma1.12} and (\ref{p1.3.5}), (\ref{p1.3.7}), we get
\begin{align}\label{p1.3.57}
	&
	\begin{aligned}
	&
	\left\|
	F_{\theta, \boldsymbol y }^{\boldsymbol\alpha, m:n}(\Lambda )
	-
	F_{\theta, \boldsymbol y }^{\boldsymbol\alpha, m:n}(\Lambda' )
	\right\|
	\leq 
	\left\|
	V_{\theta, \boldsymbol y}^{\alpha,m:n}(\Lambda)
	-
	V_{\theta, \boldsymbol y}^{\alpha,m:n}(\Lambda')
	\right\|
	\\
	&\hspace{1em}
	+
	\sum_{k=m+1}^{n}
	\left\|
	G_{\theta, \boldsymbol y}^{k:n}
	\left(F_{\theta, \boldsymbol y }^{\boldsymbol 0, m:k}(\Lambda ),
	W_{\theta, \boldsymbol y }^{\boldsymbol\alpha, m:k}(\Lambda ) \right)
	-
	G_{\theta, \boldsymbol y}^{k:n}
	\left(F_{\theta, \boldsymbol y }^{\boldsymbol 0, m:k}(\Lambda' ),
	W_{\theta, \boldsymbol y }^{\boldsymbol\alpha, m:k}(\Lambda' ) \right)
	\right\|
	\end{aligned}
	\nonumber\\
	&
	\leq
	\frac{2\tau^{2(n-m)} A_{\boldsymbol\alpha} K_{\boldsymbol\alpha}(\Lambda,\Lambda')
	L_{\boldsymbol\alpha,\boldsymbol y}^{m:n}(\Lambda,\Lambda' ) }
	{\tilde{C}_{3} }
	\nonumber\\
	&\leq
	\tau^{2(n-m)} A_{\boldsymbol\alpha} K_{\boldsymbol\alpha}(\Lambda,\Lambda')
	L_{\boldsymbol\alpha,\boldsymbol y}^{m:n}(\Lambda,\Lambda' )
\end{align}
(since $\tilde{C}_{3}\geq 2$).
Hence, (\ref{p1.3.3*}) holds for
$\boldsymbol\alpha\in\mathbb{N}_{0}^{d}$, $|\boldsymbol\alpha|=l+1$.
Then, the proposition directly follows by the principle of mathematical induction.
\end{proof}

\begin{proof}[\rm\bf Proof of Theorem \ref{theorem1.2}]
Let $\tilde{C}_{1}$, $\tilde{C}_{2}$ be the real numbers defined by
$\tilde{C}_{1}=\max_{n\geq 1} \tau^{n-1}n^{p}$,
$\tilde{C}_{2}=\max\{A_{\boldsymbol\alpha}:
\boldsymbol\alpha\in\mathbb{N}_{0}^{d}, |\boldsymbol\alpha|\leq p \}$,
while $K$ is the real number defined by $K=\tilde{C}_{1}\tilde{C}_{2}$
($A_{\boldsymbol\alpha}$ is specified in Proposition \ref{proposition1.3},
while $\tau$ is defined at the beginning of Section \ref{section1.1*}).
Then, Proposition \ref{proposition1.3} implies
\begin{align}\label{t1.2.1}
	\left\| F_{\theta,\boldsymbol y}^{\boldsymbol\alpha,m:n}(\Lambda)
	-
	F_{\theta,\boldsymbol y}^{\boldsymbol\alpha,m:n}(\Lambda') \right\|
	\leq &
	\tilde{C}_{2}\tau^{2(n-m)}(n-m)^{p} \|\Lambda-\Lambda'\|
	\left( (\|\Lambda\|+\|\Lambda'\| ) \Psi_{\boldsymbol y}^{m:n} \right)^{p}
	\nonumber\\
	\leq &
	\tilde{C}_{1}\tilde{C}_{2}\tau^{n-m} \|\Lambda-\Lambda'\|
	\left( (\|\Lambda\|+\|\Lambda'\| ) \Psi_{\boldsymbol y}^{m:n} \right)^{p}
	\nonumber\\
	\leq &
	K\tau^{n-m} \|\Lambda-\Lambda'\|
	(\|\Lambda\|+\|\Lambda'\| )^{p}
	\left( \sum_{k=m+1}^{n} \psi(y_{k} ) \right)^{p}
\end{align}
for $\theta\in\Theta$, $\Lambda,\Lambda'\in{\cal L}_{0}({\cal X} )$, $n>m\geq 0$,
$\boldsymbol\alpha\in\mathbb{N}_{0}^{d}$, $|\boldsymbol\alpha|\leq p$
and a sequence $\boldsymbol y = \{y_{n} \}_{n\geq 1}$ in ${\cal Y}$.
Proposition \ref{proposition1.3} also yields
\begin{align}\label{t1.2.3}
	\left\| F_{\theta,\boldsymbol y}^{\boldsymbol\alpha,m:n}(\Lambda) \right\|
	\leq
	\tilde{C}_{2} \left( \|\Lambda\| \Psi_{\boldsymbol y}^{m:n} \right)^{p}
	\leq
	K\|\Lambda\|^{p}
	\left( \sum_{k=m+1}^{n} \psi(y_{k} ) \right)^{p}
\end{align}
for the same $\theta$, $\Lambda$, $n,m$,
$\boldsymbol\alpha$, $\boldsymbol y$.
As (\ref{t1.2.1*}), (\ref{t1.2.3*}) are trivially satisfied when $n=m$,
the theorem directly follows from (\ref{t1.2.1}), (\ref{t1.2.3}).
\end{proof}

\section{Proof of Theorem \ref{theorem1.3} }\label{section1.2*}

In this section, we rely on the following notation.
$\tilde{\Phi}_{\theta}(x,y,\Lambda)$ is the function defined by
\begin{align}\label{5.101}
	\tilde{\Phi}_{\theta}(x,y,\Lambda)
	=
	\int\int \Phi_{\theta}(x',y',\Lambda) Q(x',dy') P(x,dx')
\end{align}
for $\theta\in\Theta$, $x\in{\cal X}$, $y\in{\cal Y}$,
$\Lambda\in{\cal L}_{0}({\cal X} )$.
$\boldsymbol X$ and $\boldsymbol Y$ denote stochastic processes
$\{X_{n} \}_{n\geq 1}$ and $\{Y_{n} \}_{n\geq 1}$
(i.e., $\boldsymbol X = \{X_{n} \}_{n\geq 1}$, $\boldsymbol Y = \{Y_{n} \}_{n\geq 1}$).
$G_{\theta,\boldsymbol X,\boldsymbol Y}^{m:n}(\Lambda)$ and
$H_{\theta,\boldsymbol X,\boldsymbol Y}^{m:n}(\Lambda)$ are the random functions defined by
\begin{align*}
	G_{\theta,\boldsymbol X,\boldsymbol Y}^{m:n}(\Lambda)
	=
	\Phi_{\theta}\left(X_{n},Y_{n},F_{\theta,\boldsymbol Y}^{m:n}(\Lambda) \right),
	\;\;\;\;\;
	H_{\theta,\boldsymbol X,\boldsymbol Y}^{m:n}(\Lambda)
	=
	\Phi_{\theta}\left(X_{n+1},Y_{n+1},F_{\theta,\boldsymbol Y}^{m:n}(\Lambda) \right)
\end{align*}
for $n\geq m\geq 0$.
$A_{\theta}^{n}(x,\Lambda)$ and $B_{\theta}^{n}(x,\Lambda)$ are the functions defined by
\begin{align*}
	&
	A_{\theta}^{n}(x,\Lambda)
	=
	E\left(\left.
	G_{\theta,\boldsymbol X,\boldsymbol Y}^{0:n}(\Lambda)
	-
	G_{\theta,\boldsymbol X,\boldsymbol Y}^{1:n}(\Lambda)
	\right|X_{0}=x \right),
	\;\;\;\;\;
	B_{\theta}^{n}(x,\Lambda)
	=
	E\left(\left.
	G_{\theta,\boldsymbol X,\boldsymbol Y}^{0:n}(\Lambda)
	\right|X_{0}=x \right)
\end{align*}
for $n\geq 1$.
$C_{\theta}^{n}(x,y,\Lambda)$ and $D_{\theta}^{n}(x,y,\Lambda)$ are the functions defined by
\begin{align*}
	&
	C_{\theta}^{n}(x,y,\Lambda)
	=
	E\left(\left.
	H_{\theta,\boldsymbol X,\boldsymbol Y}^{0:n}(\Lambda)
	-
	H_{\theta,\boldsymbol X,\boldsymbol Y}^{1:n}(\Lambda)
	\right|X_{1}=x,Y_{1}=y \right),
	\\
	&
	D_{\theta}^{n}(x,y,\Lambda)
	=
	E\left(\left.
	H_{\theta,\boldsymbol X,\boldsymbol Y}^{0:n}(\Lambda)
	\right|X_{1}=x,Y_{1}=y \right).
\end{align*}
$\tilde{A}_{\theta}^{m,n}(x,\Lambda)$ and $\tilde{B}_{\theta}^{n}(x,\Lambda)$ are the functions
defined by
\begin{align*}
	&
	\tilde{A}_{\theta}^{m,n}(x,\Lambda)
	=
	\int A_{\theta}^{n-m}(x',\Lambda) (P^{m}-\pi)(x,dx'),
	\\
	&
	\tilde{B}_{\theta}^{n}(x,\Lambda)
	=
	\int\int \Phi_{\theta}(x',y',\Lambda) Q(x',dy') (P^{n}-\pi)(x,dx')
\end{align*}
for $n>m\geq 0$.

\begin{lemma}\label{lemma2.1}
Let Assumptions \ref{a1.1}, \ref{a1.2} and \ref{a1.4} hold.
Then, there exists a real number
$C_{5}\in[1,\infty )$ (depending only on $p$, $q$, $\varepsilon$, $L_{0}$) such
that
\begin{align*}
	&
	\max\left\{
	\left| G_{\theta,\boldsymbol X,\boldsymbol Y}^{0:n}({\cal E}_{\lambda} )
	-
	G_{\theta,\boldsymbol X,\boldsymbol Y}^{0:n}(\Lambda) \right|,
	\left| G_{\theta,\boldsymbol X,\boldsymbol Y}^{0:n}({\cal E}_{\lambda} )
	-
	G_{\theta,\boldsymbol X,\boldsymbol Y}^{1:n}(\Lambda) \right|
	\right\}
	\leq
	C_{5}\tau^{n}\|\Lambda\|^{s}\varphi(X_{n},Y_{n} )
	\sum_{k=1}^{n} \psi^{r}(Y_{k} ),
	\\
	&
	\max\left\{
	\left| H_{\theta,\boldsymbol X,\boldsymbol Y}^{0:n}({\cal E}_{\lambda} )
	-
	H_{\theta,\boldsymbol X,\boldsymbol Y}^{0:n}(\Lambda) \right|,
	\left| H_{\theta,\boldsymbol X,\boldsymbol Y}^{0:n}({\cal E}_{\lambda} )
	-
	H_{\theta,\boldsymbol X,\boldsymbol Y}^{1:n}(\Lambda) \right|
	\right\}
	\leq
	C_{5}\tau^{n}\|\Lambda\|^{s}\varphi(X_{n+1},Y_{n+1} )
	\sum_{k=1}^{n} \psi^{r}(Y_{k} )
\end{align*}
for all $\theta\in\Theta$,
$\lambda\in{\cal P}({\cal X} )$, $\Lambda\in{\cal L}_{0}({\cal X} )$, $n\geq 1$
($r$ and $s$ are specified in Assumption \ref{a1.5} and
Theorem \ref{theorem1.3}, while $\tau$ is defined at the beginning of Section \ref{section1.1*}).
\end{lemma}

\begin{proof}
Throughout the proof, the following notation is used.
$\tilde{C}_{1}$, $\tilde{C}_{2}$ are the real numbers defined by
$\tilde{C}_{1}=\max_{n\geq 1} \tau^{n-1}n^{2r}$,
$\tilde{C}_{2}=\max\{A_{\boldsymbol\alpha}:\boldsymbol\alpha\in\mathbb{N}_{0},|\boldsymbol\alpha|\leq p \big\}$
($A_{\boldsymbol\alpha}$ is specified in Proposition \ref{proposition1.3}).
$\tilde{C}_{3}$, $\tilde{C}_{4}$ are the real numbers defined by
$\tilde{C}_{3}=2^{p}\tilde{C}_{2}^{p+1}$,
$\tilde{C}_{4}=2^{q}\tilde{C}_{2}^{q}\tilde{C}_{3}$,
while $C_{5}$ is the real number defined by $C_{5}=\tilde{C}_{1}\tilde{C}_{4}\tau^{-2}$.
$\theta$, $x$, $y$, $\lambda$ are any elements of
$\Theta$, ${\cal X}$, ${\cal Y}$, ${\cal P}({\cal X} )$ (respectively), while
$\Lambda,\Lambda'$ are any elements of ${\cal L}_{0}({\cal X} )$.
$\boldsymbol y = \{y_{n} \}_{n\geq 1}$ is any sequence in ${\cal Y}$.
$n$, $m$, $k$ are any integers satisfying $n\geq 1$,
$k\geq m\geq 0$.

Owing to Proposition \ref{proposition1.3}, we have
\begin{align}\label{l2.1.1}
	\left\| F_{\theta,\boldsymbol y}^{m:k}(\Lambda) \right\|
	\leq
	\tilde{C}_{2} \left( \|\Lambda\| \Phi_{\boldsymbol y}^{0:k} \right)^{p},
	\;\;\;\;\;
	\left\| F_{\theta,\boldsymbol y}^{m:k}(\Lambda) - F_{\theta,\boldsymbol y}^{m:k}(\Lambda') \right\|
	\leq
	\tilde{C}_{2} \tau^{2(k-m)} \left( (\|\Lambda\|+\|\Lambda'\| ) \Phi_{\boldsymbol y}^{0:k} \right)^{p}
\end{align}
(as $\Phi_{\boldsymbol y}^{0:k}\geq\Phi_{\boldsymbol y}^{m:k}\geq\Psi_{\boldsymbol y}^{m:k}$).
Consequently, we have
\begin{align}
	&\label{l2.1.3}
	\left\| F_{\theta,\boldsymbol y}^{0:k}({\cal E}_{\lambda} ) \right\|
	+
	\left\| F_{\theta,\boldsymbol y}^{m:k}(\Lambda) \right\|
	\leq
	2\tilde{C}_{2} \left( \|\Lambda\| \Phi_{\boldsymbol y}^{0:k} \right)^{p},
	\\
	&
	\left\| F_{\theta,\boldsymbol y}^{0:m}({\cal E}_{\lambda} ) \right\|
	+
	\|\Lambda\|
	\leq
	2\tilde{C}_{2} \left( \|\Lambda\| \Phi_{\boldsymbol y}^{0:m} \right)^{p}
	\leq
	2\tilde{C}_{2} \left( \|\Lambda\| \Phi_{\boldsymbol y}^{0:k} \right)^{p}.
	\nonumber
\end{align}
Then, (\ref{1.705}), (\ref{l2.1.1}) imply
\begin{align}\label{l2.1.5}
	\left\| F_{\theta,\boldsymbol y}^{0:k}({\cal E}_{\lambda} ) - F_{\theta,\boldsymbol y}^{m:k}(\Lambda) \right\|
	=
	\left\| F_{\theta,\boldsymbol y}^{m:k}\left( F_{\theta,\boldsymbol y}^{0:m}({\cal E}_{\lambda} ) \right)
	-
	F_{\theta,\boldsymbol y}^{m:k}(\Lambda) \right\|
	\leq &
	\tilde{C}_{2} \tau^{2(k-m)}
	\left( \left( \left\| F_{\theta,\boldsymbol y}^{0:m}({\cal E}_{\lambda} ) \right\| + \|\Lambda\| \right)
	\Phi_{\boldsymbol y}^{0:k} \right)^{p}
	\nonumber\\
	\leq &
	\tilde{C}_{3} \tau^{2(k-m)} \|\Lambda\|^{p^{2} } \left( \Phi_{\boldsymbol y}^{0:k} \right)^{p(p+1)}
\end{align}
(as $\Phi_{\boldsymbol y}^{0:k}\geq 1$).
Combining Assumption \ref{a1.4} with (\ref{l2.1.3}), (\ref{l2.1.5}), we get
\begin{align}\label{l2.1.7}
	\left| \Phi_{\theta}(x,y,F_{\theta,\boldsymbol y}^{0:k}({\cal E}_{\lambda} ) )
	-
	\Phi_{\theta}(x,y,F_{\theta,\boldsymbol y}^{m:k}(\Lambda ) ) \right|
	\leq &
	\varphi(x,y) \left\| F_{\theta,\boldsymbol y}^{0:k}({\cal E}_{\lambda} ) - F_{\theta,\boldsymbol y}^{m:k}(\Lambda) \right\|
	\left(
	\left\| F_{\theta,\boldsymbol y}^{0:k}({\cal E}_{\lambda} ) \right\|
	+
	\left\| F_{\theta,\boldsymbol y}^{m:k}(\Lambda) \right\|
	\right)^{q}
	\nonumber\\
	\leq &
	\tilde{C}_{4} \varphi(x,y) \tau^{2(k-m)} \|\Lambda\|^{s} \left( \Phi_{\boldsymbol y}^{0:k} \right)^{r}.
\end{align}
Using (\ref{l2.1.7}), we deduce
\begin{align*}
	&
	\max\left\{
	\left| G_{\theta,\boldsymbol X,\boldsymbol Y}^{0:n}({\cal E}_{\lambda} )
	-
	G_{\theta,\boldsymbol X,\boldsymbol Y}^{0:n}(\Lambda) \right|,
	\left| G_{\theta,\boldsymbol X,\boldsymbol Y}^{0:n}({\cal E}_{\lambda} )
	-
	G_{\theta,\boldsymbol X,\boldsymbol Y}^{1:n}(\Lambda) \right|
	\right\}
	\\
	&\leq
	\tilde{C}_{4} \tau^{2(n-1)} n^{2r} \|\Lambda\|^{s}
	\varphi(X_{n},Y_{n} ) \sum_{k=1}^{n} \psi^{r}(Y_{k} )
	\leq 
	C_{5} \tau^{n} \|\Lambda\|^{s}
	\varphi(X_{n},Y_{n} )	\sum_{k=1}^{n} \psi^{r}(Y_{k} ).
\end{align*}
Similarly, we conclude
\begin{align*}
	&
	\max\left\{
	\left| H_{\theta,\boldsymbol X,\boldsymbol Y}^{0:n}({\cal E}_{\lambda} )
	\!-\!
	H_{\theta,\boldsymbol X,\boldsymbol Y}^{0:n}(\Lambda) \right|\!,
	\left| H_{\theta,\boldsymbol X,\boldsymbol Y}^{0:n}({\cal E}_{\lambda} )
	\!-\!
	H_{\theta,\boldsymbol X,\boldsymbol Y}^{1:n}(\Lambda) \right|
	\right\}
	\\
	&\leq
	\tilde{C}_{4} \tau^{2(n-\!1)} n^{2r} \|\Lambda\|^{s}
	\varphi(X_{n+\!1},Y_{n+\!1} ) \! \sum_{k=1}^{n} \psi^{r}(Y_{k} )
	\leq
	C_{5} \tau^{n} \|\Lambda\|^{s}
	\varphi(X_{n+\!1},Y_{n+\!1} )	\sum_{k=1}^{n} \psi^{r}(Y_{k} ).
\end{align*}
\end{proof}

\begin{lemma}\label{lemma2.2}
Let Assumptions \ref{a1.1}, \ref{a1.2} and \ref{a1.3} -- \ref{a1.5} hold.
Moreover, let $\rho=\max\{\tau^{1/3},\delta^{1/3} \}$
($\delta$ is specified in Assumption \ref{a1.3},
while $\tau$ is defined at the beginning of Section \ref{section1.1*}).
Then, the following is true.

(i) There exists a real number $C_{6}\in[1,\infty)$ (depending only on $p$, $q$, $\varepsilon$, $\delta$,
$K_{0}$, $L_{0}$) such that
\begin{align*}
	&
	\max\left\{ \big| A_{\theta}^{n}(x,{\cal E}_{\lambda} ) \big|,
	\big|\tilde{A}_{\theta}^{m,n}(x,{\cal E}_{\lambda} ) \big|,
	\big|\tilde{B}_{\theta}^{n}(x,{\cal E}_{\lambda} ) \big|
	\right\}
	\leq
	C_{6}\rho^{2n},
	\;\;\;\;\;
	\left| B_{\theta}^{n}(x,{\cal E}_{\lambda} ) - B_{\theta}^{n}(x,\Lambda) \right|
	\leq
	C_{6}\rho^{2n}\|\Lambda\|^{s}
\end{align*}
for all $\theta\in\Theta$, $x\in{\cal X}$, $\lambda\in{\cal P}({\cal X} )$,
$\Lambda\in{\cal L}_{0}({\cal X} )$, $n>m\geq 0$.

(ii) There exists a real number $C_{7}\in[1,\infty)$ (depending only on $p$, $q$, $\varepsilon$, $\delta$,
$K_{0}$, $L_{0}$) such that
\begin{align*}
	&
	\big| C_{\theta}^{n}(x,y,{\cal E}_{\lambda} ) \big|
	\leq
	C_{7}\rho^{2n}\psi^{r}(y),
	\;\;\;\;\;
	\left| D_{\theta}^{n}(x,y,{\cal E}_{\lambda} ) - D_{\theta}^{n}(x,y,\Lambda) \right|
	\leq
	C_{7}\rho^{2n}\|\Lambda\|^{s}\psi^{r}(y)
\end{align*}
for all $\theta\in\Theta$, $x\in{\cal X}$, $y\in{\cal Y}$, $\lambda\in{\cal P}({\cal X} )$,
$\Lambda\in{\cal L}_{0}({\cal X} )$, $n\geq 1$.
\end{lemma}

\begin{proof}
Throughout the proof, the following notation is used.
$\tilde{C}_{1}$, $\tilde{C}_{2}$ are the real numbers defined by $\tilde{C}_{1}=\max_{n\geq 1} \rho^{n-1}n$,
$\tilde{C}_{2}=L_{0}^{2}$ ($L_{0}$ is specified in Assumption \ref{a1.5}).
$\theta$, $x$, $y$, $\lambda$, $\Lambda$ are any elements of
$\Theta$, ${\cal X}$, ${\cal Y}$, ${\cal P}({\cal X} )$, ${\cal L}_{0}({\cal X} )$ (respectively).
$n,m$ are any integers satisfying $n>m\geq 0$.

Owing to Assumption \ref{a1.4}, we have
\begin{align}\label{l2.2.1}
	E\left(\left.
	\varphi(X_{k},Y_{k} ) \psi^{r}(Y_{k} )
	\right|X_{0}=x \right)
	=
	E\left(\left.
	\int\varphi(X_{k},y) \psi^{r}(y) Q(X_{k},dy)
	\right|X_{0}=x \right)
	\leq
	L_{0}
\end{align}
for $k\geq 0$.
Due to the same assumption, we have
\begin{align}\label{l2.2.5}
	\max\left\{
	\int \varphi(x,y') Q(x,dy'), \int \psi^{r}(y') Q(x,dy')
	\right\}
	\leq
	\int \varphi(x,y') \psi^{r}(y') Q(x,dy')
	\leq
	L_{0}.
\end{align}
Consequently, we get
\begin{align}\label{l2.2.3}
	E\left(\left.
	\varphi(X_{l},Y_{l} ) \psi^{r}(Y_{k} )
	\right|X_{0}=x \right)
	=
	E\left(\left.
	\int \varphi(X_{l},y) Q(X_{l},dy) \int \psi^{r}(y) Q(X_{k},dy)
	\right|X_{0}=x \right)
	\leq
	L_{0}^{2}
\end{align}
for $l>k\geq 0$.
Similarly, we get
\begin{align}
	&\label{l2.2.7}
	E\left(\left.
	\varphi(X_{k},Y_{k} ) \psi^{r}(Y_{1} )
	\right|X_{1}=x,Y_{1}=y \right)
	=
	\psi^{r}(y)
	E\left(\left.
	\int \varphi(X_{k},y') Q(X_{k},dy')
	\right|X_{1}=x \right)
	\leq
	L_{0} \psi^{r}(y),
	\\
	&\label{l2.2.9}
	E\left(\left.
	\varphi(X_{l},Y_{l} ) \psi^{r}(Y_{k} )
	\right|X_{1}=x,Y_{1}=y \right)
	=
	E\left(\left.
	\int \varphi(X_{l},y') Q(X_{l},dy') \int \psi^{r}(y') Q(X_{k},dy')
	\right|X_{1}=x \right)
	\leq
	L_{0}^{2}
\end{align}
for $l>k>1$.

Let $C_{6}$ be the real number defined by $C_{6}=\tilde{C}_{1}\tilde{C}_{2}C_{5}K_{0}$
($K_{0}$, $C_{5}$ are specified in Assumption \ref{a1.3} and Lemma \ref{lemma2.1}).
Since $\tau^{n}n\leq \rho^{3n}(n+1)\leq\tilde{C}_{1}\rho^{2n}$,
Lemma \ref{lemma2.1} and (\ref{l2.2.1}), (\ref{l2.2.3}) imply
\begin{align*}
	\big| A_{\theta}^{n}(x,{\cal E}_{\lambda} ) \big|
	\leq &
	E\left(\left.
	\left|
	G_{\theta,\boldsymbol X,\boldsymbol Y}^{0:n}({\cal E}_{\lambda} )
	-
	G_{\theta,\boldsymbol X,\boldsymbol Y}^{1:n}({\cal E}_{\lambda} )
	\right|\:
	\right|X_{0}=x \right)
	\\
	\leq &
	C_{5} \tau^{n}
	\sum_{k=1}^{n}
	E\left(\left. \varphi(X_{n},Y_{n} ) \psi^{r}(Y_{k} ) \right|X_{0}=x \right)
	\\
	\leq &
	\tilde{C}_{2}C_{5}\tau^{n}n
	\leq 
	C_{6}\rho^{2n}.
\end{align*}
As $\tau^{n-m}\delta^{m}(n-m)\leq \rho^{3n}(n+1)\leq\tilde{C}_{1}\rho^{2n}$,
Assumption \ref{a1.4} yields
\begin{align*}
	\big| \tilde{A}_{\theta}^{m,n}(x,{\cal E}_{\lambda} ) \big|
	\leq
	\int \big| A_{\theta}^{n-m}(x',{\cal E}_{\lambda} ) \big| |P^{m}-\pi|(x,dx')
	\leq
	\tilde{C}_{2}C_{5}K_{0}\tau^{n-m}\delta^{m}(n-m)
	\leq
	C_{6}\rho^{2n}.
\end{align*}
Moreover, owing to Lemma \ref{lemma2.1} and (\ref{l2.2.1}), (\ref{l2.2.3}), we have
\begin{align*}
	\big| B_{\theta}^{n}(x,{\cal E}_{\lambda} ) - B_{\theta}^{n}(x,\Lambda) \big|
	\leq &
	E\left(\left.
	\left|
	G_{\theta,\boldsymbol X,\boldsymbol Y}^{0:n}({\cal E}_{\lambda} )
	-
	G_{\theta,\boldsymbol X,\boldsymbol Y}^{0:n}(\Lambda)
	\right|\:
	\right|X_{0}=x \right)
	\\
	\leq &
	C_{5} \tau^{n} \|\Lambda\|^{s}
	\sum_{k=1}^{n}
	E\left(\left. \varphi(X_{n},Y_{n} ) \psi^{r}(Y_{k} ) \right|X_{0}=x \right)
	\\
	\leq &
	\tilde{C}_{2}C_{5}\tau^{n}n \|\Lambda\|^{s}
	\leq 
	C_{6}\rho^{2n} \|\Lambda\|^{s}.
\end{align*}
Similarly, due to Assumptions \ref{a1.3}, \ref{a1.4} and (\ref{l2.2.5}), we have
\begin{align*}
	\big| \tilde{B}_{\theta}^{n}(x,{\cal E}_{\lambda} ) \big|
	\leq &
	\int\int \big| \Phi_{\theta}(x',y',{\cal E}_{\lambda} ) \big| Q(x',dy') |P^{n}-\pi|(x,dx')
	\\
	\leq &
	\int\int \varphi(x',y') Q(x',dy') |P^{n}-\pi|(x,dx')
	\\
	\leq &
	\tilde{C}_{2}K_{0}\delta^{n}
	\leq 
	C_{6}\rho^{2n}.
\end{align*}

Let $C_{7}$ be the real number defined by $C_{7}=\tilde{C}_{1}\tilde{C}_{2}C_{5}$
($C_{5}$ is specified in Lemma \ref{lemma2.1}).
Relying on Lemma \ref{lemma2.1} and (\ref{l2.2.7}), (\ref{l2.2.9}), we deduce
\begin{align*}
	\big| C_{\theta}^{n}(x,y,{\cal E}_{\lambda} ) \big|
	\leq &
	E\left(\left.
	\left|
	H_{\theta,\boldsymbol X,\boldsymbol Y}^{0:n}({\cal E}_{\lambda} )
	-
	H_{\theta,\boldsymbol X,\boldsymbol Y}^{1:n}({\cal E}_{\lambda} )
	\right|\:
	\right|X_{1}=x,Y_{1}=y \right)
	\\
	\leq &
	C_{5} \tau^{n}
	\sum_{k=1}^{n}
	E\left(\left. \varphi(X_{n+1},Y_{n+1} ) \psi^{r}(Y_{k} ) \right|X_{1}=x,Y_{1}=y \right)
	\\
	\leq &
	\tilde{C}_{2}C_{5}\tau^{n}n \psi^{r}(y)
	\leq 
	C_{7}\rho^{2n} \psi^{r}(y).
\end{align*}
Using the same arguments, we conclude
\begin{align*}
	\big| D_{\theta}^{n}(x,y,{\cal E}_{\lambda} ) - D_{\theta}^{n}(x,y,\Lambda ) \big|
	\leq &
	E\left(\left.
	\left|
	H_{\theta,\boldsymbol X,\boldsymbol Y}^{0:n}({\cal E}_{\lambda} )
	\!-\!
	H_{\theta,\boldsymbol X,\boldsymbol Y}^{0:n}(\Lambda)
	\right|\:
	\right|X_{1}=x,Y_{1}=y \right)
	\\
	\leq &
	C_{5} \tau^{n} \|\Lambda\|^{s}
	\sum_{k=1}^{n}
	E\left(\left. \varphi(X_{n+1},Y_{n+1} ) \psi^{r}(Y_{k} ) \right|X_{1}=x,Y_{1}=y \right)
	\\
	\leq &
	\tilde{C}_{2}C_{5}\tau^{n}n \|\Lambda\|^{s}\psi^{r}(y)
	\leq 
	C_{7}\rho^{2n} \|\Lambda\|^{s}\psi^{r}(y).
\end{align*}
\end{proof}

\begin{proof}[\rm\bf Proof of Theorem \ref{theorem1.3}]
Throughout the proof, the following notation is used.
$\tilde{C}_{1}$ is the real number defined by
$\tilde{C}_{1}=\max_{n\geq 1}\rho^{n-1}n$,
while $\tilde{C}_{2}$, $\tilde{C}_{3}$ are the real numbers
defined by
$\tilde{C}_{2}=4\tilde{C}_{1}C_{6}$, $\tilde{C}_{3}=\tilde{C}_{2}(1-\rho)^{-1}$
($\rho$, $C_{6}$ are specified in Lemma \ref{lemma2.2}).
$L$ is the real number defined by $L=4\tilde{C}_{3}C_{7}L_{0}\rho^{-1}$
($L_{0}$, $C_{7}$ are specified in Assumption \ref{a1.5} and Lemma \ref{lemma2.2}).
$\theta$ is any element of
$\Theta$.
$x,x'$ are any elements of ${\cal X}$, while
$y,y'$ are any elements of ${\cal Y}$.
$\lambda,\lambda'$ are any elements of ${\cal P}({\cal X} )$, while
$\Lambda,\Lambda'$ are any elements of ${\cal L}_{0}({\cal X} )$.
$n$ is any (strictly) positive integer.

It is easy to notice that $G_{\theta,\boldsymbol X,\boldsymbol Y}^{l:n}({\cal E}_{\lambda} )$ does not depend
on $X_{0},Y_{0},\dots,X_{k},Y_{k}$ for $n\geq l\geq k\geq 0$.
It is also easy to show
\begin{align*}
	E\left(\left.G_{\theta,\boldsymbol X,\boldsymbol Y}^{l:n}({\cal E}_{\lambda} )\right|X_{k}=x\right)=
	E\left(\left.G_{\theta,\boldsymbol X,\boldsymbol Y}^{l-k:n-k}({\cal E}_{\lambda} )\right|X_{0}=x\right)
\end{align*}
for the same $k,l$.
Then, we conclude
\begin{align}\label{t1.3.1}
	(\Pi^{n}\Phi)_{\theta}(x,y,\Lambda)
	= 
	E\left(\left.
	G_{\theta,\boldsymbol X,\boldsymbol Y}^{0:n}(\Lambda)
	\right|X_{0}=x\right)
	=&
	\sum_{k=0}^{n-1}
	E\left(\left.
	E\left(\left.
	G_{\theta,\boldsymbol X,\boldsymbol Y}^{k:n}({\cal E}_{\lambda} )
	-
	G_{\theta,\boldsymbol X,\boldsymbol Y}^{k+1:n}({\cal E}_{\lambda} )
	\right|X_{k}\right)
	\right|X_{0}=x \right)
	\nonumber\\
	&+
	E\left(\left.
	G_{\theta,\boldsymbol X,\boldsymbol Y}^{0:n}(\Lambda)
	-
	G_{\theta,\boldsymbol X,\boldsymbol Y}^{0:n}({\cal E}_{\lambda} )
	\right|X_{0}=x \right)
	\nonumber\\
	&+
	E\left(\left.E\left(\left.
	\Phi_{\theta}(X_{n},Y_{n},{\cal E}_{\lambda} )
	\right|X_{n}\right)\right|X_{0}=x \right)
	\nonumber\\
	=&
	\sum_{k=0}^{n-1} \left(
	\tilde{A}_{\theta}^{k,n}(x,{\cal E}_{\lambda} ) + \bar{A}_{\theta}^{k,n}({\cal E}_{\lambda} )
	\right)
	+
	B_{\theta}^{n}(x,\Lambda)
	-
	B_{\theta}^{n}(x,{\cal E}_{\lambda} )
	\nonumber\\
	&+
	\tilde{B}_{\theta}^{n}(x,{\cal E}_{\lambda} )
	+
	\bar{B}_{\theta}^{n}({\cal E}_{\lambda} ),
\end{align}
where
\begin{align*}
	\bar{A}_{\theta}^{k,n}({\cal E}_{\lambda} )
	=
	\int A_{\theta}^{n-k}(x',{\cal E}_{\lambda} ) \pi(dx'),
	\;\;\;\;\;
	\bar{B}_{\theta}^{n}({\cal E}_{\lambda} )
	=
	\int\int \Phi_{\theta}^{n}(x',y',{\cal E}_{\lambda} ) Q(x',dy')\pi(dx').
\end{align*}
We also deduce
\begin{align}\label{t1.3.3}
	(\Pi^{n}\Phi)_{\theta}(x,y,\Lambda)
	=
	E\left(\left.
	G_{\theta,\boldsymbol X,\boldsymbol Y}^{0:n}(\Lambda)
	\right|X_{0}=x\right)
	=&
	E\left(\left.
	G_{\theta,\boldsymbol X,\boldsymbol Y}^{0:n}(\Lambda)
	-
	G_{\theta,\boldsymbol X,\boldsymbol Y}^{1:n}(\Lambda)
	\right|X_{0}=x\right)
	\nonumber\\
	&+
	E\left(\left.E\left(\left.
	G_{\theta,\boldsymbol X,\boldsymbol Y}^{1:n}(\Lambda)
	\right|X_{1}\right)\right|X_{0}=x\right)
	\nonumber\\
	=&
	A_{\theta}^{n}(x,\Lambda)
	+
	E\left(\left.
	(\Pi^{n-1}\Phi)_{\theta}(X_{1},Y_{1},\Lambda)
	\right|X_{0}=x\right).
\end{align}
Since $\rho^{2n}(n+1)\leq\tilde{C}_{1}\rho^{n}$,
Lemma \ref{lemma2.2} and (\ref{t1.3.1}) imply
\begin{align}\label{t1.3.5}
	\left|
	(\Pi^{n}\Phi )_{\theta}(x,y,\Lambda) - (\Pi^{n}\Phi )_{\theta}(x',y', {\cal E}_{\lambda} )
	\right|
	\leq &
	\left|\tilde{B}_{\theta}^{n}(x,{\cal E}_{\lambda} ) \right|
	+
	\left|\tilde{B}_{\theta}^{n}(x',{\cal E}_{\lambda} ) \right|
	+
	\left| B_{\theta}^{n}(x,\Lambda) - B_{\theta}^{n}(x,{\cal E}_{\lambda} ) \right|
	\nonumber\\
	&+
	\sum_{k=0}^{n-1}
	\left|\tilde{A}_{\theta}^{k:n}(x,{\cal E}_{\lambda} ) \right|
	+
	\sum_{k=0}^{n-1}
	\left|\tilde{A}_{\theta}^{k:n}(x',{\cal E}_{\lambda} ) \right|
	\nonumber\\
	\leq &
	2C_{6} \rho^{2n} (n+1) + C_{6}\rho^{2n} \|\Lambda\|^{s}
	\leq 
	\tilde{C}_{2}\rho^{n}\|\Lambda\|^{s}.
\end{align}
Then, Lemma \ref{lemma2.2} and (\ref{t1.3.3}) yield
\begin{align}\label{t1.3.7}
	\left|
	(\Pi^{n+1} \Phi )_{\theta}(x,y,{\cal E}_{\lambda} )
	-
	(\Pi^{n} \Phi )_{\theta}(x,y,{\cal E}_{\lambda} )
	\right|
	\leq &
	E\left(\left|
	(\Pi^{n}\Phi )_{\theta}(X_{1},Y_{1}, {\cal E}_{\lambda} )
	-
	(\Pi^{n}\Phi )_{\theta}(x,y,{\cal E}_{\lambda} )
	\right|
	\Big|X_{0}=x
	\right)
	\nonumber\\
	&+
	\left| A_{\theta}^{n+1}(x,{\cal E}_{\lambda} )\right|
	\nonumber\\
	\leq &
	2C_{6} \rho^{2n} (n+2) + C_{6} \rho^{2(n+1) }
	\leq 
	\tilde{C}_{2} \rho^{n}.
\end{align}

Let $\phi_{\theta }(x,y, {\cal E}_{\lambda} )$ be the function defined by
\begin{align*}
	\phi_{\theta }(x,y, {\cal E}_{\lambda} )
	=
	\Phi_{\theta}(x,y, {\cal E}_{\lambda} )
	+
	\sum_{n=0}^{\infty }
	\left(
	(\Pi^{n+1} \Phi )_{\theta}(x,y, {\cal E}_{\lambda} )
	-
	(\Pi^{n} \Phi )_{\theta}(x,y, {\cal E}_{\lambda} )
	\right).
\end{align*}
Owing to (\ref{t1.3.7}), $\phi_{\theta }(x,y, {\cal E}_{\lambda} )$ is well-defined.
Due to the same inequality, we have
\begin{align}\label{t1.3.9}
	\left|
	(\Pi^{n} \Phi )_{\theta}(x,y, {\cal E}_{\lambda} )
	-
	\phi_{\theta}(x,y, {\cal E}_{\lambda} )
	\right|
	\leq &
	\sum_{k=n}^{\infty }
	\left|
	(\Pi^{k+1} \Phi )_{\theta}(x,y, {\cal E}_{\lambda} )
	-
	(\Pi^{k} \Phi )_{\theta}(x,y, {\cal E}_{\lambda} )
	\right|
	\leq 
	\tilde{C}_{2}
	\sum_{k=n}^{\infty } \rho^{k}
	= 
	\tilde{C}_{3} \rho^{n}.
\end{align}
Consequently, (\ref{t1.3.5}) yields
\begin{align*}
	\left|
	\phi_{\theta}(x,y, {\cal E}_{\lambda} )
	-
	\phi_{\theta}(x',y', {\cal E}_{\lambda'} )
	\right|
	\leq &
	\left|
	(\Pi^{n} \Phi )_{\theta}(x,y, {\cal E}_{\lambda} )
	-
	(\Pi^{n} \Phi )_{\theta}(x',y', {\cal E}_{\lambda'} )
	\right|
	\\
	&+
	\left|
	(\Pi^{n} \Phi )_{\theta}(x,y, {\cal E}_{\lambda} )
	-
	\phi_{\theta}(x,y, {\cal E}_{\lambda} )
	\right|
	\\
	&+
	\left|
	(\Pi^{n} \Phi )_{\theta}(x',y', {\cal E}_{\lambda'} )
	-
	\phi_{\theta}(x',y', {\cal E}_{\lambda'} )
	\right|
	\\
	\leq &
	(\tilde{C}_{2}+2\tilde{C}_{3} ) \rho^{n}.
\end{align*}
Letting $n\rightarrow\infty$, we conclude
$\phi_{\theta}(x,y, {\cal E}_{\lambda} ) = \phi_{\theta}(x',y', {\cal E}_{\lambda'} )$.
Hence,
there exists a function $\phi_{\theta}$ which maps $\theta$ to $\mathbb{R}$
and satisfies $\phi_{\theta} = \phi_{\theta}(x,y, {\cal E}_{\lambda} )$
for each $\theta\in\Theta$, $x\in{\cal X}$, $y\in{\cal Y}$,
$\lambda\in{\cal P}({\cal X} )$.
Then, (\ref{t1.3.5}), (\ref{t1.3.9}) imply
\begin{align}\label{t1.3.51}
	\left|
	(\Pi^{n} \Phi )_{\theta }(x,y,\Lambda)
	-
	\phi_{\theta}
	\right|
	\leq &
	\left|
	(\Pi^{n} \Phi )_{\theta }(x,y,\Lambda)
	-
	(\Pi^{n} \Phi )_{\theta }(x,y,{\cal E}_{\lambda} )
	\right|
	+
	\left|
	(\Pi^{n} \Phi )_{\theta }(x,y,{\cal E}_{\lambda} )
	-
	\phi_{\theta}
	\right|
	\nonumber\\
	\leq &
	\tilde{C}_{2}\rho^{n}\|\Lambda\|^{s} + \tilde{C}_{3}\rho^{n}
	\leq 
	2\tilde{C}_{3}\rho^{n}\|\Lambda\|^{s}
	\leq 
	L\rho^{n}\|\Lambda\|^{s}
\end{align}
(as $\|\Lambda\|\geq 1$).

Owing to Assumption \ref{a1.4}, we have
\begin{align}\label{t1.3.53}
	&
	\left|\tilde{\Phi}_{\theta}(x,y,\Lambda ) \right|
	\leq
	\int\int \left|\Phi_{\theta}(x',y',\Lambda ) \right| Q(x',dy') P(x,dx')
	\leq
	\int\int \varphi(x',y')\|\Lambda\|^{q} Q(x',dy') P(x,dx')
	\leq
	L_{0}\|\Lambda\|^{q}
\end{align}
(see also (\ref{l2.2.5})).
Due to the same assumption, we have
\begin{align}\label{t1.3.55}
	\left|\tilde{\Phi}_{\theta}(x,y,\Lambda ) - \tilde{\Phi}_{\theta}(x,y,\Lambda' ) \right|
	\leq &
	\int\int \left|\Phi_{\theta}(x',y',\Lambda ) - \Phi_{\theta}(x',y',\Lambda' ) \right|
	Q(x',dy') P(x,dx')
	\nonumber\\
	\leq &
	\int\int \varphi(x',y') \|\Lambda - \Lambda' \|
	\left(\|\Lambda \| + \|\Lambda' \| \right)^{q}
	Q(x',dy') P(x,dx')
	\nonumber\\
	\leq &
	L_{0} \|\Lambda - \Lambda' \|
	\left(\|\Lambda \| + \|\Lambda' \| \right)^{q}.
\end{align}
Using (\ref{t1.3.53}), (\ref{t1.3.55}),
we conclude that
Assumption \ref{a1.4} holds when $\Phi_{\theta}(x,y,\Lambda)$ is replaced by
$\tilde{\Phi}_{\theta}(x,y,\Lambda)/L_{0}$.
Consequently, Assumption \ref{a1.4} and (\ref{t1.3.51}) imply that
there exists a function $\tilde{\phi}_{\theta}$ mapping $\theta$ to $\mathbb{R}$
such that (\ref{t1.3.51})
is still true when $\Phi_{\theta}(x,y,\Lambda)$, $\phi_{\theta}$ are replaced
with $\tilde{\Phi}_{\theta}(x,y,\Lambda)/L_{0}$, $\tilde{\phi}_{\theta}/L_{0}$
(respectively).
Hence, we get
\begin{align}\label{t1.3.21}
	\left|(\Pi^{n}\tilde{\Phi} )_{\theta}(x,y,\Lambda ) - \tilde{\phi}_{\theta } \right|
	\leq
	2\tilde{C}_{3}L_{0}\rho^{n} \|\Lambda \|^{s}.
\end{align}
Moreover, it is easy notice that $H_{\theta,\boldsymbol X,\boldsymbol Y}^{1:n}({\cal E}_{\lambda} )$
does not depend on $X_{1},Y_{1},X_{2},Y_{2}$.
Then, we conclude
\begin{align*}
	(\tilde{\Pi}^{n} \Phi)_{\theta}(x,y,\Lambda)
	=
	E\left(\left.
	H_{\theta,\boldsymbol X,\boldsymbol Y}^{0:n}(\Lambda)
	\right|X_{1}=x,Y_{1}=y\right)
	=&
	E\left(\left.
	H_{\theta,\boldsymbol X,\boldsymbol Y}^{0:n}(\Lambda)
	-
	H_{\theta,\boldsymbol X,\boldsymbol Y}^{0:n}({\cal E}_{\lambda} )
	\right|X_{1}=x,Y_{1}=y\right)
	\nonumber\\
	&+
	E\left(\left.
	H_{\theta,\boldsymbol X,\boldsymbol Y}^{0:n}({\cal E}_{\lambda} )
	-
	H_{\theta,\boldsymbol X,\boldsymbol Y}^{1:n}({\cal E}_{\lambda} )
	\right|X_{1}=x,Y_{1}=y\right)
	\nonumber\\
	&+
	E\left(\left.E\left(\left.
	H_{\theta,\boldsymbol X,\boldsymbol Y}^{1:n}({\cal E}_{\lambda} )
	\right|X_{2:n},Y_{2:n}\right)\right|X_{1}=x,Y_{1}=y\right)
	\nonumber\\
	=&
	C_{\theta}^{n}(x,y,{\cal E}_{\lambda} )
	+
	D_{\theta}^{n}(x,y,\Lambda)
	-
	D_{\theta}^{n}(x,y,{\cal E}_{\lambda} )
	\nonumber\\
	&+
	E\left(\left.
	\tilde{\Phi}_{\theta}(X_{n},Y_{n},F_{\theta,\boldsymbol Y}^{1:n}({\cal E}_{\lambda} ) )
	\right|X_{1}=x,Y_{1}=y\right)
	\\
	=&
	C_{\theta}^{n}(x,y,{\cal E}_{\lambda} )
	+
	D_{\theta}^{n}(x,y,\Lambda)
	-
	D_{\theta}^{n}(x,y,{\cal E}_{\lambda} )
	\nonumber\\
	&+
	(\Pi^{n-1}\tilde{\Phi} )_{\theta}(x,y,{\cal E}_{\lambda} ).
\end{align*}
Combining this with Lemma \ref{lemma2.2} and (\ref{t1.3.21}), we get
\begin{align*}
	\left| (\tilde{\Pi}^{n}\Phi )_{\theta}(x,y,\Lambda ) - \tilde{\phi}_{\theta } \right|
	\leq &
	\left|(\Pi^{n-1}\tilde{\Phi} )_{\theta}(x,y,{\cal E}_{\lambda} ) - \tilde{\phi}_{\theta } \right|
	+
	\left| C_{\theta}^{n}(x,y,{\cal E}_{\lambda} ) \right|
	+
	\left| D_{\theta}^{n}(x,y,\Lambda) - D_{\theta}^{n}(x,y,{\cal E}_{\lambda} ) \right|
	\nonumber\\
	\leq &
	2\tilde{C}_{3}L_{0}\rho^{n-1}
	+
	C_{7} \rho^{2n} \psi^{r}(y)
	+
	C_{7} \rho^{2n} \psi^{r}(y) \|\Lambda \|^{s}
	\nonumber\\
	\leq &
	4\tilde{C}_{3}C_{7}L_{0}\rho^{n-1}\psi^{r}(y)\|\Lambda \|^{s}
	\leq 
	L \rho^{n} \psi^{r}(y) \|\Lambda \|^{s}.
\end{align*}
\end{proof}

\section{Proof of Theorems \ref{theorem1.1} and \ref{theorem2.1} }\label{section2*}

In this section, we rely on the following notation.
For $1\leq i\leq d$,
$e_{i}$ denotes the $i$-th standard unit vector in $\mathbb{N}_{0}^{d}$.
$\boldsymbol e_{\boldsymbol\alpha}$ is the vector defined by
\begin{align*}
	i(\boldsymbol\alpha )
	=
	\min\{i: e_{i}\leq\boldsymbol\alpha, 1\leq i\leq d \},
	\;\;\;\;\;
	\boldsymbol e_{\boldsymbol\alpha}
	=
	e_{i(\boldsymbol\alpha ) }
\end{align*}
for $\boldsymbol\alpha\in\mathbb{N}_{0}^{d}\setminus\{\boldsymbol 0\}$.
$\Psi_{\theta}(y,\lambda)$, $\Psi_{\theta}^{\boldsymbol 0}(y,\Lambda)$ and
$\Psi_{\theta}^{\boldsymbol\alpha}(y,\Lambda)$
are the functions defined by
\begin{align}\label{7.3}
	\Psi_{\theta}(y,\lambda )
	=
	\log\left( \big\langle R_{\theta,y}^{\boldsymbol 0}(\lambda) \big\rangle \right),
	\;\;\;\;\;
	\Psi_{\theta}^{\boldsymbol 0}(y,\Lambda )
	=
	\Psi_{\theta}(y,\lambda_{\boldsymbol 0} ),
	\;\;\;\;\;
	\Psi_{\theta}^{\boldsymbol\alpha}(y,\Lambda )
	=
	\big\langle S_{\theta,y}^{\boldsymbol\alpha}(\Lambda ) \big\rangle
\end{align}
for $\theta\in\Theta$, $y\in{\cal Y}$,
$\lambda\in{\cal P}({\cal X} )$,
$\Lambda=\left\{\lambda_{\boldsymbol\beta}:
\boldsymbol\beta\in\mathbb{N}_{0}^{d}, |\boldsymbol\beta|\leq p \right\} \in {\cal L}_{0}({\cal X} )$
and $|\boldsymbol\alpha|=1$.
$\Psi_{\theta}^{\boldsymbol\alpha}(y,\Lambda)$
is the function recursively defined by
\begin{align}\label{7.1}
	\Psi_{\theta}^{\boldsymbol\alpha}(y,\Lambda )
	=
	\big\langle S_{\theta,y}^{\boldsymbol\alpha}(\Lambda ) \big\rangle
	-
	\sum_{\stackrel{\scriptstyle \boldsymbol\beta\in\mathbb{N}_{0}^{d}\setminus\{\boldsymbol\alpha\} }
	{\boldsymbol e_{\boldsymbol\alpha} \leq \boldsymbol\beta \leq \boldsymbol\alpha } }
	\left(\boldsymbol\alpha - \boldsymbol e_{\boldsymbol\alpha} \atop \boldsymbol\beta - \boldsymbol e_{\boldsymbol\alpha}\right)
	\Psi_{\theta}^{\boldsymbol\beta}(y,\Lambda )
	\big\langle S_{\theta,y}^{\boldsymbol\alpha-\boldsymbol\beta}(\Lambda ) \big\rangle
\end{align}
for $1<|\boldsymbol\alpha|\leq p$,
where the recursion is in $|\boldsymbol\alpha|$.\footnote
{The last two functions in (\ref{7.3}) are initial conditions in (\ref{7.1}).
At iteration $k$ of (\ref{7.1}) ($1<k\leq p$),
function $\Psi_{\theta}^{\boldsymbol\alpha}(y,\Lambda )$
is computed for multi-indices $\boldsymbol\alpha\in\mathbb{N}_{0}^{d}$, $|\boldsymbol\alpha|=k$
using the results obtained at the previous iterations. }

\begin{proposition}\label{proposition5.1}
Let Assumptions \ref{a1.1} -- \ref{a1.2'} hold.
Then, $p_{\theta,\boldsymbol y}^{0:n}(x|\lambda )$,
$P_{\theta,\boldsymbol y}^{0:n}(B|\lambda)$ and
$\Psi_{\theta}^{\boldsymbol 0}(y_{n\!+\!1}\!, P_{\theta,\boldsymbol y}^{0:n}(\lambda ) )$
are $p$-times differentiable in $\theta$
for each $\theta\in\Theta$, $x\in{\cal X}$,
$B\in{\cal B}({\cal X} )$, $\lambda\in{\cal P}({\cal X} )$, $n\geq 1$
and any sequence $\boldsymbol y = \{y_{n} \}_{n\geq 1}$ in ${\cal Y}$
($y_{n+1}$ is the $(n+1)$-th element of $\boldsymbol y$).
Moreover, we have
\begin{align}\label{p5.1.1*}
	\partial_{\theta}^{\boldsymbol\alpha} p_{\theta,\boldsymbol y}^{0:n}(x|\lambda )
	=
	f_{\theta,\boldsymbol y}^{\boldsymbol\alpha, 0:n}(x|{\cal E}_{\lambda } ),
	\;\;\;
	\partial_{\theta}^{\boldsymbol\alpha} P_{\theta,\boldsymbol y}^{0:n}(B|\lambda )
	=
	F_{\theta,\boldsymbol y}^{\boldsymbol\alpha, 0:n}(B|{\cal E}_{\lambda } ),
	\;\;\;
	\partial_{\theta}^{\boldsymbol\alpha}
	\Psi_{\theta}^{\boldsymbol 0}(y_{n+1}, P_{\theta,\boldsymbol y}^{0:n}(\lambda ) )
	=
	\Psi_{\theta}^{\boldsymbol\alpha}(y_{n+1}, F_{\theta,\boldsymbol y}^{0:n}({\cal E}_{\lambda} ) )
\end{align}
for the same $\theta$, $x$,
$B$, $\lambda$, $n$, $\boldsymbol y$ and
any multi-index $\boldsymbol\alpha\in\mathbb{N}_{0}^{d}$, $|\boldsymbol\alpha|\leq p$
(${\cal E}_{\lambda }$,
$f_{\theta,\boldsymbol y}^{\boldsymbol\alpha, 0:n}(x|{\cal E}_{\lambda } )$,
$p_{\theta,\boldsymbol y}^{0:n}(x|\lambda )$,
$F_{\theta,\boldsymbol y}^{\boldsymbol\alpha, 0:n}(B|{\cal E}_{\lambda } )$,
$P_{\theta,\boldsymbol y}^{0:n}(B|\lambda )$  are defined in (\ref{1.903}), (\ref{1.705}) -- (\ref{1.911})).
\end{proposition}

\begin{proof}
Throughout the proof, the following notation is used.
$\theta$, $\lambda$, $B$ are any elements of $\Theta$, ${\cal P}({\cal X} )$, ${\cal B}({\cal X})$ (respectively),
while $x,x'$ are any elements of ${\cal X}$.
$\boldsymbol y = \{y_{n} \}_{n\geq 1}$ is any sequence in ${\cal Y}$,
while
$\boldsymbol\alpha$ is any element of $\mathbb{N}_{0}^{d}$
satisfying $|\boldsymbol\alpha|\leq p$.
$n$ is any (strictly) positive integer.
$\delta_{x}(dx')$ is the Dirac measure centered at $x$.
$\xi_{n}(dx_{0:n}|x,\lambda)$ and $\zeta(dx_{0:n}|\lambda)$ are the measures on ${\cal X}^{n+1}$
defined by
\begin{align}
	&\label{p5.1.1}
	\xi_{n}(A|x,\lambda)
	=
	\int\int\cdots\int\int
	I_{A}(x_{0:n}) \delta_{x}(dx_{n})\mu(dx_{n-1})\cdots\mu(dx_{1})\lambda(dx_{0}),
	\\
	&\label{p5.1.3}
	\zeta_{n}(A|\lambda)
	=
	\int\cdots\int\int
	I_{A}(x_{0:n}) \mu(dx_{n})\cdots\mu(dx_{1})\lambda(dx_{0})
\end{align}
for $A\in{\cal B}({\cal X}^{n+1} )$.\footnote
{When $n=1$, (\ref{p5.1.1}), (\ref{p5.1.3}) should be interpreted as
\begin{align*}
	\xi_{1}(A|x,\lambda)
	=
	\int\int I_{A}(x_{0:1}) \delta_{x}(dx_{1})\lambda(dx_{0}),
	\;\;\;\;\;
	\zeta_{1}(A|\lambda)
	=
	\int\int
	I_{A}(x_{0:1}) \mu(dx_{1})\lambda(dx_{0}).
\end{align*} }
$u_{\theta,\boldsymbol y}^{n}(x_{0:n} )$ is the function defined by
\begin{align*}
	u_{\theta,\boldsymbol y}^{n}(x_{0:n} )
	=
	\prod_{k=1}^{n} r_{\theta}(y_{k}, x_{k}|x_{k-1} )
\end{align*}	
for $x_{0},\dots,x_{n}\in{\cal X}$.
$v_{\theta,\boldsymbol y}^{n}(x|\lambda)$ and $w_{\theta,\boldsymbol y}^{n}(\lambda)$
are the functions defined by
\begin{align}\label{p5.1.1501}
	v_{\theta,\boldsymbol y}^{n}(x|\lambda)
	=
	\int u_{\theta,\boldsymbol y}^{n}(x_{0:n} ) \xi_{n}(dx_{0:n}|x,\lambda),
	\;\;\;\;\;
	w_{\theta,\boldsymbol y}^{n}(\lambda)
	=
	\int u_{\theta,\boldsymbol y}^{n}(x_{0:n} ) \zeta_{n}(dx_{0:n}|\lambda).
\end{align}

Using (\ref{1.903}), it is straightforward to verify
\begin{align}\label{p5.1.5}
	p_{\theta}^{0:n}(x|\lambda)
	=
	\frac{v_{\theta,\boldsymbol y}^{n}(x|\lambda)}{w_{\theta,\boldsymbol y}^{n}(\lambda)},
	\;\;\;\;\;
	P_{\theta}^{0:n}(B|\lambda)
	=
	\int_{B} \frac{v_{\theta,\boldsymbol y}^{n}(x'|\lambda)}{w_{\theta,\boldsymbol y}^{n}(\lambda)}
	\mu(dx'),
	\;\;\;\;\;
	w_{\theta,\boldsymbol y}^{n}(\lambda)
	=
	\int v_{\theta,\boldsymbol y}^{n}(x'|\lambda) \mu(dx').
\end{align}
It is also easy to show
\begin{align*}
	w_{\theta,\boldsymbol y}^{1}(\lambda)
	=
	\int\left(\int r_{\theta}(y_{1},x'|x)\mu(dx') \right) \lambda(dx),
	\;\;\;\;\;
	w_{\theta,\boldsymbol y}^{n+1}(\lambda)
	=
	\int\left(\int r_{\theta}(y_{n+1},x'|x) \mu(dx') \right) v_{\theta,\boldsymbol y}^{n}(x|\lambda) \mu(dx).
\end{align*}
Consequently, Assumption \ref{a1.2} implies
\begin{align}
	&
	w_{\theta,\boldsymbol y}^{1}(\lambda)
	\geq
	\varepsilon \int \mu_{\theta}({\cal X}|y_{1} ) \lambda(dx)
	=
	\varepsilon \mu_{\theta}({\cal X}|y_{1} ),
	\nonumber\\
	&\label{p5.1.7}
	w_{\theta,\boldsymbol y}^{n+1}(\lambda)
	\geq
	\varepsilon \int \mu_{\theta}({\cal X}|y_{n+1} ) v_{\theta,\boldsymbol y}^{n}(x|\lambda) \mu(dx)
	=
	\varepsilon \mu_{\theta}({\cal X}|y_{n+1} ) w_{\theta,\boldsymbol y}^{n}(\lambda).
\end{align}
Iterating (\ref{p5.1.7}), we get
\begin{align}\label{p5.1.9}
	w_{\theta,\boldsymbol y}^{n}(\lambda)
	\geq
	\varepsilon^{n} \prod_{k=1}^{n} \mu_{\theta}({\cal X}|y_{k} )
	>0.
\end{align}

Owing to Leibniz rule and Assumptions \ref{a1.2}, \ref{a1.2'}, we have
\begin{align}\label{p5.1.21}
	\left|\partial_{\theta}^{\boldsymbol\alpha} u_{\theta,\boldsymbol y}^{n}(x_{0:n} ) \right|
	\leq&
	\sum_{\stackrel{\scriptstyle\boldsymbol\beta_{1},\dots,\boldsymbol\beta_{n}
	\in \mathbb{N}_{0}^{d} }
	{\boldsymbol\beta_{1} + \dots + \boldsymbol\beta_{n} = \boldsymbol\alpha} }
	\left( \boldsymbol\alpha \atop {\boldsymbol\beta_{1},\dots,\boldsymbol\beta_{n} } \right)
	\prod_{k=1}^{n} \left|\partial_{\theta}^{\boldsymbol\beta_{k} } r_{\theta}(y_{k},x_{k}|x_{k-1} ) \right|
	\nonumber\\
	\leq &
	\left(\prod_{k=1}^{n} \phi(y_{k},x_{k} ) \right)
	\left(
	\sum_{\stackrel{\scriptstyle\boldsymbol\beta_{1},\dots,\boldsymbol\beta_{n}
	\in \mathbb{N}_{0}^{d} }
	{\boldsymbol\beta_{1} + \dots + \boldsymbol\beta_{n} = \boldsymbol\alpha} }
	\left( \boldsymbol\alpha \atop {\boldsymbol\beta_{1},\dots,\boldsymbol\beta_{n} } \right)
	\prod_{k=1}^{n} (\psi(y_{k} ) )^{|\boldsymbol\beta_{k} |}
	\right)
	\nonumber\\
	\leq &
	2^{|\boldsymbol\alpha|}
	\left(\prod_{k=1}^{n} \psi(y_{k} ) \right)^{|\boldsymbol\alpha|}
	\left(\prod_{k=1}^{n} \phi(y_{k},x_{k} ) \right)
\end{align}
for $x_{0},\dots,x_{n}\in{\cal X}$.
Due to the same assumptions, we have
\begin{align}
	&\label{p5.1.23}
	\int \left( \prod_{k=1}^{n} \phi(y_{k},x_{k} ) \right)
	\xi_{n}(dx_{0:n}|x,\lambda )
	=
	\phi(y_{n},x) \left( \prod_{k=1}^{n-1} \int \phi(y_{k},x_{k} ) \mu(dx_{k} ) \right)
	<\infty,
	\\
	&\label{p5.1.25}
	\int \left( \prod_{k=1}^{n} \phi(y_{k},x_{k} ) \right)
	\zeta_{n}(dx_{0:n}|\lambda )
	=
	\left( \prod_{k=1}^{n} \int \phi(y_{k},x_{k} ) \mu(dx_{k} ) \right)
	<\infty.
\end{align}
Here and throughout the proof, we rely on the convention that
$\prod_{k=i}^{j}$ is equal to one whenever $j<i$.
Using Lemma \ref{lemmaa3} (see Appendix \ref{appendix3}) and (\ref{p5.1.21}) -- (\ref{p5.1.25}),
we conclude that
$v_{\theta,\boldsymbol y}^{n}(x|\lambda)$, $w_{\theta,\boldsymbol y}^{n}(\lambda)$
are well-defined and $p$-times differentiable in $\theta$.
Relying on the same arguments, we deduce
\begin{align}\label{p5.1.27}
	\partial_{\theta}^{\boldsymbol\alpha} v_{\theta,\boldsymbol y}^{n}(x|\lambda)
	=
	\int \partial_{\theta}^{\boldsymbol\alpha} u_{\theta,\boldsymbol y}^{n}(x_{0:n} ) \xi_{n}(dx_{0:n}|x,\lambda),
	\;\;\;\;\;
	\partial_{\theta}^{\boldsymbol\alpha} w_{\theta,\boldsymbol y}^{n}(\lambda)
	=
	\int \partial_{\theta}^{\boldsymbol\alpha} u_{\theta,\boldsymbol y}^{n}(x_{0:n} ) \zeta_{n}(dx_{0:n}|\lambda).
\end{align}
Then, (\ref{p5.1.5}), (\ref{p5.1.9}) imply that $p_{\theta,\boldsymbol y}^{0:n}(x|\lambda)$
is $p$-times differentiable in $\theta$.
Moreover, (\ref{p5.1.21}), (\ref{p5.1.27}) yield
\begin{align}\label{p5.1.29}
	&
	\left|
	\partial_{\theta}^{\boldsymbol\alpha} v_{\theta,\boldsymbol y}^{n}(x|\lambda )
	\right|
	\leq
	\int |\partial_{\theta}^{\boldsymbol\alpha} u_{\theta,\boldsymbol y}^{n}(x_{0:n} ) |
	\xi_{n}(dx_{0:n}|x,\lambda )
	\leq
	2^{|\boldsymbol\alpha|}
	\phi(y_{n},x)
	\left(\prod_{k=1}^{n} \psi(y_{k} ) \right)^{|\boldsymbol\alpha|}
	\left(\prod_{k=1}^{n-1} \int\phi(y_{k},x_{k} )\mu(dx_{k} ) \right).
\end{align}

Let $\tilde{P}_{\theta,\boldsymbol y}^{\boldsymbol\alpha,n}(dx|\lambda)$ be the signed measure
defined by
\begin{align}\label{p5.1.1503}
	\tilde{P}_{\theta,\boldsymbol y}^{\boldsymbol\alpha,n}(B|\lambda)
	=
	\int_{B} \partial_{\theta}^{\boldsymbol\alpha} p_{\theta,\boldsymbol y}^{0:n}(x|\lambda),
	\mu(dx)
\end{align}
while $\tilde{P}_{\theta,\boldsymbol y}^{\boldsymbol\alpha,n}(\lambda)$
is a `short-hand' notation for $\tilde{P}_{\theta,\boldsymbol y}^{\boldsymbol\alpha,n}(dx|\lambda)$.
Moreover, let $\tilde{P}_{\theta,\boldsymbol y}^{0}(\lambda)$ and
$\tilde{P}_{\theta,\boldsymbol y}^{n}(\lambda)$
be the vector measures defined by
\begin{align}\label{p5.1.1505}
	\tilde{P}_{\theta,\boldsymbol y}^{0}(\lambda) = {\cal E}_{\lambda},
	\;\;\;\;\;
	\tilde{P}_{\theta,\boldsymbol y}^{n}(\lambda)
	=
	\left\{ \tilde{P}_{\theta,\boldsymbol y}^{\boldsymbol\alpha,n}(\lambda):
	\boldsymbol\alpha\in\mathbb{N}_{0}^{d}, |\boldsymbol\alpha|\leq p \right\},
\end{align}
where $\tilde{P}_{\theta,\boldsymbol y}^{\boldsymbol\alpha,n}(\lambda)$  is the component $\boldsymbol\alpha$ of
$\tilde{P}_{\theta,\boldsymbol y}^{n}(\lambda)$.
Owing to Lemma \ref{lemmaa3} and (\ref{p5.1.5}), (\ref{p5.1.9}), (\ref{p5.1.29}),
$P_{\theta,\boldsymbol y}^{0:n}(B|\lambda)$ is $p$-times differentiable in $\theta$.
Due to the same arguments, $\tilde{P}_{\theta,\boldsymbol y}^{\boldsymbol\alpha,n}(B|\lambda)$ is
well-defined and satisfies
\begin{align}\label{p5.1.31}
	\tilde{P}_{\theta,\boldsymbol y}^{\boldsymbol\alpha,n}(B|\lambda)
	=
	\partial_{\theta}^{\boldsymbol\alpha} P_{\theta,\boldsymbol y}^{0:n}(B|\lambda)
	=
	\partial_{\theta}^{\boldsymbol\alpha} \tilde{P}_{\theta,\boldsymbol y}^{\boldsymbol 0, n}(B|\lambda)
\end{align}
(as $P_{\theta,\boldsymbol y}^{0:n}(\lambda)=\tilde{P}_{\theta,\boldsymbol y}^{\boldsymbol 0, n}(\lambda)$).

Using (\ref{1.703}), (\ref{1.907}), (\ref{p5.1.5}), (\ref{p5.1.1503}), it is straightforward to verify
\begin{align}
	&\label{p5.1.33}
	\begin{aligned}[b]
	r_{\theta,y_{n+1} }^{\boldsymbol 0}
	\big(x\big|\tilde{P}_{\theta,\boldsymbol y}^{\boldsymbol 0,n}(\lambda) \big)
	=&
	\frac{\int r_{\theta}(y_{n+1},x|x') v_{\theta,\boldsymbol y}^{n}(x'|\lambda ) \mu(dx') }
	{w_{\theta,\boldsymbol y}^{n}(\lambda ) },
	\end{aligned}
	\\
	&\label{p5.1.35}
	\begin{aligned}[b]
	\left\langle
	R_{\theta,y_{n+1} }^{\boldsymbol 0}
	\big(\tilde{P}_{\theta,\boldsymbol y}^{\boldsymbol 0,n}(\lambda) \big)
	\right\rangle
	=&
	\frac{\int\int r_{\theta}(y_{n+1},x''|x') v_{\theta,\boldsymbol y}^{n}(x'|\lambda )
	\mu(dx') \mu(dx'') }
	{w_{\theta,\boldsymbol y}^{n}(\lambda ) }.
	\end{aligned}
\end{align}
Moreover, Leibniz rule, Assumptions \ref{a1.2}, \ref{a1.2'} and (\ref{p5.1.29}) imply
\begin{align}\label{p5.1.37}
	&
	\left|
	\partial_{\theta}^{\boldsymbol\alpha}
	\left(r_{\theta}(y_{n+1},x|x') v_{\theta,\boldsymbol y}^{n}(x'|\lambda ) \right)
	\right|
	\leq 
	\sum_{\stackrel{\scriptstyle\boldsymbol\beta\in\mathbb{N}_{0}^{d} }
	{\boldsymbol\beta\leq\boldsymbol\alpha} }
	\left(\boldsymbol\alpha \atop \boldsymbol\beta \right)
	\left|\partial_{\theta}^{\boldsymbol\alpha-\boldsymbol\beta} r_{\theta}(y_{n+1},x|x') \right|\:
	\left|\partial_{\theta}^{\boldsymbol\beta} v_{\theta,\boldsymbol y}^{n}(x'|\lambda ) \right|
	\nonumber\\
	&\;\;\;\;\;\;\;\;\;
	\begin{aligned}[b]
	\leq &
	\phi(y_{n+1},x)\phi(y_{n},x')
	\left(\prod_{k=1}^{n-1}\int\phi(y_{k},x_{k} )\mu(dx_{k} ) \right)
	\left(
	\sum_{\stackrel{\scriptstyle\boldsymbol\beta\in\mathbb{N}_{0}^{d} }
	{\boldsymbol\beta\leq\boldsymbol\alpha} }
	\left(\boldsymbol\alpha \atop \boldsymbol\beta \right)
	2^{|\boldsymbol\beta|}
	(\psi(y_{n+1} ) )^{|\boldsymbol\alpha-\boldsymbol\beta|}
	\left( \prod_{k=1}^{n} \psi(y_{k} ) \right)^{|\boldsymbol\beta|}
	\right)
	\\
	\leq &
	4^{|\boldsymbol\alpha|}
	\phi(y_{n+1},x)\phi(y_{n},x')
	\left(\prod_{k=1}^{n+1} \psi(y_{k} ) \right)^{|\boldsymbol\alpha|}
	\left(\prod_{k=1}^{n-1}\int\phi(y_{k},x_{k} )\mu(dx_{k} ) \right)
	<\infty.
	\end{aligned}
\end{align}
The same assumptions also yield
\begin{align}\label{p5.1.41}
	\int \phi(y_{n},x')\mu(dx')<\infty,
	\;\;\;\;\;
	\int\int \phi(y_{n+1},x) \phi(y_{n},x') \mu(dx)\mu(dx') < \infty.
\end{align}
Using Lemma \ref{lemmaa3} and (\ref{p5.1.9}), (\ref{p5.1.33}) -- (\ref{p5.1.41}),
we conclude that
$r_{\theta,y_{n+1} }^{\boldsymbol 0}\big(x\big|\tilde{P}_{\theta,\boldsymbol y}^{\boldsymbol 0,n}(\lambda) \big)$,
$\big\langle R_{\theta,y_{n+1} }^{\boldsymbol 0}\big(\tilde{P}_{\theta,\boldsymbol y}^{\boldsymbol 0,n}(\lambda) \big)
\big\rangle$
are well-defined and $p$-times differentiable in $\theta$.\footnote
{To conclude that
$r_{\theta,y_{n+1} }^{\boldsymbol 0}\big(x\big|\tilde{P}_{\theta,\boldsymbol y}^{\boldsymbol 0,n}(\lambda) \big)$
is well-defined and satisfy (\ref{p5.1.43}),
set $z=x'$, $\nu(dz)=\mu(dx')$,
$F_{\theta}(z)=r_{\theta}(y_{n+1},x|x')v_{\theta,\boldsymbol y}^{n}(x'|\lambda)$,
$g_{\theta}=w_{\theta,\boldsymbol y}^{n}(\lambda)$ in Lemma \ref{lemmaa3}
($x$ is treated as a fixed value).
To conclude that
$\big\langle R_{\theta,y_{n+1} }^{\boldsymbol 0}\big(\tilde{P}_{\theta,\boldsymbol y}^{\boldsymbol 0,n}(\lambda) \big)
\big\rangle$
is well-defined and satisfy (\ref{p5.1.45}),
set $z=(x,x')$, $\nu(dz)=\mu(dx)\mu(dx')$,
$F_{\theta}(z)=r_{\theta}(y_{n+1},x|x')v_{\theta,\boldsymbol y}^{n}(x'|\lambda)$,
$g_{\theta}=w_{\theta,\boldsymbol y}^{n}(\lambda)$ in Lemma \ref{lemmaa3}. }
Relying on the same arguments and (\ref{p5.1.5}), we deduce
\begin{align}
	&\label{p5.1.43}
	\partial_{\theta}^{\boldsymbol\alpha}
	r_{\theta,y_{n+1} }^{\boldsymbol 0}\big(x\big|\tilde{P}_{\theta,\boldsymbol y}^{\boldsymbol 0,n}(\lambda) \big)
	=
	\int \partial_{\theta}^{\boldsymbol\alpha}
	\left( r_{\theta}(y_{n+1},x|x') p_{\theta,\boldsymbol y}^{0:n}(x'|\lambda) \right)
	\mu(dx'),
	\\
	&\label{p5.1.45}
	\partial_{\theta}^{\boldsymbol\alpha}
	\left\langle R_{\theta,y_{n+1} }^{\boldsymbol 0}\big(\tilde{P}_{\theta,\boldsymbol y}^{\boldsymbol 0,n}(\lambda) \big)
	\right\rangle
	=
	\int\int \partial_{\theta}^{\boldsymbol\alpha}
	\left( r_{\theta}(y_{n+1},x''|x') p_{\theta,\boldsymbol y}^{0:n}(x'|\lambda) \right)
	\mu(dx'')\mu(dx').
\end{align}
Consequently, Leibniz rule and (\ref{1.703}), (\ref{1.907}), (\ref{p5.1.1503}) imply
\begin{align}\label{p5.1.47}
	\partial_{\theta}^{\boldsymbol\alpha}
	r_{\theta,y_{n+1} }^{\boldsymbol 0}\big(x\big|\tilde{P}_{\theta,\boldsymbol y}^{\boldsymbol 0,n}(\lambda ) \big)
	=&
	\sum_{\stackrel{\scriptstyle\boldsymbol\beta\in\mathbb{N}_{0}^{d} }
	{\boldsymbol\beta\leq\boldsymbol\alpha} }
	\left(\boldsymbol\alpha \atop \boldsymbol\beta \right)
	\int \partial_{\theta}^{\boldsymbol\alpha-\boldsymbol\beta} r_{\theta}(y_{n+1},x|x')
	\partial_{\theta}^{\boldsymbol\beta} p_{\theta,\boldsymbol y}^{0:n}(x'|\lambda )
	\mu(dx')
	\nonumber\\
	=&
	\sum_{\stackrel{\scriptstyle\boldsymbol\beta\in\mathbb{N}_{0}^{d} }
	{\boldsymbol\beta\leq\boldsymbol\alpha} }
	\left(\boldsymbol\alpha \atop \boldsymbol\beta \right)
	r_{\theta,y_{n+1} }^{\boldsymbol\alpha-\boldsymbol\beta}
	\big(x\big|\tilde{P}_{\theta,\boldsymbol y}^{\boldsymbol\beta,n}(\lambda ) \big)
	\nonumber\\
	=&
	s_{\theta,y_{n+1} }^{\boldsymbol\alpha}(x|\tilde{P}_{\theta,\boldsymbol y}^{n}(\lambda ) )
	\left\langle
	R_{\theta,y_{n+1} }^{\boldsymbol 0}
	\big(\tilde{P}_{\theta,\boldsymbol y}^{\boldsymbol 0,n}(\lambda ) \big)
	\right\rangle.
\end{align}
Leibniz rule and (\ref{1.703}), (\ref{1.907}), (\ref{p5.1.1503}) also yield
\begin{align}\label{p5.1.49}
	\partial_{\theta}^{\boldsymbol\alpha}
	\left\langle
	R_{\theta,y_{n+1} }^{\boldsymbol 0}\big( \tilde{P}_{\theta,\boldsymbol y}^{\boldsymbol 0,n}(\lambda ) \big)
	\right\rangle
	=&
	\sum_{\stackrel{\scriptstyle\boldsymbol\beta\in\mathbb{N}_{0}^{d} }
	{\boldsymbol\beta\leq\boldsymbol\alpha} }
	\left(\boldsymbol\alpha \atop \boldsymbol\beta \right)
	\int\int \partial_{\theta}^{\boldsymbol\alpha-\boldsymbol\beta} r_{\theta}(y_{n+1},x''|x')
	\partial_{\theta}^{\boldsymbol\beta} p_{\theta,\boldsymbol y}^{0:n}(x'|\lambda )
	\mu(dx') \mu(dx'')
	\nonumber\\
	=&
	\sum_{\stackrel{\scriptstyle\boldsymbol\beta\in\mathbb{N}_{0}^{d} }
	{\boldsymbol\beta\leq\boldsymbol\alpha} }
	\left(\boldsymbol\alpha \atop \boldsymbol\beta \right)
	\left\langle
	R_{\theta,y_{n+1} }^{\boldsymbol\alpha-\boldsymbol\beta}
	\big(\tilde{P}_{\theta,\boldsymbol y}^{\boldsymbol\beta,n}(\lambda ) \big)
	\right\rangle
	\nonumber\\
	=&
	\left\langle
	S_{\theta,y_{n+1} }^{\boldsymbol\alpha}\big( \tilde{P}_{\theta,\boldsymbol y}^{n}(\lambda ) \big)
	\right\rangle
	\left\langle
	R_{\theta,y_{n+1} }^{\boldsymbol 0}
	\big(\tilde{P}_{\theta,\boldsymbol y}^{\boldsymbol 0,n}(\lambda ) \big)
	\right\rangle.
\end{align}
Moreover, using (\ref{1.703}), (\ref{1.907}), (\ref{p5.1.1501}), (\ref{p5.1.1503}), we get
\begin{align}
	&\label{p5.1.51}
	\begin{aligned}[b]
	\partial_{\theta}^{\boldsymbol\alpha}
	r_{\theta,y_{1} }^{\boldsymbol 0}
	\big(x\big|\tilde{P}_{\theta,\boldsymbol y}^{\boldsymbol 0,0}(\lambda) \big)
	=
	\partial_{\theta}^{\boldsymbol\alpha} v_{\theta,\boldsymbol y}^{1}(x|\lambda)
	=&
	\int \partial_{\theta}^{\boldsymbol\alpha}
	r_{\theta}(x,y_{1}|x') \lambda(dx')
	\\
	=&
	s_{\theta,y_{1} }^{\boldsymbol\alpha}\big(x\big|\tilde{P}_{\theta,\boldsymbol y}^{0}(\lambda) \big)
	\left\langle
	R_{\theta,y_{1} }^{\boldsymbol 0}\big(\tilde{P}_{\theta,\boldsymbol y}^{\boldsymbol 0,0}(\lambda) \big)
	\right\rangle,
	\end{aligned}
	\\
	&\label{p5.1.53}
	\begin{aligned}[b]
	\partial_{\theta}^{\boldsymbol\alpha}
	\left\langle
	R_{\theta,y_{1} }^{\boldsymbol 0}
	\big(\tilde{P}_{\theta,\boldsymbol y}^{\boldsymbol 0,0}(\lambda) \big)
	\right\rangle
	=
	\partial_{\theta}^{\boldsymbol\alpha} w_{\theta,\boldsymbol y}^{1}(\lambda)
	=&
	\int\int \partial_{\theta}^{\boldsymbol\alpha}
	r_{\theta}(x'',y_{1}|x') \mu(dx'') \lambda(dx')
	\\
	=&
	\left\langle
	S_{\theta,y_{1} }^{\boldsymbol\alpha}\big(\tilde{P}_{\theta,\boldsymbol y}^{0}(\lambda) \big)
	\right\rangle
	\left\langle
	R_{\theta,y_{1} }^{\boldsymbol 0}\big(\tilde{P}_{\theta,\boldsymbol y}^{\boldsymbol 0,0}(\lambda) \big)
	\right\rangle
	\end{aligned}
\end{align}
(as $\tilde{P}_{\theta,\boldsymbol y}^{\boldsymbol 0,0}(B|\lambda)=\lambda(B)$).

Relying on (\ref{1.903}), (\ref{1.703}), (\ref{1.907}), (\ref{p5.1.1503}), it is straightforward to verify
\begin{align*}
	p_{\theta,\boldsymbol y}^{0:n}(x|\lambda )
	\left\langle
	R_{\theta,y_{n} }^{\boldsymbol 0}
	\big(\tilde{P}_{\theta,\boldsymbol y}^{\boldsymbol 0,n-1}(\lambda ) \big)
	\right\rangle
	=
	r_{\theta, y_{n} }^{\boldsymbol 0}
	\big(x\big|\tilde{P}_{\theta,\boldsymbol y}^{\boldsymbol 0,n-1}(\lambda) \big).
\end{align*}
Then, Leibniz rule and (\ref{p5.1.49}), (\ref{p5.1.53}) imply
\begin{align*}
	\partial_{\theta}^{\boldsymbol\alpha}
	r_{\theta, y_{n} }^{\boldsymbol 0}
	\big(x\big|\tilde{P}_{\theta,\boldsymbol y}^{\boldsymbol 0,n-1}(\lambda) \big)
	=&
	\sum_{\stackrel{\scriptstyle\boldsymbol\beta\in\mathbb{N}_{0}^{d} }
	{\boldsymbol\beta\leq\boldsymbol\alpha} }
	\left(\boldsymbol\alpha \atop \boldsymbol\beta \right)
	\partial_{\theta}^{\boldsymbol\beta}
	p_{\theta,\boldsymbol y}^{0:n}(x|\lambda )
	\:
	\partial_{\theta}^{\boldsymbol\alpha-\boldsymbol\beta}
	\left\langle
	R_{\theta,y_{n} }^{\boldsymbol 0}
	\big(\tilde{P}_{\theta,\boldsymbol y}^{\boldsymbol 0,n-1}(\lambda ) \big)
	\right\rangle
	\nonumber\\
	=&
	\sum_{\stackrel{\scriptstyle\boldsymbol\beta\in\mathbb{N}_{0}^{d} }
	{\boldsymbol\beta\leq\boldsymbol\alpha} }
	\left(\boldsymbol\alpha \atop \boldsymbol\beta \right)
	\partial_{\theta}^{\boldsymbol\beta}
	p_{\theta,\boldsymbol y}^{0:n}(x|\lambda )
	\left\langle
	S_{\theta,y_{n} }^{\boldsymbol\alpha-\boldsymbol\beta}
	\big(\tilde{P}_{\theta,\boldsymbol y}^{n-1}(\lambda ) \big)
	\right\rangle
	\left\langle
	R_{\theta,y_{n} }^{\boldsymbol 0}
	\big(\tilde{P}_{\theta,\boldsymbol y}^{\boldsymbol 0,n-1}(\lambda ) \big)
	\right\rangle.
\end{align*}
Since
$0<\big\langle	R_{\theta,y_{n+1} }^{\boldsymbol 0}
\big(\tilde{P}_{\theta,\boldsymbol y}^{\boldsymbol 0,n}(\lambda ) \big)
\big\rangle<\infty$ (due to Assumption \ref{a1.1}), we have
\begin{align*}
	\partial_{\theta}^{\boldsymbol\alpha}
	p_{\theta,\boldsymbol y}^{0:n}(x|\lambda )
	=
	\frac{\partial_{\theta}^{\boldsymbol\alpha}
	r_{\theta, y_{n} }^{\boldsymbol 0}
	\big(x\big|\tilde{P}_{\theta,\boldsymbol y}^{\boldsymbol 0,n-1}(\lambda) \big) }
	{\left\langle
	R_{\theta,y_{n} }^{\boldsymbol 0}
	\big(\tilde{P}_{\theta,\boldsymbol y}^{\boldsymbol 0,n-1}(\lambda ) \big)
	\right\rangle }
	-
	\sum_{\stackrel{\scriptstyle\boldsymbol\beta\in\mathbb{N}_{0}^{d}\setminus\{\boldsymbol\alpha\} }
	{\boldsymbol\beta\leq\boldsymbol\alpha} }
	\left(\boldsymbol\alpha \atop \boldsymbol\beta \right)
	\partial_{\theta}^{\boldsymbol\beta}
	p_{\theta,\boldsymbol y}^{0:n}(x|\lambda )
	\left\langle
	S_{\theta,y_{n} }^{\boldsymbol\alpha-\boldsymbol\beta}
	\big(\tilde{P}_{\theta,\boldsymbol y}^{n-1}(\lambda ) \big)
	\right\rangle.
\end{align*}
Combining this with (\ref{p5.1.47}), (\ref{p5.1.51}), we get
\begin{align}\label{p5.1.55}
	\partial_{\theta}^{\boldsymbol\alpha}
	p_{\theta,\boldsymbol y}^{0:n}(x|\lambda )
	=
	s_{\theta,y_{n} }^{\boldsymbol\alpha }\big(x\big|\tilde{P}_{\theta,\boldsymbol y}^{n-1}(\lambda ) \big)
	-
	\sum_{\stackrel{\scriptstyle\boldsymbol\beta\in\mathbb{N}_{0}^{d}\setminus\{\boldsymbol\alpha \} }
	{\boldsymbol\beta\leq\boldsymbol\alpha} }
	\left( \boldsymbol\alpha \atop \boldsymbol\beta \right)
	\partial_{\theta}^{\boldsymbol\beta}
	p_{\theta,\boldsymbol y}^{0:n}(x|\lambda )
	\left\langle
	S_{\theta,y_{n} }^{\boldsymbol\alpha-\boldsymbol\beta}
	\big(\tilde{P}_{\theta,\boldsymbol y}^{n-1}(\lambda ) \big)
	\right\rangle.
\end{align}
Equation (\ref{p5.1.55}) can be interpreted as a recursion in $|\boldsymbol\alpha|$
which generates functions
$\big\{ \partial_{\theta}^{\boldsymbol\alpha} p_{\theta,\boldsymbol y}^{0:n}(x|\lambda ):
\boldsymbol\alpha\in\mathbb{N}_{0}^{d}, |\boldsymbol\alpha|\leq p \big\}$.\footnote
{In (\ref{p5.1.55}),
$p_{\theta,\boldsymbol y}^{0:n}(x|\lambda)$ is the initial condition.
At iteration $k$ of recursion (\ref{p5.1.55}) ($1\leq k\leq p$),
function $\partial_{\theta}^{\boldsymbol\alpha} p_{\theta,\boldsymbol y}^{0:n}(x|\lambda )$
is computed for multi-indices $\boldsymbol\alpha\in\mathbb{N}_{0}^{d}$,
$|\boldsymbol\alpha|=k$ using the results obtained at the previous iterations.
}
Equation (\ref{p5.1.55}) can also be considered as a particular case of (\ref{1.1}) ---
to get (\ref{p5.1.55}), set $\Lambda=\tilde{P}_{\theta,\boldsymbol y}^{n-1}(\lambda )$, $y=y_{n}$
in (\ref{1.1}).
Hence, comparing (\ref{p5.1.55}) with (\ref{1.1})
and using (\ref{1.905}), (\ref{1.909}),
(\ref{p5.1.1503}), (\ref{p5.1.1505}),
we conclude
\begin{align}\label{p5.1.57}
	\partial_{\theta}^{\boldsymbol\alpha}
	p_{\theta,\boldsymbol y}^{0:n}(x|\lambda )
	=
	f_{\theta,y_{n} }^{\boldsymbol\alpha}\big(x\big|\tilde{P}_{\theta,\boldsymbol y}^{n-1}(\lambda) \big),
	\;\;\;
	\tilde{P}_{\theta,\boldsymbol y}^{\boldsymbol\alpha,n}(\lambda)
	=
	F_{\theta,y_{n} }^{\boldsymbol\alpha}\big( \tilde{P}_{\theta,\boldsymbol y}^{n-1}(\lambda) \big),
	\;\;\;
	\tilde{P}_{\theta,\boldsymbol y}^{n}(\lambda)
	=
	F_{\theta,y_{n} }\big( \tilde{P}_{\theta,\boldsymbol y}^{n-1}(\lambda) \big).
\end{align}
Iterating (in $n$) the last part of (\ref{p5.1.57}), we also get
$\tilde{P}_{\theta,\boldsymbol y}^{n}(\lambda) =
F_{\theta,\boldsymbol y}^{0:n}({\cal E}_{\lambda} )$.
Combining this with (\ref{1.921}), (\ref{p5.1.31}), we deduce that the first two parts of
(\ref{p5.1.1*}) hold.

In the rest of the proof, we assume $1<|\boldsymbol\alpha|\leq p$.
Owing to (\ref{7.3}), (\ref{p5.1.49}), (\ref{p5.1.53}), we have
\begin{align}\label{p5.1.59}
	\partial_{\theta}^{\boldsymbol e}
	\Psi_{\theta}^{\boldsymbol 0}\big(y_{n}, \tilde{P}_{\theta,\boldsymbol y}^{n-1}(\lambda ) \big)
	=
	\frac{\partial_{\theta}^{\boldsymbol e}
	\big\langle
	R_{\theta,y_{n} }^{\boldsymbol 0}\big( \tilde{P}_{\theta,\boldsymbol y}^{\boldsymbol 0,n-1}(\lambda ) \big)
	\big\rangle}
	{\big\langle
	R_{\theta,y_{n} }^{\boldsymbol 0}\big( \tilde{P}_{\theta,\boldsymbol y}^{\boldsymbol 0,n-1}(\lambda ) \big)
	\big\rangle}
	=
	\big\langle
	S_{\theta,y_{n} }^{\boldsymbol e}
	\big( \tilde{P}_{\theta,\boldsymbol y}^{n-1}(\lambda) \big)
	\big\rangle
	=
	\Psi_{\theta}^{\boldsymbol e}\big(y_{n}, \tilde{P}_{\theta,\boldsymbol y}^{n-1}(\lambda ) \big),
\end{align}
where $\boldsymbol e\in\mathbb{N}_{0}^{d}$, $|\boldsymbol e|=1$.
Hence, we get
\begin{align*}
	\partial_{\theta}^{\boldsymbol e_{\boldsymbol\alpha} }
	\big\langle
	R_{\theta,y_{n} }^{\boldsymbol 0}\big( \tilde{P}_{\theta,\boldsymbol y}^{\boldsymbol 0,n-1}(\lambda ) \big)
	\big\rangle
	=
	\partial_{\theta}^{\boldsymbol e_{\boldsymbol\alpha} }
	\Psi_{\theta}^{\boldsymbol 0}\big(y_{n}, \tilde{P}_{\theta,\boldsymbol y}^{n-1}(\lambda ) \big) \:
	\big\langle
	R_{\theta,y_{n} }^{\boldsymbol 0}\big( \tilde{P}_{\theta,\boldsymbol y}^{\boldsymbol 0,n-1}(\lambda ) \big)
	\big\rangle
\end{align*}
(as $|\boldsymbol e_{\boldsymbol\alpha}|=1$).
Therefore, we have
\begin{align*}
	\partial_{\theta}^{\boldsymbol\alpha}
	\big\langle
	R_{\theta,y_{n} }^{\boldsymbol 0}\big( \tilde{P}_{\theta,\boldsymbol y}^{\boldsymbol 0,n-1}(\lambda ) \big)
	\big\rangle
	=
	\partial_{\theta}^{\boldsymbol\alpha-\boldsymbol e_{\boldsymbol\alpha} }
	\left(
	\partial_{\theta}^{\boldsymbol e_{\boldsymbol\alpha} }
	\Psi_{\theta}^{\boldsymbol 0}\big(y_{n}, \tilde{P}_{\theta,\boldsymbol y}^{n-1}(\lambda ) \big) \:
	\big\langle
	R_{\theta,y_{n} }^{\boldsymbol 0}\big( \tilde{P}_{\theta,\boldsymbol y}^{\boldsymbol 0,n-1}(\lambda ) \big)
	\big\rangle
	\right).
\end{align*}
Consequently, Leibniz rule and (\ref{p5.1.49}), (\ref{p5.1.53}) imply
\begin{align*}
	\partial_{\theta}^{\boldsymbol\alpha}
	\big\langle
	R_{\theta,y_{n} }^{\boldsymbol 0}\big( \tilde{P}_{\theta,\boldsymbol y}^{\boldsymbol 0, n-1}(\lambda ) \big)
	\big\rangle
	=&
	\sum_{\stackrel{\scriptstyle \boldsymbol\beta\in\mathbb{N}_{0}^{d} }
	{\boldsymbol\beta\leq\boldsymbol\alpha-\boldsymbol e_{\boldsymbol\alpha} } }
	\left( \boldsymbol\alpha-\boldsymbol e_{\boldsymbol\alpha} \atop \boldsymbol\beta \right)
	\partial_{\theta}^{\boldsymbol\beta+\boldsymbol e_{\boldsymbol\alpha} }
	\Psi_{\theta}^{\boldsymbol 0}\big(y_{n}, \tilde{P}_{\theta,\boldsymbol y}^{n-1}(\lambda ) \big) \:
	\partial_{\theta}^{\boldsymbol\alpha-\boldsymbol\beta-\boldsymbol e_{\boldsymbol\alpha} }
	\big\langle
	R_{\theta,y_{n} }^{\boldsymbol 0}\big( \tilde{P}_{\theta,\boldsymbol y}^{\boldsymbol 0,n-1}(\lambda ) \big)
	\big\rangle
	\\
	=&
	\sum_{\stackrel{\scriptstyle \boldsymbol\beta\in\mathbb{N}_{0}^{d} }
	{\boldsymbol e_{\boldsymbol\alpha}\leq\boldsymbol\beta\leq\boldsymbol\alpha } }
	\left( \boldsymbol\alpha-\boldsymbol e_{\boldsymbol\alpha} \atop \boldsymbol\beta-\boldsymbol e_{\boldsymbol\alpha} \right)
	\partial_{\theta}^{\boldsymbol\beta}
	\Psi_{\theta}^{\boldsymbol 0}\big(y_{n}, \tilde{P}_{\theta,\boldsymbol y}^{n-1}(\lambda ) \big) \:
	\partial_{\theta}^{\boldsymbol\alpha-\boldsymbol\beta}
	\big\langle
	R_{\theta,y_{n} }^{\boldsymbol 0}\big( \tilde{P}_{\theta,\boldsymbol y}^{\boldsymbol 0,n-1}(\lambda ) \big)
	\big\rangle
	\\
	=&
	\sum_{\stackrel{\scriptstyle \boldsymbol\beta\in\mathbb{N}_{0}^{d} }
	{\boldsymbol e_{\boldsymbol\alpha}\leq\boldsymbol\beta\leq\boldsymbol\alpha } }
	\left( \boldsymbol\alpha-\boldsymbol e_{\boldsymbol\alpha} \atop \boldsymbol\beta-\boldsymbol e_{\boldsymbol\alpha} \right)
	\partial_{\theta}^{\boldsymbol\beta}
	\Psi_{\theta}^{\boldsymbol 0}(y_{n}, \tilde{P}_{\theta,\boldsymbol y}^{n-1}(\lambda ) ) \:
	\big\langle
	S_{\theta,y_{n} }^{\boldsymbol\alpha-\boldsymbol\beta}
	\big( \tilde{P}_{\theta,\boldsymbol y}^{n-1}(\lambda) \big)
	\big\rangle
	\big\langle
	R_{\theta,y_{n} }^{\boldsymbol 0}\big( \tilde{P}_{\theta,\boldsymbol y}^{\boldsymbol 0,n-1}(\lambda ) \big)
	\big\rangle.
\end{align*}
As
$\big\langle S_{\theta,y_{n} }^{\boldsymbol 0}
\big( \tilde{P}_{\theta,\boldsymbol y}^{n-1}(\lambda) \big) \big\rangle = 1$
(due to (\ref{1.703}), (\ref{1.907})),
the same arguments then yield
\begin{align}\label{p5.1.61}
	\partial_{\theta}^{\boldsymbol\alpha}
	\Psi_{\theta}^{\boldsymbol 0}\big(y_{n}, \tilde{P}_{\theta,\boldsymbol y}^{n-1}(\lambda ) \big)
	=&
	-
	\sum_{\stackrel{\scriptstyle \boldsymbol\beta\in\mathbb{N}_{0}^{d}\setminus\{\boldsymbol\alpha\} }
	{\boldsymbol e_{\boldsymbol\alpha}\leq\boldsymbol\beta\leq\boldsymbol\alpha } }
	\left( \boldsymbol\alpha-\boldsymbol e_{\boldsymbol\alpha} \atop \boldsymbol\beta-\boldsymbol e_{\boldsymbol\alpha} \right)
	\partial_{\theta}^{\boldsymbol\beta}
	\Psi_{\theta}^{\boldsymbol 0}\big(y_{n}, \tilde{P}_{\theta,\boldsymbol y}^{n-1}(\lambda ) \big) \:
	\big\langle
	S_{\theta,y_{n} }^{\boldsymbol\alpha-\boldsymbol\beta}
	\big( \tilde{P}_{\theta,\boldsymbol y}^{n-1}(\lambda) \big)
	\big\rangle
	\nonumber\\
	&
	+
	\frac{\partial_{\theta}^{\boldsymbol\alpha}
	\big\langle
	R_{\theta,y_{n} }^{\boldsymbol 0}\big( \tilde{P}_{\theta,\boldsymbol y}^{\boldsymbol 0,n-1}(\lambda ) \big)
	\big\rangle }
	{\big\langle
	R_{\theta,y_{n} }^{\boldsymbol 0}\big( \tilde{P}_{\theta,\boldsymbol y}^{\boldsymbol 0,n-1}(\lambda ) \big)
	\big\rangle}
	\nonumber\\
	=&
	-
	\sum_{\stackrel{\scriptstyle \boldsymbol\beta\in\mathbb{N}_{0}^{d}\setminus\{\boldsymbol\alpha\} }
	{\boldsymbol e_{\boldsymbol\alpha}\leq\boldsymbol\beta\leq\boldsymbol\alpha } }
	\left( \boldsymbol\alpha-\boldsymbol e_{\boldsymbol\alpha} \atop \boldsymbol\beta-\boldsymbol e_{\boldsymbol\alpha} \right)
	\partial_{\theta}^{\boldsymbol\beta}
	\Psi_{\theta}^{\boldsymbol 0}\big(y_{n}, \tilde{P}_{\theta,\boldsymbol y}^{n-1}(\lambda ) \big) \:
	\big\langle
	S_{\theta,y_{n} }^{\boldsymbol\alpha-\boldsymbol\beta}
	\big( \tilde{P}_{\theta,\boldsymbol y}^{n-1}(\lambda) \big)
	\big\rangle
	\nonumber\\
	&
	+
	\big\langle
	S_{\theta,y_{n} }^{\boldsymbol\alpha}
	\big( \tilde{P}_{\theta,\boldsymbol y}^{n-1}(\lambda) \big)
	\big\rangle.
\end{align}
Equation (\ref{p5.1.61}) can be viewed as a recursion in $|\boldsymbol\alpha|$
which generates functions
$\big\{\partial_{\theta}^{\boldsymbol\alpha}
\Psi_{\theta}^{\boldsymbol 0}\big(y_{n}, \tilde{P}_{\theta,\boldsymbol y}^{n-1}(\lambda ) \big):
\boldsymbol\alpha\in\mathbb{N}_{0}^{d}, 1<|\boldsymbol\alpha|\leq p \big\}$.\footnote
{In (\ref{p5.1.61}), functions
$\big\{\partial_{\theta}^{\boldsymbol\alpha}
\Psi_{\theta}^{\boldsymbol 0}\big(y_{n+1}, \tilde{P}_{\theta,\boldsymbol y}^{n}(\lambda ) \big):
\boldsymbol\alpha\in\mathbb{N}_{0}^{d}, |\boldsymbol\alpha|=1 \big\}$
are the initial conditions.
At iteration $k$ of (\ref{p5.1.61}) ($1<k\leq p$),
function $\partial_{\theta}^{\boldsymbol\alpha}
\Psi_{\theta}^{\boldsymbol 0}\big(y_{n+1}, \tilde{P}_{\theta,\boldsymbol y}^{n}(\lambda ) \big)$
is computed for multi-indices $\boldsymbol\alpha\in\mathbb{N}_{0}^{d}$,
$|\boldsymbol\alpha|=k$ using the results obtained at the previous iterations.
}
Equation (\ref{p5.1.61}) can also be considered as a special case of (\ref{7.1}) ---
to get (\ref{p5.1.61}),
set $\Lambda=\tilde{P}_{\theta,\boldsymbol y}^{n-1}(\lambda )$, $y=y_{n}$ in (\ref{7.1}).
Hence, comparing (\ref{p5.1.61}) with (\ref{7.3}), (\ref{7.1}),
we conclude
\begin{align}\label{p5.1.63}
	\partial_{\theta}^{\boldsymbol\alpha}
	\Psi_{\theta}^{\boldsymbol 0}\big( y_{n}, \tilde{P}_{\theta,\boldsymbol y}^{n-1}(\lambda) \big)
	=
	\Psi_{\theta}^{\boldsymbol\alpha}\big( y_{n}, \tilde{P}_{\theta,\boldsymbol y}^{n-1}(\lambda) \big).
\end{align}
Using (\ref{p5.1.59}), (\ref{p5.1.63}), we deduce that the last part of (\ref{p5.1.1*})
holds.
\end{proof}

\begin{proof}[\rm\bf Proof of Theorem \ref{theorem1.1}]
Let $m\geq 0$ be any (fixed) integer, while $\boldsymbol y = \{y_{n} \}_{n\geq 1}$,
$\boldsymbol y' = \{y'_{n} \}_{n\geq 1}$ are any sequences in ${\cal Y}$
satisfying $y'_{n}=y_{n+m}$ for $n>m$.
Then, using (\ref{1.903}), it is straightforward to verify
$p_{\theta,\boldsymbol y}^{m:n}(x|\lambda)=p_{\theta,\boldsymbol y'}^{0:n-m}(x|\lambda)$
for $\theta\in\Theta$, $x\in{\cal X}$, $\lambda\in{\cal P}({\cal X} )$, $n>m$.
Consequently, Proposition \ref{proposition5.1} implies that
(\ref{t1.1.1*}) holds for the same $\theta$, $x$, $\lambda$, $n,m$
and $B\in{\cal B}({\cal X} )$,
$\boldsymbol\alpha\in\mathbb{N}_{0}^{d}$,
$|\boldsymbol\alpha|\leq p$.
\end{proof}

\begin{lemma}\label{lemma3.1}
Let Assumptions \ref{a1.1} and \ref{a2.1} hold.
Then, there exists a real number $C_{8}\in[1,\infty )$
(depending only on $\varepsilon$) such that
\begin{align*}
	&
	\left|
	\Psi_{\theta}^{\boldsymbol 0 }(y,\Lambda)
	\right|
	\leq
	C_{8}
	\varphi(y),
	\;\;\;\;\;
	\left|
	\Psi_{\theta}^{\boldsymbol 0}(y,\Lambda)
	-
	\Psi_{\theta}^{\boldsymbol 0}(y,\Lambda')
	\right|
	\leq
	C_{8} \varphi(y) \|\Lambda-\Lambda' \|
\end{align*}
for all $\theta\in\Theta$, $y\in{\cal Y}$,
$\Lambda,\Lambda'\in{\cal L}_{0}({\cal X} )$.
\end{lemma}

\begin{proof}
Throughout the proof, the following notation is used.
$\tilde{C}$ is the real number defined by $\tilde{C}=1+|\log\varepsilon|$,
while $C_{8}$ is the real number defined by $C_{8}=2C_{1}\tilde{C}$
($\varepsilon$, $C_{1}$ are specified in Assumption \ref{a1.1} and Lemma \ref{lemma1.2}).
$\theta$, $y$ are any elements of $\Theta$, ${\cal Y}$ (respectively).
$\lambda,\lambda'$ are any elements of ${\cal P}({\cal X} )$, while
$\Lambda=\left\{\lambda_{\boldsymbol\alpha}:
\boldsymbol\alpha\in\mathbb{N}_{0}^{d}, |\boldsymbol\alpha|\leq p \right\}$,
$\Lambda'=\left\{\lambda'_{\boldsymbol\alpha}:
\boldsymbol\alpha\in\mathbb{N}_{0}^{d}, |\boldsymbol\alpha|\leq p \right\}$
are any elements of ${\cal L}_{0}({\cal X} )$.

Relying on Assumption \ref{a1.1}, we conclude
\begin{align*}
	\varepsilon\mu_{\theta}({\cal X}|y)
	\leq
	\int\left(\int r_{\theta}(y,x'|x) \mu(dx') \right) \lambda(dx)
	\leq
	\frac{\mu_{\theta}({\cal X}|y) }{\varepsilon}.
\end{align*}
Consequently, Assumption \ref{a2.1} and (\ref{1.907}) imply
\begin{align*}
	\left|
	\log\left(
	\big\langle R_{\theta,y}^{\boldsymbol 0}(\lambda )\big\rangle
	\right)
	\right|
	\leq
	|\log\varepsilon|
	+
	\left|
	\log\mu_{\theta}({\cal X}|y)
	\right|
	\leq
	\tilde{C} + \varphi(y)
	\leq
	2\tilde{C}\varphi(y).
\end{align*}
Therefore, (\ref{7.3}) yields
\begin{align*}
	\left|
	\Psi_{\theta}^{\boldsymbol 0}(y,\Lambda )
	\right|
	= &
	\left|
	\log\left(
	\big\langle R_{\theta,y}^{\boldsymbol 0}(\lambda_{\boldsymbol 0} )\big\rangle
	\right)
	\right|
	\leq
	2\tilde{C}\varphi(y)
	\leq 
	C_{8} \varphi(y).
\end{align*}
Moreover, using Lemma \ref{lemma1.2}, we deduce
\begin{align*}
	\left|
	\frac{\big\langle R_{\theta,y}^{\boldsymbol 0}(\lambda) \big\rangle }
	{\big\langle R_{\theta,y}^{\boldsymbol 0}(\lambda') \big\rangle}
	-
	1
	\right|
	=
	\left|
	\frac{\big\langle R_{\theta,y}^{\boldsymbol 0}(\lambda-\lambda') \big\rangle }
	{\big\langle R_{\theta,y}^{\boldsymbol 0}(\lambda') \big\rangle}
	\right|
	\leq
	\frac{\left\|\big\langle R_{\theta,y}^{\boldsymbol 0}(\lambda-\lambda') \big\rangle \right\| }
	{\big\langle R_{\theta,y}^{\boldsymbol 0}(\lambda') \big\rangle}
	\leq
	C_{1}	\left\|\lambda-\lambda'\right\|.
\end{align*}
Consequently, we have
\begin{align}\label{l3.1.1}
	\log\left(
	\frac{\big\langle R_{\theta,y}^{\boldsymbol 0}(\lambda) \big\rangle }
	{\big\langle R_{\theta,y}^{\boldsymbol 0}(\lambda') \big\rangle}
	\right)
	\leq
	\left|
	\frac{\big\langle R_{\theta,y}^{\boldsymbol 0}(\lambda) \big\rangle }
	{\big\langle R_{\theta,y}^{\boldsymbol 0}(\lambda') \big\rangle}
	-
	1
	\right|
	\leq
	C_{1}	\left\|\lambda-\lambda'\right\|.
\end{align}
Reverting the roles of $\lambda$, $\lambda'$, we get
\begin{align}\label{l3.1.3}
	-\log\left(
	\frac{\big\langle R_{\theta,y}^{\boldsymbol 0}(\lambda) \big\rangle }
	{\big\langle R_{\theta,y}^{\boldsymbol 0}(\lambda') \big\rangle}
	\right)
	=
	\log\left(
	\frac{\big\langle R_{\theta,y}^{\boldsymbol 0}(\lambda') \big\rangle }
	{\big\langle R_{\theta,y}^{\boldsymbol 0}(\lambda) \big\rangle}
	\right)
	\leq
	\left|
	\frac{\big\langle R_{\theta,y}^{\boldsymbol 0}(\lambda') \big\rangle }
	{\big\langle R_{\theta,y}^{\boldsymbol 0}(\lambda) \big\rangle}
	-
	1
	\right|
	\leq
	C_{1} \left\|\lambda-\lambda'\right\|.
\end{align}
Owing to (\ref{l3.1.1}), (\ref{l3.1.3}), we have
\begin{align*}
	\left|
	\log\left(
	\frac{\big\langle R_{\theta,y}^{\boldsymbol 0}(\lambda) \big\rangle }
	{\big\langle R_{\theta,y}^{\boldsymbol 0}(\lambda') \big\rangle}
	\right)
	\right|
	\leq
	C_{1}	\left\|\lambda-\lambda'\right\|.
\end{align*}
Hence, we get
\begin{align*}
	\left|
	\Psi_{\theta}^{\boldsymbol 0}(y,\Lambda )
	-
	\Psi_{\theta}^{\boldsymbol 0}(y,\Lambda' )
	\right|
	= &
	\left|
	\log\left(
	\frac{\big\langle R_{\theta,y}^{\boldsymbol 0}(\lambda_{\boldsymbol 0} ) \big\rangle }
	{\big\langle R_{\theta,y}^{\boldsymbol 0}(\lambda'_{\boldsymbol 0} ) \big\rangle}
	\right)
	\right|
	\leq 
	C_{1} \left\|\lambda_{\boldsymbol 0} - \lambda'_{\boldsymbol 0} \right\|
	\leq 
	C_{8} \varphi(y) \left\|\Lambda - \Lambda' \right\|
\end{align*}
(as $\varphi(y)\geq 1$,
$\left\|\lambda_{\boldsymbol 0} - \lambda'_{\boldsymbol 0} \right\| \leq
\left\|\Lambda - \Lambda' \right\|$).
\end{proof}

\begin{lemma}\label{lemma3.2}
Let Assumptions \ref{a1.1} and \ref{a1.2} hold.
Then, there exists a real number $C_{9}\in[1,\infty )$ (depending only on $\varepsilon$, $p$) such that
\begin{align*}
	\left|\Psi_{\theta}^{\boldsymbol\alpha}(y,\Lambda) \right|
	\leq
	C_{9} \big(\psi(y) \|\Lambda \| \big)^{p},
	\;\;\;\;\;
	\left|
	\Psi_{\theta}^{\boldsymbol\alpha}(y,\Lambda)
	-
	\Psi_{\theta}^{\boldsymbol\alpha}(y,\Lambda')
	\right|
	\leq
	C_{9} \|\Lambda-\Lambda' \|
	\big(\psi(y) \left(\|\Lambda \| + \|\Lambda' \|\right) \big)^{p}
\end{align*}
for all $\theta\in\Theta$, $y\in{\cal Y}$,
$\Lambda,\Lambda'\in{\cal L}_{0}({\cal X} )$
and any multi-index $\boldsymbol\alpha\in\mathbb{N}_{0}^{d}\setminus\{\boldsymbol 0\}$,
$|\boldsymbol\alpha|\leq p$.
\end{lemma}

\begin{proof}
Throughout the proof, the following notation is used.
$\theta$, $y$ are any elements of $\Theta$, ${\cal Y}$ (respectively).
$\tilde{C}_{1}$, $\tilde{C}_{2}$ are the real numbers defined by
$\tilde{C}_{1}=2^{p}C_{2}$, $\tilde{C}_{2}=3\tilde{C}_{1}^{2}$,
while $C_{9}$ is the real number defined by $C_{9}=\exp(\tilde{C}_{2}p)$
($C_{2}$ is specified in Lemma \ref{lemma1.4}).
$B_{\boldsymbol\alpha}$ is the real number defined by
$B_{\boldsymbol\alpha}=\exp\big(\tilde{C}_{2}|\boldsymbol\alpha| \big)$
for $\boldsymbol\alpha\in\mathbb{N}_{0}^{d}$.

Let $\boldsymbol\gamma$ be any element of $\mathbb{N}_{0}^{d}\setminus\{\boldsymbol 0\}$
satisfying $|\boldsymbol\gamma|\leq p$.
Then, it easy to conclude $B_{\boldsymbol\gamma}\leq\exp(\tilde{C}_{2})\leq 3\tilde{C}_{1}^{2}$.
Consequently, Lemma \ref{lemma1.4} and (\ref{1.907}) imply
\begin{align}\label{l3.2.3}
	\left|
	\big\langle S_{\theta,y}^{\boldsymbol\gamma}(\Lambda) \big\rangle
	-
	\big\langle S_{\theta,y}^{\boldsymbol\gamma}(\Lambda') \big\rangle
	\right|
	\leq &
	C_{2}\sum_{\stackrel{\scriptstyle \boldsymbol\delta\in\mathbb{N}_{0}^{d}}
	{\boldsymbol\delta\leq\boldsymbol\gamma } }
	(\psi(y) )^{|\boldsymbol\gamma-\boldsymbol\delta|}
	(\|\lambda_{\boldsymbol\delta}-\lambda'_{\boldsymbol\delta}\|
	+
	\|\lambda_{\boldsymbol 0}-\lambda'_{\boldsymbol 0}\| \|\lambda'_{\boldsymbol\delta}\| )
	\nonumber\\
	\leq &
	2^{|\boldsymbol\gamma|} C_{2} \left(\psi(y) \right)^{|\boldsymbol\gamma|}
	\|\Lambda - \Lambda' \| (\|\Lambda\| + \|\Lambda' \| )
	\nonumber\\
	\leq &
	\tilde{C}_{1} \|\Lambda - \Lambda' \|
	\big(\psi(y) (\|\Lambda\| + \|\Lambda' \| ) \big)^{|\boldsymbol\gamma|}
	\nonumber\\
	\leq &
	\frac{B_{\boldsymbol\gamma} \|\Lambda - \Lambda' \|
	\big(\psi(y) (\|\Lambda\| + \|\Lambda' \| ) \big)^{|\boldsymbol\gamma|} }
	{3\tilde{C}_{1} }
\end{align}
for $\Lambda=\left\{\lambda_{\boldsymbol\delta}:
\boldsymbol\delta\in\mathbb{N}_{0}^{d},|\boldsymbol\delta|\leq p\right\}\in{\cal L}_{0}({\cal X} )$,
$\Lambda'=\left\{\lambda'_{\boldsymbol\delta}:
\boldsymbol\delta\in\mathbb{N}_{0}^{d},|\boldsymbol\delta|\leq p\right\}\in{\cal L}_{0}({\cal X} )$
(as $\|\Lambda\|\geq 1$, $\|\Lambda'\|\geq \|\lambda'_{\boldsymbol\delta}\|$,
$\|\Lambda-\Lambda'\|\geq\|\lambda_{\boldsymbol\delta}-\lambda'_{\boldsymbol\delta}\|$).
The same arguments yield
\begin{align}\label{l3.2.1}
	\left|\big\langle S_{\theta,y}^{\boldsymbol\gamma}(\Lambda) \big\rangle \right|
	\leq
	C_{2}\sum_{\stackrel{\scriptstyle \boldsymbol\delta\in\mathbb{N}_{0}^{d}}
	{\boldsymbol\delta\leq\boldsymbol\gamma } }
	(\psi(y) )^{|\boldsymbol\gamma-\boldsymbol\delta|}
	\|\lambda_{\boldsymbol\delta}\|
	\leq
	2^{|\boldsymbol\gamma|}C_{2}(\psi(y) )^{|\boldsymbol\gamma|}\|\Lambda\|
	\leq
	\tilde{C}_{1} \big(\psi(y) \|\Lambda \| \big)^{|\boldsymbol\gamma|}
	\leq
	\frac{B_{\boldsymbol\alpha} \big(\psi(y) \|\Lambda \| \big)^{|\boldsymbol\gamma|} }
	{3\tilde{C}_{1} }
\end{align}
for the same $\Lambda$.

To prove the lemma, it is sufficient to show
\begin{align}\label{l3.2.5}
	\big|\Psi_{\theta}^{\boldsymbol\alpha}(y,\Lambda) \big|
	\leq
	B_{\boldsymbol\alpha} \big(\psi(y) \|\Lambda \| \big)^{|\boldsymbol\alpha|},
	\;\;\;\;\;
	\big|
	\Psi_{\theta}^{\boldsymbol\alpha}(y,\Lambda)
	-
	\Psi_{\theta}^{\boldsymbol\alpha}(y,\Lambda')
	\big|
	\leq
	B_{\boldsymbol\alpha} \|\Lambda-\Lambda' \|
	\big(\psi(y) \left(\|\Lambda \| + \|\Lambda' \|\right) \big)^{|\boldsymbol\alpha|}
\end{align}
for $\Lambda,\Lambda'\in{\cal L}_{0}({\cal X} )$,
$\boldsymbol\alpha\in\mathbb{N}_{0}^{d}\setminus\{\boldsymbol 0\}$,
$|\boldsymbol\alpha|\leq p$.
We prove (\ref{l3.2.5}) by mathematical induction in
$|\boldsymbol\alpha|$.
When $|\boldsymbol\alpha|=1$, (\ref{l3.2.3}), (\ref{l3.2.1}) imply that
(\ref{l3.2.5}) is true for all $\Lambda,\Lambda'\in{\cal L}_{0}({\cal X} )$.
Now, the induction hypothesis is formulated:
Suppose that (\ref{l3.2.5}) holds for some $l\in\mathbb{N}_{0}^{d}$
and all $\Lambda,\Lambda'\in{\cal L}_{0}({\cal X} )$,
$\boldsymbol\alpha\in\mathbb{N}_{0}^{d}$ satisfying $1\leq l<p$, $|\boldsymbol\alpha|\leq l$.
Then, to prove (\ref{l3.2.5}),
it is sufficient to show (\ref{l3.2.5})
for all $\Lambda,\Lambda'\in{\cal L}_{0}({\cal X} )$,
$\boldsymbol\alpha\in\mathbb{N}_{0}^{d}$ satisfying $|\boldsymbol\alpha|= l+1$.
In what follows in the proof, $\Lambda,\Lambda'$ are any elements of ${\cal L}_{0}({\cal X} )$.
$\boldsymbol\alpha$ is any element of $\mathbb{N}_{0}^{d}$
satisfying $|\boldsymbol\alpha|= l+1$,
while $\boldsymbol\beta$ is any element of $\mathbb{N}_{0}^{d}\setminus\{\boldsymbol 0,\boldsymbol \alpha\}$
fulfilling $\boldsymbol\beta\leq\boldsymbol\alpha$.

Since $\boldsymbol\beta\leq\boldsymbol\alpha$, $\boldsymbol\beta\neq\boldsymbol 0$,
$\boldsymbol\beta\neq\boldsymbol\alpha$,
we have $1\leq |\boldsymbol\beta|\leq|\boldsymbol\alpha|-1=l$.
Then, owing to the induction hypothesis, we have
\begin{align}
	&\label{l3.2.21}
	\max\left\{
	\frac{\big|\Psi_{\theta}^{\boldsymbol\beta}(y,\Lambda ) \big| }
	{ \big(\psi(y)\|\Lambda\| \big)^{|\boldsymbol\beta|} },
	\frac{\big|\Psi_{\theta}^{\boldsymbol\beta}(y,\Lambda' ) \big| }
	{ \big(\psi(y)\|\Lambda'\| \big)^{|\boldsymbol\beta|} }
	\right\}
	\leq
	B_{\boldsymbol\beta}
	\leq
	\frac{B_{\boldsymbol\alpha} }{3\tilde{C}_{1}^{2} },
	\\
	&\label{l3.2.23}
	\frac{\big|\Psi_{\theta}^{\boldsymbol\beta}(y,\Lambda) - \Psi_{\theta}^{\boldsymbol\beta}(y,\Lambda') \big| }
	{ \big(\psi(y) (\|\Lambda\|+\|\Lambda'\|) \big)^{|\boldsymbol\beta|} }
	\leq
	B_{\boldsymbol\beta}\|\Lambda-\Lambda'\|
	\leq
	\frac{B_{\boldsymbol\alpha} \|\Lambda-\Lambda'\| }{3\tilde{C}_{1}^{2} }.
\end{align}
Consequently, (\ref{l3.2.1}) implies
\begin{align}\label{l3.2.25}
	\big|
	\Psi_{\theta}^{\boldsymbol\beta}(y,\Lambda)
	\big|\:
	\big|
	\big\langle S_{\theta,y}^{\boldsymbol\alpha-\boldsymbol\beta}(\Lambda ) \big\rangle
	\big|
	\leq
	\frac{B_{\boldsymbol\beta}B_{\boldsymbol\alpha-\boldsymbol\beta} \big(\psi(y)\|\Lambda\| \big)^{|\boldsymbol\alpha|} }
	{3\tilde{C}_{1} }
	\leq
	\frac{B_{\boldsymbol\alpha} \big(\psi(y)\|\Lambda\| \big)^{|\boldsymbol\alpha|} }
	{3\tilde{C}_{1} }
\end{align}
(as $|\boldsymbol\alpha|=|\boldsymbol\beta|+|\boldsymbol\alpha-\boldsymbol\beta|$).
Similarly, (\ref{l3.2.3}), (\ref{l3.2.1}), (\ref{l3.2.21}), (\ref{l3.2.23}) yield
\begin{align}
	&\label{l3.2.27}
	\big|\Psi_{\theta}^{\boldsymbol\beta}(y,\Lambda) - \Psi_{\theta}^{\boldsymbol\beta}(y,\Lambda') \big| \:
	\big|
	\big\langle S_{\theta,y}^{\boldsymbol\alpha-\boldsymbol\beta}(\Lambda') \big\rangle
	\big|
	\leq
	\frac{B_{\boldsymbol\alpha} \|\Lambda-\Lambda'\| \big(\psi(y)(\|\Lambda\|+\|\Lambda'\| ) \big)^{|\boldsymbol\alpha|} }
	{3\tilde{C}_{1} },
	\\
	&\label{l3.2.29}
	\big|\Psi_{\theta}^{\boldsymbol\beta}(y,\Lambda) \big| \:
	\big|
	\big\langle S_{\theta,y}^{\boldsymbol\alpha-\boldsymbol\beta}(\Lambda) \big\rangle
	-
	\big\langle S_{\theta,y}^{\boldsymbol\alpha-\boldsymbol\beta}(\Lambda') \big\rangle
	\big|
	\leq
	\frac{B_{\boldsymbol\alpha} \|\Lambda-\Lambda'\| \big(\psi(y)(\|\Lambda\|+\|\Lambda'\| ) \big)^{|\boldsymbol\alpha|} }
	{3\tilde{C}_{1} }.
\end{align}
Using (\ref{7.1}), (\ref{l3.2.1}), (\ref{l3.2.25}), we conclude
\begin{align*}
	\big|
	\Psi_{\theta}^{\boldsymbol\alpha}(y,\Lambda)
	\big|
	\leq &
	\big|
	\big\langle S_{\theta,y}^{\boldsymbol\alpha}(\Lambda ) \big\rangle
	\big|
	+
	\sum_{\stackrel{\scriptstyle\boldsymbol\beta\in\mathbb{N}_{0}^{d}\setminus\{\boldsymbol\alpha \} }
	{\boldsymbol e_{\boldsymbol\alpha} \leq \boldsymbol\beta \leq \boldsymbol\alpha } }
	\left( \boldsymbol\alpha - \boldsymbol e_{\boldsymbol\alpha} \atop \boldsymbol\beta - \boldsymbol e_{\boldsymbol\alpha} \right)
	\big|
	\Psi_{\theta}^{\boldsymbol\beta}(y,\Lambda)
	\big|\:
	\big|
	\big\langle S_{\theta,y}^{\boldsymbol\alpha-\boldsymbol\beta}(\Lambda ) \big\rangle
	\big|
	\\
	\leq &
	\frac{2^{|\boldsymbol\alpha|} B_{\boldsymbol\alpha}
	\big(\psi(y) \|\Lambda \| \big)^{|\boldsymbol\alpha|} }
	{\tilde{C}_{1} }
	\\
	\leq &
	B_{\boldsymbol\alpha} \big(\psi(y) \|\Lambda \| \big)^{|\boldsymbol\alpha|}.
\end{align*}
(as $\tilde{C}_{1}\geq 2^{|\boldsymbol\alpha|}$).
Relying on (\ref{7.1}), (\ref{l3.2.3}), (\ref{l3.2.27}), (\ref{l3.2.29}), we deduce
\begin{align*}
	\big|\Psi_{\theta}^{\boldsymbol\alpha}(y,\Lambda)
	-
	\Psi_{\theta}^{\boldsymbol\alpha}(y,\Lambda') \big|
	\leq &
	\sum_{\stackrel{\scriptstyle\boldsymbol\beta\in\mathbb{N}_{0}^{d}\setminus\{\boldsymbol\alpha \} }
	{\boldsymbol e_{\boldsymbol\alpha} \leq \boldsymbol\beta \leq \boldsymbol\alpha } }
	\left( \boldsymbol\alpha - \boldsymbol e_{\boldsymbol\alpha} \atop \boldsymbol\beta - \boldsymbol e_{\boldsymbol\alpha} \right)
	\big|\Psi_{\theta}^{\boldsymbol\beta}(y,\Lambda) \big| \:
	\big|
	\big\langle S_{\theta,y}^{\boldsymbol\alpha-\boldsymbol\beta}(\Lambda) \big\rangle
	-
	\big\langle S_{\theta,y}^{\boldsymbol\alpha-\boldsymbol\beta}(\Lambda') \big\rangle
	\big|
	\\
	&
	+
	\sum_{\stackrel{\scriptstyle\boldsymbol\beta\in\mathbb{N}_{0}^{d}\setminus\{\boldsymbol\alpha \} }
	{\boldsymbol e_{\boldsymbol\alpha} \leq \boldsymbol\beta \leq \boldsymbol\alpha } }
	\left( \boldsymbol\alpha - \boldsymbol e_{\boldsymbol\alpha} \atop \boldsymbol\beta - \boldsymbol e_{\boldsymbol\alpha} \right)
	\big|
	\Psi_{\theta}^{\boldsymbol\beta}(y,\Lambda)
	-
	\Psi_{\theta}^{\boldsymbol\beta}(y,\Lambda')
	\big| \:
	\big|
	\big\langle S_{\theta,y}^{\boldsymbol\alpha-\boldsymbol\beta}(\Lambda') \big\rangle
	\big|
	\\
	&
	+
	\big|
	\big\langle S_{\theta,y}^{\boldsymbol\alpha}(\Lambda) \big\rangle
	-
	\big\langle S_{\theta,y}^{\boldsymbol\alpha}(\Lambda') \big\rangle
	\big|
	\\
	\leq &
	\frac{2^{|\boldsymbol\alpha|} B_{\boldsymbol\alpha} \|\Lambda-\Lambda' \|
	\big(\psi(y) (\|\Lambda\| + \|\Lambda'\| ) \big)^{|\boldsymbol\alpha|} }
	{\tilde{C}_{1} }
	\\
	\leq &
	B_{\boldsymbol\alpha} \|\Lambda-\Lambda' \|
	\big(\psi(y) (\|\Lambda\| + \|\Lambda'\| ) \big)^{|\boldsymbol\alpha|}.
\end{align*}
Hence, (\ref{l3.2.5}) holds for $\boldsymbol\alpha\in\mathbb{N}_{0}^{d}$,
$|\boldsymbol\alpha|=l+1$.
Then, the lemma directly follows by the principle of mathematical induction.
\end{proof}

\begin{proof}[\rm\bf Proof of Theorem \ref{theorem2.1}]
Let $w=p(p+1)$. Using Theorem \ref{theorem1.3} and Lemmas \ref{lemma3.1}, \ref{lemma3.2},
we conclude that for each multi-index $\boldsymbol\alpha\in\mathbb{N}_{0}^{d}$, $|\boldsymbol\alpha|\leq p$,
there exists a function $\psi_{\theta}^{\boldsymbol\alpha}$
which maps $\theta$ to $\mathbb{R}$ and satisfies
\begin{align}\label{t2.1.1}
	\psi_{\theta}^{\boldsymbol\alpha}
	=
	\lim_{n\rightarrow\infty }
	(\tilde{\Pi}^{n} \Psi^{\boldsymbol\alpha} )_{\theta}(x,y,\Lambda )
\end{align}
for $\theta\in\Theta$, $x\in{\cal X}$, $y\in{\cal Y}$, $\Lambda\in{\cal L}_{0}({\cal X} )$.
Relying on the same arguments, we deduce that
there also exist real numbers $\rho\in(0,1)$, $\tilde{C}_{1}\in[1,\infty )$
(depending only on $\varepsilon$, $\delta$, $K_{0}$, $M_{0}$)
such that
\begin{align}\label{t2.1.3}
	\left|
	(\tilde{\Pi}^{n}\Psi^{\boldsymbol\alpha} )_{\theta}(x,y,\Lambda )
	-
	\psi_{\theta}^{\boldsymbol\alpha}
	\right|
	\leq
	\tilde{C}_{1}\rho^{n}\psi^{u}(y) \|\Lambda \|^{w}
\end{align}
for the same $\theta$, $x$, $y$, $\Lambda$ and
$n\geq 1$,
$\boldsymbol\alpha\in\mathbb{N}_{0}^{d}$,
$|\boldsymbol\alpha|\leq p$
($u$ is specified in Assumption \ref{a2.1}).

Throughout the rest of the proof, the following notation is used.
$\theta$ is any element of $\Theta$,
while $x$, $y$, $\lambda$  are any elements of ${\cal X}$, ${\cal Y}$, ${\cal P}({\cal X} )$ (respectively).
$\boldsymbol\alpha$ is any element of $\mathbb{N}_{0}^{d}$
satisfying $|\boldsymbol\alpha|\leq p$.
$n$ is any (strictly) positive integer.

Owing to Assumption \ref{a2.1}, we have
\begin{align}\label{t2.1.5}
	\max\left\{
	E\left( \varphi(Y_{n} ) \right), E\left( \psi^{u}(Y_{n} ) \right)
	\right\}
	\leq
	E\left( \varphi(Y_{n} )\psi^{u}(Y_{n} ) \right)
	=
	E\left( \int \varphi(y)\psi^{u}(y) Q(X_{n},dy) \right)
	\leq
	M_{0}.
\end{align}
Due to the same assumption, we also have
\begin{align*}
	&
	E\left(\left.
	\psi^{p}(Y_{k} ) \psi^{u}(Y_{1} )
	\right|X_{1}=x,Y_{1}=y
	\right)
	=
	\psi^{u}(y)
	E\left(
	\int \psi^{p}(y) Q(X_{k},dy)
	\right)
	\leq
	M_{0} \psi^{u}(y),
	\\
	&
	E\left(\left.
	\psi^{p}(Y_{l} ) \psi^{u}(Y_{k} )
	\right|X_{1}=x,Y_{1}=y
	\right)
	=
	E\left(
	\int \psi^{p}(y) Q(X_{l},dy) \int \psi^{u}(y) Q(X_{k},dy)
	\right)
	\leq
	M_{0}^{2}
\end{align*}
for $l>k>1$.
Therefore, we get
\begin{align}\label{t2.1.7}
	E\left(\left.
	\psi^{p}(Y_{n+1} )
	\sum_{k=1}^{n} \psi^{u}(Y_{k} )
	\right|X_{1}=x,Y_{1}=y
	\right)
	\leq
	M_{0}^{2}n + M_{0}\psi^{u}(y)
	<\infty.
\end{align}

Using (\ref{1.903}), (\ref{1.705}), (\ref{1.915}), (\ref{7.3}), (\ref{7.1}), it is straightforward to verify
\begin{align*}
	\log q_{\theta}^{n}(Y_{1:n} |\lambda )
	=&
	\sum_{k=1}^{n-1}
	\log\left(
	\int\int r_{\theta}(Y_{k+1},x''|x') p_{\theta,\boldsymbol Y}^{0:k}(x'|\lambda )
	\mu(dx'') \mu(dx')
	\right)
	\\
	&+
	\log\left(\int\int r_{\theta}(Y_{1},x'|x) \mu(dx')\lambda(dx) \right)
	\\
	=&
	\sum_{k=0}^{n-1}
	\Psi_{\theta}^{\boldsymbol 0}(Y_{k+1}, F_{\theta,\boldsymbol Y}^{0:k}({\cal E}_{\lambda} ) )
\end{align*}
(here, $\boldsymbol Y$ denotes stochastic process $\{Y_{n} \}_{n\geq 1}$,
i.e., $\boldsymbol Y = \{Y_{n} \}_{n\geq 1}$).
It is also easy to show
\begin{align*}
	\big(\tilde{\Pi}^{n}\Psi^{\boldsymbol\alpha} \big)_{\theta}(x,y,{\cal E}_{\lambda} )
	=
	E\left(\left.
	\Psi_{\theta}^{\boldsymbol\alpha}(Y_{n+1},F_{\theta,\boldsymbol Y}^{0:n}({\cal E}_{\lambda} ) )
	\right|X_{1}=x,Y_{1}=y
	\right).
\end{align*}
Therefore, we have
\begin{align}\label{t2.1.11}
	E\left(\left.
	\log q_{\theta}^{n}(Y_{1:n} |\lambda )
	\right|X_{1}=x,Y_{1}=y
	\right)
	=
	\sum_{k=1}^{n-1}
	\big(\tilde{\Pi}^{k}\Psi^{\boldsymbol 0} \big)_{\theta}(x,y,{\cal E}_{\lambda} )
	+
	\Psi^{\boldsymbol 0}_{\theta}(y,{\cal E}_{\lambda} ).
\end{align}
Consequently, Lemma \ref{lemma3.1} and (\ref{t2.1.3}) imply
\begin{align*}
	\left|
	E\left(\left.
	\frac{1}{n}
	\log q_{\theta}^{n}(Y_{1:n} |\lambda )
	\right|X_{1}=x,Y_{1}=y
	\right)
	-
	\psi_{\theta}^{\boldsymbol 0}
	\right|
	\leq &
	\frac{1}{n} \sum_{k=1}^{n-1}
	\left|
	\big(\tilde{\Pi}^{k}\Psi^{\boldsymbol 0} \big)_{\theta}(x,y,{\cal E}_{\lambda} )
	-
	\psi_{\theta}^{\boldsymbol 0}
	\right|
	+
	\frac{\big|\psi_{\theta}^{\boldsymbol 0}\big| + \big|\Psi^{\boldsymbol 0}_{\theta}(y,{\cal E}_{\lambda} )\big|}
	{n}
	\\
	\leq &
	\frac{\tilde{C}_{1} \psi^{u}(y) }{n}
	\sum_{k=1}^{n-1} \rho^{k}
	+
	\frac{|\psi_{\theta}^{\boldsymbol 0} | + C_{8} \phi(y) }{n}
	\\
	\leq &
	\frac{\tilde{C}_{1} \psi^{u}(y) }{n(1-\rho ) }
	+
	\frac{|\psi_{\theta}^{\boldsymbol 0} | + C_{8} \phi(y) }{n}.
\end{align*}
Then, (\ref{t2.1.5}) yields
\begin{align*}
	\left|
	E\left(
	\frac{1}{n}
	\log q_{\theta}^{n}(Y_{1:n} |\lambda )
	\right)
	-
	\psi_{\theta}^{\boldsymbol 0}
	\right|
	&\leq
	E\left(
	\left|
	E\left(\left.
	\frac{1}{n}
	\log q_{\theta}^{n}(Y_{1:n} |\lambda )
	\right|X_{1},Y_{1}
	\right)
	-
	\psi_{\theta}^{\boldsymbol 0}
	\right|
	\right)
	\\
	&\leq
	\frac{\tilde{C}_{1} E(\psi^{u}(Y_{1} ) ) }{n(1-\rho ) }
	+
	\frac{|\psi_{\theta}^{\boldsymbol 0} | + C_{8} E(\phi(Y_{1} ) ) }{n}
	\\
	&\leq
	\frac{\tilde{C}_{1} M_{0} }{n(1-\rho ) }
	+
	\frac{|\psi_{\theta}^{\boldsymbol 0} | + C_{8} M_{0} }{n}.
\end{align*}
Therefore, we get
\begin{align}\label{t2.1.9}
	\lim_{n\rightarrow\infty}
	E\left(
	\frac{1}{n}
	\log q_{\theta}^{n}(Y_{1:n} |\lambda )
	\right)
	=
	\psi_{\theta}^{\boldsymbol 0}.
\end{align}

Let $\tilde{C}_{2}=\max\{A_{\boldsymbol\alpha}: \boldsymbol\alpha\in\mathbb{N}_{0}^{d},
|\boldsymbol\alpha|\leq p\}$
($A_{\boldsymbol\alpha}$ is specified in Proposition \ref{proposition1.3}).
Owing to Proposition \ref{proposition1.3} and Lemma \ref{lemma3.2}, we have
\begin{align*}
	\left|
	\Psi_{\theta}^{\boldsymbol\alpha}(Y_{n+1},F_{\theta,\boldsymbol Y}^{0:n}({\cal E}_{\lambda} ) )
	\right|
	\leq 
	C_{9} \psi^{p}(Y_{n+1} ) \|F_{\theta,\boldsymbol Y}^{0:n}({\cal E}_{\lambda} ) \|^{p}
	\leq &
	\tilde{C}_{2}^{p} C_{9} \psi^{p}(Y_{n+1} ) \left(\Psi_{\boldsymbol Y}^{0:n} \right)^{u}
	\\
	\leq &
	\tilde{C}_{2}^{p} C_{9} n^{u} \psi^{p}(Y_{n+1} )
	\sum_{k=1}^{n} \psi^{u}(Y_{k} )
\end{align*}
(as $\Psi_{\boldsymbol Y}^{0:n}\geq 1$, $u>p^{2}$).
Consequently, Proposition \ref{proposition5.1}, Lemma \ref{lemmaa3}
and (\ref{t2.1.7}), (\ref{t2.1.11}) imply that
\linebreak
$\big(\tilde{\Pi}^{n}\Psi^{\boldsymbol 0} \big)_{\theta}(x,y,{\cal E}_{\lambda} )$
is $p$-times differentiable in $\theta$ and satisfies
\begin{align*}
	\partial_{\theta}^{\boldsymbol\alpha}
	\big(\tilde{\Pi}^{n}\Psi^{\boldsymbol 0} \big)_{\theta}(x,y,{\cal E}_{\lambda} )
	=&
	E\left(\left.
	\partial_{\theta}^{\boldsymbol\alpha}
	\Psi_{\theta}^{\boldsymbol 0}(Y_{n+1},F_{\theta,\boldsymbol Y}^{0:n}({\cal E}_{\lambda} ) )
	\right|X_{1}=x,Y_{1}=y
	\right)
	\nonumber\\
	=&
	E\left(\left.
	\Psi_{\theta}^{\boldsymbol\alpha}(Y_{n+1},F_{\theta,\boldsymbol Y}^{0:n}({\cal E}_{\lambda} ) )
	\right|X_{1}=x,Y_{1}=y
	\right)
	\nonumber\\
	=&
	\big(\tilde{\Pi}^{n}\Psi^{\boldsymbol\alpha} \big)_{\theta}(x,y,{\cal E}_{\lambda} ).
\end{align*}
Then, the uniform convergence theorem and (\ref{t2.1.3}) yield that
$\psi_{\theta}^{\boldsymbol 0}$ is $p$-times differentiable in $\theta$
and satisfies
$\partial_{\theta}^{\boldsymbol\alpha} \psi_{\theta}^{\boldsymbol 0} =
\psi_{\theta}^{\boldsymbol\alpha}$.
Combining this with (\ref{t2.1.9}), we conclude that there exists
function $l(\theta)$ with the properties specified in the statement of the theorem.
\end{proof}

\section{Proof of Corollaries \ref{corollary3.1} and \ref{corollary3.2} }\label{section3*}

Throughout this section, we rely on the following notation.
$\tilde{A}'_{\theta}(x'|x)$, $\tilde{B}'_{\theta}(x)$, $\tilde{B}_{\theta}(x)$,
$\tilde{C}'_{\theta}(y|x)$, $\tilde{D}'_{\theta}(x)$ and $\tilde{D}_{\theta}(x)$
are the functions defined by
\begin{align*}
	\tilde{A}'_{\theta}(x'|x) = x'-A_{\theta}(x),
	&\;\;\;\;\;
	\tilde{B}'_{\theta}(x) = \text{adj}B_{\theta}(x),
	\;\;\;\;\;
	\tilde{B}_{\theta}(x) = \text{det}B_{\theta}(x),
	\\
	\tilde{C}'_{\theta}(y|x) = y-C_{\theta}(x),
	&\;\;\;\;\;
	\tilde{D}'_{\theta}(x) = \text{adj}D_{\theta}(x),
	\;\;\;\;\;
	\tilde{D}_{\theta}(x) = \text{det}D_{\theta}(x)
\end{align*}
for $\theta\in\tilde{\Theta}$, $x,x'\in{\cal X}$, $y\in{\cal Y}$.
$\tilde{A}_{\theta}(x'|x)$, $\tilde{C}_{\theta}(y|x)$,
$U_{\theta}(x'|x)$ and $V_{\theta}(y|x)$ are the functions defined by
\begin{align*}
	\tilde{A}_{\theta}(x'|x)
	=
	\tilde{B}'_{\theta}(x)\tilde{A}'_{\theta}(x'|x),
	&\;\;\;\;\;
	U_{\theta}(x'|x) = \frac{\tilde{A}_{\theta}(x'|x) }{\tilde{B}_{\theta}(x) },
	\\
	\tilde{C}_{\theta}(y|x)
	=
	\tilde{D}'_{\theta}(x)\tilde{C}'_{\theta}(y|x),
	&\;\;\;\;\;
	V_{\theta}(y|x) = \frac{\tilde{C}_{\theta}(y|x) }{\tilde{D}_{\theta}(x) }.
\end{align*}
$u_{\theta}(x'|x)$, $\bar{u}_{\theta}(x)$, $v_{\theta}(y|x)$ and $\bar{v}_{\theta}(x)$
are the functions defined by
\begin{align*}
	u_{\theta}(x'|x)
	=
	r\left( U_{\theta}(x'|x) \right),
	&\;\;\;\;\;
	\bar{u}_{\theta}(x)
	=
	\int_{\cal X} u_{\theta}(x''|x)dx'',
	\\
	v_{\theta}(y|x)
	=
	s\left( V_{\theta}(y|x) \right),
	&\;\;\;\;\;
	\bar{v}_{\theta}(x)
	=
	\int_{\cal Y} v_{\theta}(y'|x)dy'.
\end{align*}
Then, it is easy to show
\begin{align*}
	U_{\theta}(x'|x)
	=
	B_{\theta}^{-1}(x) \left( x'-A_{\theta}(x) \right),
	\;\;\;\;\;
	V_{\theta}(y|x)
	=
	D_{\theta}^{-1}(x) \left( y-C_{\theta}(x) \right)
\end{align*}
for all $\theta\in\tilde{\Theta}$, $x,x'\in{\cal X}$, $y\in{\cal Y}$.
It is also easy to demonstrate
\begin{align*}
	p_{\theta}(x'|x) = \frac{u_{\theta}(x'|x) }{\bar{u}_{\theta}(x) },
	\;\;\;\;\;
	q_{\theta}(y|x) = \frac{v_{\theta}(y|x) }{\bar{v}_{\theta}(x) }.
\end{align*}

\begin{lemma}\label{lemma4.1}
Let Assumptions \ref{a3.1} -- \ref{a3.4} hold.
Then, $p_{\theta}(x'|x)$ and $q_{\theta}(y|x)$
are $p$-times differentiable in $\theta$ for each $\theta\in\Theta$,
$x\in{\cal X}$, $y\in{\cal Y}$.
Moreover, there exist real numbers $\varepsilon_{1}\in(0,1)$,
$K_{1}\in[1,\infty)$ such that
\begin{align}\label{l4.1.1*}
	&
	\min\left\{ p_{\theta}(x'|x), q_{\theta}(y|x) \right\}
	\geq
	\varepsilon_{1},
	\;\;\;\;\;
	\max\left\{ \left|\partial_{\theta}^{\boldsymbol\alpha}p_{\theta}(x'|x) \right|,
	\left|\partial_{\theta}^{\boldsymbol\alpha}q_{\theta}(y|x) \right|
	\right\}
	\leq
	K_{1}
\end{align}
for all $\theta\in\Theta$, $x,x'\in{\cal X}$, $y\in{\cal Y}$
and any multi-index $\boldsymbol\alpha\in\mathbb{N}_{0}^{d}$, $|\boldsymbol\alpha|\leq p$.
\end{lemma}

\begin{proof}
Throughout the proof, $\boldsymbol\alpha$ is any multi-index in $\mathbb{N}_{0}^{d}$
satisfying $|\boldsymbol\alpha|\leq p$.
It is easy to notice that $\tilde{B}_{\theta}(x)$ and the entries of $\tilde{B}'_{\theta}(x)$
are polynomial in the entries of $B_{\theta}(x)$.
It is also easy to notice that $\tilde{D}_{\theta}(x)$ and the entries of $\tilde{D}'_{\theta}(x)$
are polynomial in the entries of $D_{\theta}(x)$.
Consequently, Assumptions \ref{a3.2}, \ref{a3.3} imply that
$\partial_{\theta}^{\boldsymbol\alpha}\tilde{A}_{\theta}(x'|x)$,
$\partial_{\theta}^{\boldsymbol\alpha}\tilde{B}_{\theta}(x)$
exist and are continuous in $(\theta,x,x')$, $(\theta,x)$
on $\tilde{\Theta}\times{\cal X}\times{\cal X}$, $\tilde{\Theta}\times{\cal X}$.
The same assumptions also imply that
$\partial_{\theta}^{\boldsymbol\alpha}\tilde{C}_{\theta}(y|x)$,
$\partial_{\theta}^{\boldsymbol\alpha}\tilde{D}_{\theta}(x)$
exist and are continuous in $(\theta,x,y)$, $(\theta,x)$
on $\tilde{\Theta}\times{\cal X}\times{\cal Y}$, $\tilde{\Theta}\times{\cal X}$.
As $\tilde{B}_{\theta}(x)$, $\tilde{D}_{\theta}(x)$ are non-zero
(due to Assumption \ref{a3.1}),
we conclude from Lemma \ref{lemmaa2} (see Appendix \ref{appendix2}) that
$\partial_{\theta}^{\boldsymbol\alpha}\tilde{U}_{\theta}(x'|x)$,
$\partial_{\theta}^{\boldsymbol\alpha}\tilde{V}_{\theta}(y|x)$
exist and are continuous in $(\theta,x,x)$, $(\theta,x,y)$
on $\tilde{\Theta}\times{\cal X}\times{\cal X}$, $\tilde{\Theta}\times{\cal X}\times{\cal Y}$.
Then, using Assumption \ref{a3.2} and Lemma \ref{lemmaa1} (see Appendix \ref{appendix1}),
we deduce that
$\partial_{\theta}^{\boldsymbol\alpha}\tilde{u}_{\theta}(x'|x)$,
$\partial_{\theta}^{\boldsymbol\alpha}\tilde{v}_{\theta}(y|x)$
exist and are continuous in $(\theta,x,x)$, $(\theta,x,y)$
on $\tilde{\Theta}\times{\cal X}\times{\cal X}$, $\tilde{\Theta}\times{\cal X}\times{\cal Y}$.

Let $\theta$ be any element of $\Theta$. Moreover, let $x,x'$ be any elements of ${\cal X}$,
while $y$ is any element of ${\cal Y}$.
Since $\Theta$ is bounded and $\text{cl}\Theta\subset\tilde{\Theta}$,
Assumptions \ref{a3.1}, \ref{a3.4} imply that there exist real numbers
$\delta\in(0,1)$, $\tilde{C}\in[1,\infty)$
(independent of $\theta$, $x,x'$, $y$, $\boldsymbol\alpha$) such that
\begin{align}\label{l4.1.1}
	\min\left\{ u_{\theta}(x'|x), v_{\theta}(y|x) \right\}
	\geq\delta,
	\;\;\;\;\;
	\max\left\{
	\left|\partial_{\theta}^{\boldsymbol\alpha}u_{\theta}(x'|x)\right|,
	\left|\partial_{\theta}^{\boldsymbol\alpha}v_{\theta}(x'|x)\right|
	\right\}
	\leq\tilde{C}.
\end{align}
Consequently, Lemma \ref{lemmaa3} (see Appendix \ref{appendix3}) yields that
$\partial_{\theta}^{\boldsymbol\alpha}\bar{u}_{\theta}(x)$,
$\partial_{\theta}^{\boldsymbol\alpha}\bar{v}_{\theta}(x)$ exist.
Moreover, combining Assumption \ref{a3.4} and (\ref{l4.1.1}), we get
\begin{align}\label{l4.1.3}
	\bar{u}_{\theta}(x)
	=
	\int_{\cal X} u_{\theta}(x'|x) dx'
	\geq\delta\text{m}({\cal X})>0,
	\;\;\;\;\;
	\bar{v}_{\theta}(x)
	=
	\int_{\cal Y} v_{\theta}(y|x) dy
	\geq\delta\text{m}({\cal Y})>0,
\end{align}
where $\text{m}({\cal X})$, $\text{m}({\cal Y})$ are the Lebesgue measures of
${\cal X}$, ${\cal Y}$ (respectively).
Then, using Lemma \ref{lemmaa2}, we conclude that
$\partial_{\theta}^{\boldsymbol\alpha}p_{\theta}(x'|x)$,
$\partial_{\theta}^{\boldsymbol\alpha}q_{\theta}(y|x)$ exist.
Relying on the same lemma and (\ref{l4.1.1}), (\ref{l4.1.3}),
we deduce that there exists a real number
$K_{1}\in[1,\infty)$ with the properties specified in the lemma's statement.
\end{proof}

\begin{proof}[\rm\bf Proof of Corollary \ref{corollary3.1}]
Throughout the proof, the following notation is used.
$\varepsilon$, $\tilde{C}_{1}$, $\tilde{C}_{2}$, $\tilde{C}_{3}$ are the real numbers defined by
$\varepsilon=\min\{\varepsilon_{1}^{2},K_{1}^{-2} \}$,
$\tilde{C}_{1}=2K_{1}^{2}\varepsilon_{1}^{-2}$,
$\tilde{C}_{2}=K_{1}^{2}$,
$\tilde{C}_{3}=1+|\log\mu({\cal X}) |$
($\varepsilon_{1}$, $K_{1}$ are specified in Lemma \ref{lemma4.1}).
$r$, $u$, $v$ are the real numbers specified in Assumptions \ref{a1.5}, \ref{a2.2}.
$\psi(y)$, $\phi(x,y)$, $\varphi(y)$ and $\mu_{\theta}(dx|y)$
are the functions and the measure defined by
\begin{align*}
	\psi(y)=\tilde{C}_{1},
	\;\;\;\;\;
	\phi(x,y)=\tilde{C}_{2},
	\;\;\;\;\;
	\varphi(y)=\tilde{C}_{3},
	\;\;\;\;\;
	\mu_{\theta}(B|y)=\mu(B)
\end{align*}
for $\theta\in\Theta$, $x\in{\cal X}$, $y\in{\cal Y}$, $B\in{\cal B}({\cal X})$
($\mu(dx)$ is specified in Subsection \ref{subsection1.1}).
$r_{\theta}(y,x'|x)$ has the same meaning as in (\ref{1.901}),
while $p_{\theta}(x'|x)$, $q_{\theta}(y|x)$ are defined in (\ref{3.5}).
$\theta$ is any element of $\Theta$,
while $\boldsymbol\alpha$ is any multi-index in $\mathbb{N}_{0}^{d}$
satisfying $|\boldsymbol\alpha|\leq p$.
$x,x'$ are any elements of ${\cal X}$,
while $y$ is any element of ${\cal Y}$.

(i) Owing to Lemma \ref{lemma4.1}, we have
\begin{align}\label{c3.1.1}
	\varepsilon_{1}^{2} \leq r_{\theta}(y,x'|x) \leq K_{1}^{2}.
\end{align}
Consequently, we get
\begin{align*}
	\int_{B} r_{\theta}(y,x'|x)\mu(dx')
	\geq
	\varepsilon_{1}^{2}\mu(B)
	\geq
	\varepsilon\mu_{\theta}(B|y),
	\;\;\;\;\;
	\int_{B} r_{\theta}(y,x'|x)\mu(dx')
	\leq
	K_{1}^{2}\mu(B)
	\leq
	\frac{1}{\varepsilon}\mu_{\theta}(B|y)
\end{align*}
for $B\in{\cal B}({\cal X})$. We also get
\begin{align*}
	r_{\theta}(y,x'|x)
	\leq
	\tilde{C}_{2}
	=
	\phi(y,x'),
	\;\;\;\;\;
	\int\phi(y,x)\mu(dx)
	=
	\tilde{C}_{2}\mu({\cal X})
	<\infty.
\end{align*}
Hence, Assumptions \ref{a1.1}, \ref{a1.2'} hold for
$p_{\theta}(x'|x)$, $q_{\theta}(y|x)$ specified in (\ref{3.5}).

Due to Leibniz formula and Lemma \ref{lemma4.1}, we have
\begin{align*}
	\left| \partial_{\theta}^{\boldsymbol\alpha}r_{\theta}(y,x'|x) \right|
	\leq
	\sum_{\stackrel{\scriptstyle \boldsymbol\beta \in \mathbb{N}_{0}^{d}}
	{\boldsymbol\beta \leq \boldsymbol\alpha }}
	\left( \boldsymbol\alpha \atop \boldsymbol\beta \right)
	\left|\partial_{\theta}^{\boldsymbol\beta}q_{\theta}(y|x') \right|
	\left|\partial_{\theta}^{\boldsymbol\alpha-\boldsymbol\beta}p_{\theta}(x'|x) \right|
	\leq
	K_{1}^{2}
	\sum_{\stackrel{\scriptstyle \boldsymbol\beta \in \mathbb{N}_{0}^{d}}
	{\boldsymbol\beta \leq \boldsymbol\alpha }}
	\left( \boldsymbol\alpha \atop \boldsymbol\beta \right)
	=
	2^{|\boldsymbol\alpha|}K_{1}^{2}.
\end{align*}
Then, (\ref{c3.1.1}) implies
\begin{align*}
	\left| \partial_{\theta}^{\boldsymbol\alpha}r_{\theta}(y,x'|x) \right|
	\leq
	2^{|\boldsymbol\alpha|}K_{1}^{2}\varepsilon_{1}^{-2}
	r_{\theta}(y,x'|x)
	\leq
	\left( \psi(y) \right)^{|\boldsymbol\alpha|} r_{\theta}(y,x'|x).
\end{align*}
Thus, Assumption \ref{a1.2} holds for
$p_{\theta}(x'|x)$, $q_{\theta}(y|x)$ specified in (\ref{3.5}).
Consequently, all conclusions of Theorems \ref{theorem1.1}, \ref{theorem1.2} are true for the model
introduced in Section \ref{section3}.

(ii) Owing to (\ref{c3.1.1*}), we have
\begin{align*}
	\int\varphi(x,y)\psi^{r}(y)Q(x,dy)
	\leq
	\tilde{C}_{1}^{r}
	\sup_{x'\in{\cal X}}
	\int\varphi(x',y)Q(x',dy)
	<\infty.
\end{align*}
Hence, in addition to Assumptions \ref{a1.1} -- \ref{a1.2'},
Assumptions \ref{a1.3} -- \ref{a1.5} also hold for
$p_{\theta}(x'|x)$, $q_{\theta}(y|x)$ specified in (\ref{3.5}).
Therefore, all conclusions of Theorem \ref{theorem1.3} are true for the model
introduced in Section \ref{section3}.

(iii) It is easy to conclude
\begin{align*}
	\left| \log\mu_{\theta}({\cal X}|y) \right|
	=
	\left| \log\mu({\cal X}) \right|
	\leq
	\tilde{C}_{3}
	=
	\varphi(y).
\end{align*}
It is also easy to deduce
\begin{align*}
	\int\varphi(y)\psi^{u}(y)Q(x,dy)
	=
	\tilde{C}_{1}^{u}\tilde{C}_{3}
	<\infty,
	\;\;\;\;\;
	\int\psi^{v}(y)Q(x,dy)
	=
	\tilde{C}_{1}^{v}
	<\infty.
\end{align*}
Thus, in addition to Assumptions \ref{a1.1} -- \ref{a1.2'},
Assumptions \ref{a1.3}, \ref{a2.1}, \ref{a2.2} also hold for
$p_{\theta}(x'|x)$, $q_{\theta}(y|x)$ specified in (\ref{3.5}).
Consequently, all conclusions of Theorem \ref{theorem2.1} are true for the model
introduced in Section \ref{section3}.
\end{proof}

\begin{lemma}\label{lemma4.2}
(i) Let Assumptions \ref{a3.1} -- \ref{a3.3} and \ref{a3.5} hold.
Then, $p_{\theta}(x'|x)$ and $q_{\theta}(y|x)$
are $p$-times differentiable in $\theta$ for each $\theta\in\Theta$,
$x\in{\cal X}$, $y\in{\cal Y}$.
Moreover, there exist real numbers $\varepsilon_{2}\in(0,1)$,
$K_{2},K_{3}\in[1,\infty)$ such that
\begin{align}\label{l4.2.1*}
	&
	p_{\theta}(x'|x) \geq\varepsilon_{2},
	\;\;\;\;\;
	\left|\partial_{\theta}^{\boldsymbol\alpha}p_{\theta}(x'|x) \right| \leq K_{2},
	\;\;\;\;\;
	q_{\theta}(y|x)\leq K_{3},
	\;\;\;\;\;
	\left|\partial_{\theta}^{\boldsymbol\alpha}q_{\theta}(y|x) \right|
	\leq K_{3}q_{\theta}(y|x)(1+\|y\| )^{2|\boldsymbol\alpha|}
\end{align}
for all $\theta\in\Theta$, $x,x'\in{\cal X}$, $y\in{\cal Y}$
and any multi-index $\boldsymbol\alpha\in\mathbb{N}_{0}^{d}$, $|\boldsymbol\alpha|\leq p$.

(ii) Let Assumptions \ref{a3.1} -- \ref{a3.3}, \ref{a3.5} and \ref{a3.6} hold.
Then, there exist a real number $K_{4}\in[1,\infty)$ such that
\begin{align}\label{l4.2.3*}
	\left| \log q_{\theta}(y|x) \right|
	\leq K_{4}(1+\|y\| )^{2}
\end{align}
for all $\theta\in\Theta$, $x,x'\in{\cal X}$, $y\in{\cal Y}$.
\end{lemma}

\begin{proof}
Throughout the proof, the following notation is used.
$\theta$ is any element of $\Theta$, while
$\boldsymbol\alpha$ is any multi-index in $\mathbb{N}_{0}^{d}$
satisfying $|\boldsymbol\alpha|\leq p$.
$x,x'$ are any elements of ${\cal X}$, while $y$ is any element of ${\cal Y}$.

(i) Using the same arguments as in the proof of Lemma \ref{lemma4.1},
it can be shown that $\partial_{\theta}^{\boldsymbol\alpha}p_{\theta}(x'|x)$ exists.
Relying on the same arguments, it can also be demonstrated that
there exist real numbers $\varepsilon_{2}\in(0,1)$, $K_{2}\in[1,\infty)$
(independent of $\theta$, $x,x'$)
such that the first two inequalities in (\ref{l4.2.1*}) hold.
In what follows in the proof of (i),
we show that $\partial_{\theta}^{\boldsymbol\alpha}q_{\theta}(y|x)$ exists.
We also demonstrate that
there exists a real number $K_{3}\in[1,\infty)$
(independent of $\theta$, $x$, $y$)
such that the last two inequalities in (\ref{l4.2.1*}) hold.

Relying on the same arguments as in the proof of Lemma \ref{lemma4.1},
it can be shown that
$\partial_{\theta}^{\boldsymbol\alpha}\tilde{C}_{\theta}(y|x)$,
$\partial_{\theta}^{\boldsymbol\alpha}\tilde{C}'_{\theta}(y|x)$,
$\partial_{\theta}^{\boldsymbol\alpha}V_{\theta}(y|x)$,
$\partial_{\theta}^{\boldsymbol\alpha}v_{\theta}(y|x)$
exist and are continuous in $(\theta,x,y)$ on $\tilde{\Theta}\times{\cal X}\times{\cal Y}$.
Using the same arguments, it can be demonstrated that
$\partial_{\theta}^{\boldsymbol\alpha}\tilde{D}_{\theta}(x)$,
$\partial_{\theta}^{\boldsymbol\alpha}\tilde{D}'_{\theta}(x)$
exist and are continuous in $(\theta,x)$ on $\tilde{\Theta}\times{\cal X}$.
Since $\Theta$ is bounded and $\text{cl}\Theta\subset\tilde{\Theta}$,
Assumptions \ref{a3.1}, \ref{a3.3}, \ref{a3.5} imply that
there exist real numbers $\delta\in(0,1)$, $\tilde{C}_{1}\in[1,\infty )$
(independent of $\theta$, $x$, $\boldsymbol\beta$)
such that
\begin{align}\label{l4.2.1}
	\left| \tilde{D}_{\theta}(x) \right|\geq\delta,
	\;\;\;\;\;
	\max\left\{
	\left| \partial_{\theta}^{\boldsymbol\beta}\tilde{D}_{\theta}(x) \right|,
	\left\| \partial_{\theta}^{\boldsymbol\beta}\tilde{D}'_{\theta}(x) \right\|
	\right\}
	\leq\tilde{C}_{1}
\end{align}
for $\boldsymbol\beta\in\mathbb{N}_{0}^{d}$,
$|\boldsymbol\beta|\leq p$.
The same arguments also yield that there exists a real number
$\tilde{C}_{2}\in[1,\infty )$ (independent of $\theta$, $x$, $y$, $\boldsymbol\gamma$)
such that
\begin{align}\label{l4.2.3}
	\left\| \tilde{C}'_{\theta}(y|x) \right\|
	\leq\tilde{C}_{2}(1+\|y\| ),
	\;\;\;\;\;
	\left\| \partial_{\theta}^{\boldsymbol\gamma}\tilde{C}'_{\theta}(y|x) \right\|
	\leq\tilde{C}_{2}
\end{align}
for $\boldsymbol\gamma\in\mathbb{N}_{0}^{d}\setminus\{\boldsymbol 0\}$,
$|\boldsymbol\gamma|\leq p$.

Let $\tilde{C}_{3}=2^{p}\tilde{C}_{1}\tilde{C}_{2}$.
Owing to Leibniz formula and (\ref{l4.2.1}), (\ref{l4.2.3}),
we have
\begin{align*}
	\left\|\partial_{\theta}^{\boldsymbol\alpha} \tilde{C}_{\theta}(y|x) \right\|
	\leq
	\sum_{\stackrel{\scriptstyle \boldsymbol\beta \in \mathbb{N}_{0}^{d}}
	{\boldsymbol\beta \leq \boldsymbol\alpha }}
	\left( \boldsymbol\alpha \atop \boldsymbol\beta \right)
	\left\|\partial_{\theta}^{\boldsymbol\beta}\tilde{D}'_{\theta}(x) \right\|
	\left\|\partial_{\theta}^{\boldsymbol\alpha-\boldsymbol\beta}\tilde{C}'_{\theta}(y|x) \right\|
	\leq &
	\tilde{C}_{1}\tilde{C}_{2}
	\left(
	1 + \|y\|
	+
	\sum_{\stackrel{\scriptstyle \boldsymbol\beta \in \mathbb{N}_{0}^{d}
	\setminus\{\boldsymbol\alpha \} }
	{\boldsymbol\beta \leq \boldsymbol\alpha }}
	\left( \boldsymbol\alpha \atop \boldsymbol\beta \right)
	\right)
	\\
	\leq &
	2^{|\boldsymbol\alpha|}\tilde{C}_{1}\tilde{C}_{2} (1+\|y\| )
	\\
	\leq &
	\tilde{C}_{3} (1+\|y\| ).
\end{align*}
Consequently, Lemma \ref{lemmaa2} (see Appendix \ref{appendix2}) and (\ref{l4.2.1}) imply that
there exists a real number $\tilde{C}_{4}\in[1,\infty )$
(independent of $\theta$, $x$, $y$, $\boldsymbol\alpha$) such that
\begin{align}\label{l4.2.5}
	\left\|\partial_{\theta}^{\boldsymbol\alpha}V_{\theta}(y|x) \right\|
	\leq\tilde{C}_{4}(1+\|y\| ).
\end{align}
Then, Lemma \ref{lemmaa1} (see Appendix \ref{appendix1}) and Assumption \ref{a3.5},
yield that there exists a real number $\tilde{C}_{5}\in[1,\infty )$
(independent of $\theta$, $x$, $y$, $\boldsymbol\alpha$) such that
\begin{align}\label{l4.2.7}
	v_{\theta}(y|x)\leq\tilde{C}_{5},
	\;\;\;\;\;
	\left|\partial_{\theta}^{\boldsymbol\alpha}v_{\theta}(y|x) \right|
	\leq\tilde{C}_{5}v_{\theta}(y|x)(1+\|y\| )^{2|\boldsymbol\alpha|}.
\end{align}
Moreover, due to Assumptions \ref{a3.1}, \ref{a3.2},
the sign of $\tilde{D}_{\theta}(x)$ is constant in $\theta$ on each connected component of $\Theta$.
Since $\Theta$ is open, all connected components of $\Theta$ are open, too.
As $\bar{v}_{\theta}(x)=|\tilde{D}_{\theta}(x)|$
(due to Assumption \ref{a3.5} and ${\cal Y}=\mathbb{R}^{d_{y}}$),
we conclude that $\partial_{\theta}^{\boldsymbol\alpha}\bar{v}_{\theta}(x)$ exists.
Using (\ref{l4.2.1}), we also deduce
\begin{align}\label{l4.2.9}
	\bar{v}_{\theta}(x)\geq\delta,
	\;\;\;\;\;
	\left|\partial_{\theta}^{\boldsymbol\alpha}\bar{v}_{\theta}(x) \right|
	=
	\left|\partial_{\theta}^{\boldsymbol\alpha}\tilde{D}_{\theta}(x) \right|
	\leq\tilde{C}_{1}.
\end{align}
Consequently, Lemma \ref{lemmaa2} implies that
$\partial_{\theta}^{\boldsymbol\alpha}q_{\theta}(y|x)$ exists.
The same lemma, Assumption \ref{a3.5} and (\ref{l4.2.7}), (\ref{l4.2.9}) also yield that there exists
a real number $K_{3}\in[1,\infty )$ (independent of $\theta$, $x$, $y$, $\boldsymbol\alpha$) such that
the last two inequalities in (\ref{l4.2.1*}) hold.

(ii) Let $\tilde{C}_{6}=5L_{0}\tilde{C}_{1}\tilde{C}_{4}^{2}$,
$K_{4}=K_{0}\tilde{C}_{6}$.
Owing to Assumption \ref{a3.6} and (\ref{l4.2.1}), (\ref{l4.2.5}), we have
\begin{align}\label{l4.2.21}
	\log q_{\theta}(y|x)
	=
	\log v_{\theta}(y|x) - \log \bar{v}_{\theta}(x)
	=&
	\log s( V_{\theta}(y|x) ) - \log | \tilde{D}_{\theta}(x) |
	\nonumber\\
	\geq &
	-L_{0}(1 + \|V_{\theta}(y|x) \| )^{2} - \tilde{C}_{1}
	\nonumber\\
	\geq &
	-4L_{0}\tilde{C}_{4}^{2}(1 + \|y\| )^{2} - \tilde{C}_{1}
	\nonumber\\
	\geq &
	-\tilde{C}_{6}(1+\|y\| )^{2}.
\end{align}
Moreover, due to Assumption \ref{a3.5}, we have
\begin{align}\label{l4.2.23}
	\log q_{\theta}(y|x)
	\leq \log K_{0}
	\leq K_{0}(1+\|y\| )^{2}.
\end{align}
Combining (\ref{l4.2.21}), (\ref{l4.2.23}), we conclude that (\ref{l4.2.3*}) holds.
\end{proof}

\begin{proof}[\rm\bf Proof of Corollary \ref{corollary3.2}]
Throughout the proof, the following notation is used.
$\varepsilon$, $\tilde{C}_{1}$, $\tilde{C}_{2}$, $\tilde{C}_{3}$ are the real numbers defined by
$\varepsilon=\min\{\varepsilon_{2},K_{2}^{-1} \}$,
$\tilde{C}_{1}=2K_{2}K_{3}\varepsilon_{2}^{-2}$,
$\tilde{C}_{2}=K_{2}K_{3}$,
$\tilde{C}_{3}=K_{3}K_{4}(1+|\log\mu({\cal X})| )$
($\varepsilon_{2}$, $K_{2}$, $K_{3}$, $K_{4}$ are specified in Lemma \ref{lemma4.2}).
$r$, $u$, $v$ are the real numbers specified in Assumptions \ref{a1.5}, \ref{a2.2}.
$\psi(y)$, $\phi(x,y)$, $\varphi(y)$ and $\mu_{\theta}(dx|y)$
are the functions and the measure defined by
\begin{align*}
	\psi(y)=\tilde{C}_{1}(1+\|y\| )^{2},
	\;\;\;\;\;
	\phi(x,y)=\tilde{C}_{2},
	\;\;\;\;\;
	\varphi(y)=\tilde{C}_{3}(1+\|y\| )^{2},
	\;\;\;\;\;
	\mu_{\theta}(B|y)=\int_{B}q_{\theta}(y|x)\mu(dx)
\end{align*}
for $\theta\in\Theta$, $x\in{\cal X}$, $y\in{\cal Y}$, $B\in{\cal B}({\cal X})$
($\mu(dx)$ is specified in Subsection \ref{subsection1.1}).
$r_{\theta}(y,x'|x)$ has the same meaning as in (\ref{1.901}),
while $p_{\theta}(x'|x)$, $q_{\theta}(y|x)$ are defined in (\ref{3.5}).
$\theta$ is any element of $\Theta$,
while $\boldsymbol\alpha$ is any multi-index in $\mathbb{N}_{0}^{d}$
satisfying $|\boldsymbol\alpha|\leq p$.
$x,x'$ are any elements of ${\cal X}$,
while $y$ is any element of ${\cal Y}$.

(i) Owing to Lemma \ref{lemma4.2}, we have
\begin{align}\label{c3.2.1}
	\varepsilon_{2}q_{\theta}(y|x') \leq r_{\theta}(y,x'|x) \leq K_{2}q_{\theta}(y|x')
	\leq K_{2}K_{3}.
\end{align}
Consequently, we get
\begin{align*}
	&
	\int_{B} r_{\theta}(y,x'|x)\mu(dx')
	\geq
	\varepsilon_{2}\int_{B}q_{\theta}(y|x')\mu(dx')
	\geq
	\varepsilon\mu_{\theta}(B|y),
	\\
	&
	\int_{B} r_{\theta}(y,x'|x)\mu(dx')
	\leq
	K_{2}\int_{B}q_{\theta}(y|x')\mu(dx')
	\leq
	\frac{1}{\varepsilon}\mu_{\theta}(B|y)
\end{align*}
for $B\in{\cal B}({\cal X})$. We also get
\begin{align*}
	r_{\theta}(y,x'|x)
	\leq
	\tilde{C}_{2}
	=
	\phi(y,x'),
	\;\;\;\;\;
	\int\phi(y,x)\mu(dx)
	=
	\tilde{C}_{2}\mu({\cal X})
	<\infty.
\end{align*}
Hence, Assumptions \ref{a1.1}, \ref{a1.2'} hold for
$p_{\theta}(x'|x)$, $q_{\theta}(y|x)$ specified in (\ref{3.5}).

Due to Leibniz formula and Lemma \ref{lemma4.2}, we have
\begin{align*}
	\left| \partial_{\theta}^{\boldsymbol\alpha}r_{\theta}(y,x'|x) \right|
	\leq
	\sum_{\stackrel{\scriptstyle \boldsymbol\beta \in \mathbb{N}_{0}^{d}}
	{\boldsymbol\beta \leq \boldsymbol\alpha }}
	\left( \boldsymbol\alpha \atop \boldsymbol\beta \right)
	\left|\partial_{\theta}^{\boldsymbol\beta}q_{\theta}(y|x') \right|
	\left|\partial_{\theta}^{\boldsymbol\alpha-\boldsymbol\beta}p_{\theta}(x'|x) \right|
	\leq &
	K_{2}K_{3}q_{\theta}(y|x')
	\sum_{\stackrel{\scriptstyle \boldsymbol\beta \in \mathbb{N}_{0}^{d}}
	{\boldsymbol\beta \leq \boldsymbol\alpha }}
	\left( \boldsymbol\alpha \atop \boldsymbol\beta \right)
	(1+\|y\| )^{2|\boldsymbol\beta|}
	\\
	\leq &
	2^{|\boldsymbol\alpha|}K_{2}K_{3}q_{\theta}(y|x')(1+\|y\| )^{2|\boldsymbol\alpha|}.
\end{align*}
Then, (\ref{c3.2.1}) implies
\begin{align*}
	\left| \partial_{\theta}^{\boldsymbol\alpha}r_{\theta}(y,x'|x) \right|
	\leq
	2^{|\boldsymbol\alpha|}K_{2}K_{3}\varepsilon_{2}^{-1}
	(1+\|y\| )^{2|\boldsymbol\alpha|}
	r_{\theta}(y,x'|x)
	\leq
	\left( \psi(y) \right)^{|\boldsymbol\alpha|} r_{\theta}(y,x'|x).
\end{align*}
Thus, Assumption \ref{a1.2} holds for
$p_{\theta}(x'|x)$, $q_{\theta}(y|x)$ specified in (\ref{3.5}).
Consequently, all conclusions of Theorems \ref{theorem1.1}, \ref{theorem1.2} are true for the model
introduced in Section \ref{section3}.

(ii) Owing to (\ref{c3.2.1*}), we have
\begin{align*}
	\int\varphi(x,y)\psi^{r}(y)Q(x,dy)
	\leq
	\tilde{C}_{1}^{r}
	\sup_{x'\in{\cal X}}
	\int\varphi(x',y)(1+\|y\| )^{2r}Q(x',dy)
	<\infty.
\end{align*}
Hence, in addition to Assumptions \ref{a1.1} -- \ref{a1.2'},
Assumptions \ref{a1.3} -- \ref{a1.5} also hold for
$p_{\theta}(x'|x)$, $q_{\theta}(y|x)$ specified in (\ref{3.5}).
Therefore, all conclusions of Theorem \ref{theorem1.3} are true for the model
introduced in Section \ref{section3}.

(iii) Owing to Lemma \ref{lemma4.2}, we have
\begin{align}\label{c3.2.3}
	\mu_{\theta}({\cal X}|y)
	=
	\int q_{\theta}(y|x)\mu(dx)
	\leq
	K_{3}\mu({\cal X}).
\end{align}
Due to the same lemma and Jensen inequality, we also have
\begin{align}\label{c3.2.5}
	\log\mu_{\theta}({\cal X}|y)
	\geq
	\log\mu({\cal X})
	+
	\frac{1}{\mu({\cal X})}
	\int \log q_{\theta}(y|x) \mu(dx)
	\geq
	-|\log\mu({\cal X})|
	-K_{4}(1+\|y\| )^{2}.
\end{align}
Combining (\ref{c3.2.3}), (\ref{c3.2.5}), we get
\begin{align*}
	\left| \log\mu_{\theta}({\cal X}|y) \right|
	\leq
	K_{3}|\log\mu({\cal X})|
	+
	K_{4}(1+\|y\| )^{2}
	\leq
	\tilde{C}_{3}(1+\|y\| )^{2}
	=
	\varphi(y).
\end{align*}
Moreover, (\ref{c3.2.3*}) implies
\begin{align*}
	\int\psi^{v}(y)Q(x,dy)
	\leq
	\tilde{C}_{1}^{v}
	\sup_{x'\in{\cal X}}
	\int (1+\|y\| )^{2v}Q(x',dy)
	<\infty.
\end{align*}
As $v\geq u+1$, (\ref{c3.2.3*}) also yields
\begin{align*}
	\int\varphi(y)\psi^{u}(y)Q(x,dy)
	\leq
	\tilde{C}_{1}^{u}\tilde{C}_{3}
	\sup_{x'\in{\cal X}}
	\int (1+\|y\| )^{2(u+1)}Q(x',dy)
	<\infty.
\end{align*}
Thus, in addition to Assumptions \ref{a1.1} -- \ref{a1.2'},
Assumptions \ref{a1.3}, \ref{a2.1}, \ref{a2.2} also hold for
$p_{\theta}(x'|x)$, $q_{\theta}(y|x)$ specified in (\ref{3.5}).
Consequently, all conclusions of Theorem \ref{theorem2.1} are true for the model
introduced in Section \ref{section3}.
\end{proof}

\refstepcounter{appendixcounter}\label{appendix1}
\section*{Appendix \arabic{appendixcounter} }

In this section, we present auxiliary results crucially important for the proof
of Corollaries \ref{corollary3.1} and \ref{corollary3.2}.
Let $\Theta$ and $d$ have the same meaning as in Subsection \ref{subsection1.1}.
Moreover, let ${\cal Z}$ be an open set in $\mathbb{R}^{d_{z}}$,
where $d_{z}\geq 1$ is an integer.
We consider here functions $f_{\theta}$ and $g(z)$
mapping $\theta\in\Theta$, $z\in{\cal Z}$ to
${\cal Z}$ and $\mathbb{R}$ (respectively).
We also consider function $h_{\theta}$ defined by
$h_{\theta}=g(f_{\theta})$ for $\theta\in\Theta$.
The analysis carried out in this section relies on the following
assumptions.

\begin{assumptionappendix}\label{aa1.1}
$f_{\theta}$ and $g(z)$ are $p$-times differentiable
on $\Theta$ and ${\cal Z}$ (respectively),
where $p\geq 1$ is an integer.
\end{assumptionappendix}

\begin{assumptionappendix}\label{aa1.2}
There exist a real number $K\in[1,\infty)$ and a function
$\phi_{\theta}$ mapping $\theta\in\Theta$ to $[1,\infty)$ such that
\begin{align*}
	\max\left\{
	\left\|f_{\theta}\right\|,
	\left\| \partial_{\theta}^{\boldsymbol\alpha}f_{\theta} \right\|
	\right\}
	\leq \phi_{\theta},
	\;\;\;\;\;
	\left| \partial^{\boldsymbol\beta}g(z) \right|
	\leq
	K|g(z)|(1+\|z\| )^{|\boldsymbol\beta|}
\end{align*}
for all $\theta\in\Theta$, $z\in{\cal Z}$
and any multi-indices $\boldsymbol\alpha\in\mathbb{N}_{0}^{d}\setminus\{\boldsymbol 0\}$,
$\boldsymbol\beta\in\mathbb{N}_{0}^{d_{z} }\setminus\{\boldsymbol 0\}$
satisfying $|\boldsymbol\alpha|\leq p$,
$|\boldsymbol\beta|\leq p$.
\end{assumptionappendix}

Throughout this section, the following notation is used.
For $\boldsymbol\alpha=(\alpha_{1},\dots,\alpha_{d} )\in\mathbb{N}_{0}^{d}$,
$n_{\boldsymbol\alpha}$ and $m_{\boldsymbol\alpha}$ are the integers defined by
\begin{align*}
	n_{\boldsymbol\alpha}
	=
	(\alpha_{1}+1 )\cdots(\alpha_{d}+1 )-1,
	\;\;\;\;\;
	m_{\boldsymbol\alpha}
	=
	\left( n_{\boldsymbol\alpha} \atop |\boldsymbol\alpha| \right).
\end{align*}
For $\theta\in\Theta$, $1\leq k\leq d_{z}$,
$f_{\theta,k}$ is the $k$-th component of $f_{\theta}$.
For the same $\theta$ and
$\boldsymbol\alpha\in\mathbb{N}_{0}^{d}\setminus\{\boldsymbol 0\}$,
$|\boldsymbol\alpha|\leq p$,
$F_{\theta,\boldsymbol\alpha}$ is the $n_{\boldsymbol\alpha}$-dimensional vector
whose components are derivatives
$\left\{ \partial_{\theta}^{\boldsymbol\beta}f_{\theta,k}:
\boldsymbol\beta\in\mathbb{N}_{0}^{d}\setminus\{\boldsymbol 0\},
\boldsymbol\beta\leq\boldsymbol\alpha, 1\leq k\leq d_{z} \right\}$.
In $F_{\theta,\boldsymbol\alpha}$,
the components are ordered lexicographically in $(k,\boldsymbol\beta)$.

\begin{lemmaappendix}\label{lemmaa1}
(i) Let Assumption \ref{aa1.1} hold.
Then, $h_{\theta}$ is $p$-times differentiable on $\Theta$.
Moreover, the first and higher-order derivatives of $h_{\theta}$
admit representation
\begin{align}\label{la1.1*}
	\partial_{\theta}^{\boldsymbol\alpha}h_{\theta}
	=
	\sum_{\stackrel{\scriptstyle \boldsymbol\beta \in \mathbb{N}_{0}^{d_{z}}\setminus\{\boldsymbol 0\}}
	{|\boldsymbol\beta| \leq |\boldsymbol\alpha| }}
	\partial^{\boldsymbol\beta}g(f_{\theta} )
	\:
	P_{\boldsymbol\alpha,\boldsymbol\beta}(F_{\theta,\boldsymbol\alpha} )
\end{align}
for all $\theta\in\Theta$ and any multi-index
$\boldsymbol\alpha\in\mathbb{N}_{0}^{d}\setminus\{\boldsymbol 0\}$
satisfying $|\boldsymbol\alpha|\leq p$.
Here, $P_{\boldsymbol\alpha,\boldsymbol\beta}:\mathbb{R}^{n_{\boldsymbol\alpha} }\rightarrow\mathbb{R}$
is a polynomial of degree up to $|\boldsymbol\alpha|$
whose coefficients are independent of $\theta$ and depend only on $\boldsymbol\alpha$,
$\boldsymbol\beta$.

(ii) Let Assumptions \ref{aa1.1} and \ref{aa1.2} hold.
Then, there exists a real number $L\in[1,\infty)$
such that
\begin{align*}
	\left|\partial_{\theta}^{\boldsymbol\alpha}h_{\theta} \right|
	\leq
	L\left|h_{\theta} \right| \phi_{\theta}^{2|\boldsymbol\alpha|}
\end{align*}
for all $\theta\in\Theta$ and any multi-index
$\boldsymbol\alpha\in\mathbb{N}_{0}^{d}\setminus\{\boldsymbol 0\}$
satisfying $|\boldsymbol\alpha|\leq p$.
\end{lemmaappendix}

\begin{proof}
(i) This part of lemma is proved by induction in $|\boldsymbol\alpha|$.
It is straightforward to show that
$\partial_{\theta}^{\boldsymbol\alpha}h_{\theta}$ exists
and satisfies (\ref{la1.1*}) for all $\theta\in\Theta$,
$\boldsymbol\alpha\in\mathbb{N}_{0}^{d}$, $|\boldsymbol\alpha|=1$.
Now, the induction hypothesis is formulated.
Let $1\leq l<p$ be an integer.
Suppose that $\partial_{\theta}^{\boldsymbol\alpha}h_{\theta}$ exists
and satisfies (\ref{la1.1*}) for each $\theta\in\Theta$,
$\boldsymbol\alpha\in\mathbb{N}_{0}^{d}$, $|\boldsymbol\alpha|\leq l$.
Then, to show (i), it is sufficient to demonstrate
that $\partial_{\theta}^{\boldsymbol\alpha}h_{\theta}$ exists
and satisfies (\ref{la1.1*}) for all $\theta\in\Theta$,
$\boldsymbol\alpha\in\mathbb{N}_{0}^{d}$, $|\boldsymbol\alpha|=l+1$.

Let $\theta$ be any element of $\Theta$,
while $\boldsymbol\alpha$ is any multi-index in $\mathbb{N}_{0}^{d}$
satisfying $|\boldsymbol\alpha|=l+1$.
Then, there exists $\boldsymbol e\in\mathbb{N}_{0}^{d}$
such that $\boldsymbol e\leq\boldsymbol\alpha$, $|\boldsymbol e|=1$.
As $|\boldsymbol\alpha-\boldsymbol e|=|\boldsymbol\alpha|-1=l$,
the induction hypothesis yields
\begin{align}\label{la1.1}
	\partial_{\theta}^{\boldsymbol\alpha-\boldsymbol e}h_{\theta}
	=
	\sum_{\stackrel{\scriptstyle \boldsymbol\beta \in \mathbb{N}_{0}^{d_{z}}\setminus\{\boldsymbol 0\}}
	{|\boldsymbol\beta| \leq l }}
	\:
	\partial^{\boldsymbol\beta}g(f_{\theta} )
	P_{\boldsymbol\alpha-\boldsymbol e,\boldsymbol\beta}(F_{\theta,\boldsymbol\alpha-\boldsymbol e} ).
\end{align}
Since $l<p$, the right-hand side of (\ref{la1.1}) involves only the derivatives of
$f_{\theta}$, $g(z)$ of the order up to $p-1$.
Then, Assumption \ref{aa1.1} implies that
$\partial_{\theta}^{\boldsymbol\alpha}h_{\theta}=
\partial_{\theta}^{\boldsymbol e}\left( \partial_{\theta}^{\boldsymbol\alpha-\boldsymbol e}h_{\theta} \right)$
exist and satisfies
\begin{align}\label{la1.5}
	\partial_{\theta}^{\boldsymbol\alpha}h_{\theta}
	=&
	\sum_{\stackrel{\scriptstyle \boldsymbol\beta \in \mathbb{N}_{0}^{d_{z}}\setminus\{\boldsymbol 0\}}
	{|\boldsymbol\beta| \leq l }}
	\partial^{\boldsymbol\beta}g(f_{\theta} ) \:
	\partial_{\theta}^{\boldsymbol e}
	P_{\boldsymbol\alpha-\boldsymbol e,\boldsymbol\beta}(F_{\theta,\boldsymbol\alpha-\boldsymbol e} )
	+
	\sum_{k=1}^{d_{z} }
	\sum_{\stackrel{\scriptstyle \boldsymbol\beta \in \mathbb{N}_{0}^{d_{z}}\setminus\{\boldsymbol 0\}}
	{|\boldsymbol\beta| \leq l }}
	\partial^{\boldsymbol\beta + \boldsymbol e_{k} }g(f_{\theta} ) \:
	\partial_{\theta}^{\boldsymbol e} f_{\theta,k} \:
	P_{\boldsymbol\alpha-\boldsymbol e,\boldsymbol\beta}(F_{\theta,\boldsymbol\alpha-\boldsymbol e} )
	\nonumber\\
	=&
	\sum_{\stackrel{\scriptstyle \boldsymbol\beta \in \mathbb{N}_{0}^{d_{z}}\setminus\{\boldsymbol 0\}}
	{|\boldsymbol\beta| \leq l }}
	\partial^{\boldsymbol\beta}g(f_{\theta} ) \:
	\partial_{\theta}^{\boldsymbol e}
	P_{\boldsymbol\alpha-\boldsymbol e,\boldsymbol\beta}(F_{\theta,\boldsymbol\alpha-\boldsymbol e} )
	+
	\sum_{k=1}^{d_{z} }
	\sum_{\stackrel{\scriptstyle \boldsymbol\beta \in \mathbb{N}_{0}^{d_{z}}\setminus\{\boldsymbol 0\}}
	{\boldsymbol e_{k}\leq\boldsymbol\beta, |\boldsymbol\beta| \leq l+1 }}
	\partial^{\boldsymbol\beta}g(f_{\theta} ) \:
	\partial_{\theta}^{\boldsymbol e} f_{\theta,k} \:
	P_{\boldsymbol\alpha-\boldsymbol e,\boldsymbol\beta-\boldsymbol e_{k} }(F_{\theta,\boldsymbol\alpha-\boldsymbol e} ),
\end{align}
where $\boldsymbol e_{k}$ is the $k$-th standard unit vector in $\mathbb{N}_{0}^{d_{z} }$.
Moreover, terms
\begin{align*}
	\partial_{\theta}^{\boldsymbol e}
	P_{\boldsymbol\alpha-\boldsymbol e,\boldsymbol\beta}(F_{\theta,\boldsymbol\alpha-\boldsymbol e} ),
	\;\;\;\;\;
	\partial_{\theta}^{\boldsymbol e} f_{\theta,k} \:
	P_{\boldsymbol\alpha-\boldsymbol e,\boldsymbol\beta-\boldsymbol e_{k} }(F_{\theta,\boldsymbol\alpha-\boldsymbol e} )
\end{align*}
are polynomial in derivatives
$\left\{ \partial_{\theta}^{\boldsymbol\gamma}f_{\theta,j}:
\boldsymbol\gamma\in\mathbb{N}_{0}^{d}\setminus\{\boldsymbol 0\},
\boldsymbol\gamma\leq\boldsymbol\alpha, 1\leq j\leq d_{z} \right\}$.
Apparently, the order of these polynomials is up to $|\boldsymbol\alpha-\boldsymbol e|+1=|\boldsymbol\alpha|$,
while the corresponding coefficients are independent of $\theta$
and depend only on $\boldsymbol\alpha$, $\boldsymbol\beta$.
Therefore, the right-hand side of (\ref{la1.5}) admits representation (\ref{la1.1*}).
Hence, the same holds for $\partial_{\theta}^{\boldsymbol\alpha}h_{\theta}$.

(ii) Let $\tilde{C}_{\boldsymbol\alpha,\boldsymbol\beta}$
be the maximum absolute value of the coefficients of polynomial
$P_{\boldsymbol\alpha,\boldsymbol\beta}(\cdot)$,
where $\boldsymbol\alpha\in\mathbb{N}_{0}^{d}\setminus\{\boldsymbol 0\}$,
$\boldsymbol\beta\in\mathbb{N}_{0}^{d_{z}}\setminus\{\boldsymbol 0\}$,
$|\boldsymbol\alpha|\leq p$, $|\boldsymbol\beta|\leq|\boldsymbol\alpha|$.
As the number of different power terms in $P_{\boldsymbol\alpha,\boldsymbol\beta}(\cdot)$
is at most $m_{\boldsymbol\alpha}$, Assumption \ref{aa1.2} and (i) yield
\begin{align*}
	\left|
	\partial^{\boldsymbol\beta}g(f_{\theta} )
	\:
	P_{\boldsymbol\alpha,\boldsymbol\beta}(F_{\theta,\boldsymbol\alpha} )
	\right|
	\leq
	K\tilde{C}_{\boldsymbol\alpha,\boldsymbol\beta}\:m_{\boldsymbol\alpha}
	|g(f_{\theta} ) |
	(1+\|f_{\theta}\| )^{|\boldsymbol\beta|}\phi_{\theta}^{|\boldsymbol\alpha|}
	\leq &
	K\tilde{C}_{\boldsymbol\alpha,\boldsymbol\beta}\:m_{\boldsymbol\alpha}
	|h_{\theta} |
	(1+\phi_{\theta} )^{|\boldsymbol\alpha|}\phi_{\theta}^{|\boldsymbol\alpha|}
	\\
	\leq &
	2^{|\boldsymbol\alpha|}K\tilde{C}_{\boldsymbol\alpha,\boldsymbol\beta}\:m_{\boldsymbol\alpha}
	|h_{\theta} |\phi_{\theta}^{2|\boldsymbol\alpha|}.
\end{align*}
Then, using (i) again, we conclude that there exists a real number $L\in[1,\infty)$
with the properties specified in the lemma's statement.
\end{proof}

\refstepcounter{appendixcounter}\label{appendix2}
\section*{Appendix \arabic{appendixcounter} }

As the previous section, this section provides auxiliary results relevant for the proof
of Corollaries \ref{corollary3.1} and \ref{corollary3.2}.
Let $\Theta$ and $d$ have the same meaning as in Subsection \ref{subsection1.1}.
We consider here functions $f_{\theta}$ and $g_{\theta}$
mapping $\theta\in\Theta$ to
$\mathbb{R}$ and $\mathbb{R}\setminus\{0\}$ (respectively).
We also consider function $h_{\theta}$ defined by
$h_{\theta}=f_{\theta}/g_{\theta}$ for $\theta\in\Theta$.
The results presented in this section rely on the following
assumptions.

\begin{assumptionappendix}\label{aa2.1}
$f_{\theta}$ and $g_{\theta}$ are $p$-times differentiable
on $\Theta$,
where $p\geq 1$ is an integer.
\end{assumptionappendix}

\begin{assumptionappendix}\label{aa2.2}
There exist functions
$\phi_{\theta}$ and $\psi_{\theta}$ mapping $\theta\in\Theta$ to $[1,\infty)$ such that
\begin{align*}
	\left| \partial_{\theta}^{\boldsymbol\alpha}f_{\theta} \right|
	\leq |f_{\theta}| \phi_{\theta}^{|\boldsymbol\alpha|},
	\;\;\;\;\;
	\left| \partial_{\theta}^{\boldsymbol\alpha}g_{\theta} \right|
	\leq
	\psi_{\theta}
\end{align*}
for all $\theta\in\Theta$
and any multi-index $\boldsymbol\alpha\in\mathbb{N}_{0}^{d}$
satisfying $|\boldsymbol\alpha|\leq p$.
\end{assumptionappendix}

Throughout this section, we use the following notation.
For $\boldsymbol\alpha=(\alpha_{1},\dots,\alpha_{d} )\in\mathbb{N}_{0}^{d}$,
$n_{\boldsymbol\alpha}$ and $m_{\boldsymbol\alpha}$ are the integers defined by
\begin{align*}
	n_{\boldsymbol\alpha}
	=
	(\alpha_{1}+1 )\cdots(\alpha_{d}+1 ),
	\;\;\;\;\;
	m_{\boldsymbol\alpha}
	=
	\left( n_{\boldsymbol\alpha} \atop |\boldsymbol\alpha| \right).
\end{align*}
For $\theta\in\Theta$ and
$\boldsymbol\alpha\in\mathbb{N}_{0}^{d}$,
$|\boldsymbol\alpha|\leq p$,
$G_{\theta,\boldsymbol\alpha}$ is the $n_{\boldsymbol\alpha}$-dimensional vector
whose components are derivatives
$\left\{ \partial_{\theta}^{\boldsymbol\beta}g_{\theta}:
\boldsymbol\beta\in\mathbb{N}_{0}^{d}\setminus\{\boldsymbol 0\},
\boldsymbol\beta\leq\boldsymbol\alpha, 1\leq k\leq d_{z} \right\}$.
In $G_{\theta,\boldsymbol\alpha}$,
the components are ordered lexicographically in $\boldsymbol\beta$.

\begin{lemmaappendix}\label{lemmaa2}
(i) Let Assumption \ref{aa2.1} hold.
Then, $h_{\theta}$ is $p$-times differentiable on $\Theta$.
Moreover, the first and higher-order derivatives of $h_{\theta}$
admit representation
\begin{align}\label{la2.1*}
	\partial_{\theta}^{\boldsymbol\alpha}h_{\theta}
	=
	\sum_{\stackrel{\scriptstyle \boldsymbol\beta \in \mathbb{N}_{0}^{d}}
	{\boldsymbol\beta\leq\boldsymbol\alpha}}
	\frac{\partial_{\theta}^{\boldsymbol\beta}f_{\theta}
	\: P_{\boldsymbol\alpha,\boldsymbol\beta}(G_{\theta,\boldsymbol\alpha} ) }
	{ g_{\theta}^{|\boldsymbol\alpha|+1} }
\end{align}
for all $\theta\in\Theta$ and any multi-index
$\boldsymbol\alpha\in\mathbb{N}_{0}^{d}$
satisfying $|\boldsymbol\alpha|\leq p$.
Here, $P_{\boldsymbol\alpha,\boldsymbol\beta}:\mathbb{R}^{n_{\boldsymbol\alpha} }\rightarrow\mathbb{R}$
is a polynomial of the degree up to $|\boldsymbol\alpha|$
whose coefficients are independent of $\theta$ and depend only on $\boldsymbol\alpha$,
$\boldsymbol\beta$.

(ii) Let Assumptions \ref{aa2.1} and \ref{aa2.2} hold.
Then, there exists a real number $K\in[1,\infty)$
such that
\begin{align*}
	\left|\partial_{\theta}^{\boldsymbol\alpha}h_{\theta} \right|
	\leq
	K\left|\frac{f_{\theta}}{g_{\theta}}\right|
	\left(\frac{\phi_{\theta}\psi_{\theta} }{|g_{\theta}| } \right)^{|\boldsymbol\alpha|}
\end{align*}
for all $\theta\in\Theta$ and any multi-index
$\boldsymbol\alpha\in\mathbb{N}_{0}^{d}$
satisfying $|\boldsymbol\alpha|\leq p$.
\end{lemmaappendix}

\begin{proof}
(i) This part of lemma is proved by induction in $|\boldsymbol\alpha|$.
It is straightforward to show that
$\partial_{\theta}^{\boldsymbol\alpha}h_{\theta}$ exists
and satisfies (\ref{la2.1*}) for all $\theta\in\Theta$,
$\boldsymbol\alpha\in\mathbb{N}_{0}^{d}$, $|\boldsymbol\alpha|\in\{0,1\}$.
Now, the induction hypothesis is formulated.
Let $1\leq l<p$ be an integer.
Suppose that $\partial_{\theta}^{\boldsymbol\alpha}h_{\theta}$ exists
and satisfies (\ref{la2.1*}) for each $\theta\in\Theta$,
$\boldsymbol\alpha\in\mathbb{N}_{0}^{d}$, $|\boldsymbol\alpha|\leq l$.
Then, to show (i), it is sufficient to demonstrate
that $\partial_{\theta}^{\boldsymbol\alpha}h_{\theta}$ exists
and satisfies (\ref{la2.1*}) for all $\theta\in\Theta$,
$\boldsymbol\alpha\in\mathbb{N}_{0}^{d}$, $|\boldsymbol\alpha|=l+1$.

Let $\theta$ be any element of $\Theta$,
while $\boldsymbol\alpha$ is any multi-index in $\mathbb{N}_{0}^{d}$
satisfying $|\boldsymbol\alpha|=l+1$.
Then, there exists $\boldsymbol e\in\mathbb{N}_{0}^{d}$
such that $\boldsymbol e\leq\boldsymbol\alpha$, $|\boldsymbol e|=1$.
As $|\boldsymbol\alpha-\boldsymbol e|=|\boldsymbol\alpha|-1=l$,
the induction hypothesis yields
\begin{align}\label{la2.1}
	\partial_{\theta}^{\boldsymbol\alpha-\boldsymbol e}h_{\theta}
	=
	\sum_{\stackrel{\scriptstyle \boldsymbol\beta \in \mathbb{N}_{0}^{d}}
	{\boldsymbol\beta\leq\boldsymbol\alpha-\boldsymbol e}}
	\frac{\partial_{\theta}^{\boldsymbol\beta}f_{\theta}
	\: P_{\boldsymbol\alpha-\boldsymbol e,\boldsymbol\beta}(G_{\theta,\boldsymbol\alpha-\boldsymbol e} ) }
	{ g_{\theta}^{|\boldsymbol\alpha|} }.
\end{align}
Since $l<p$, the right-hand side of (\ref{la2.1}) involves only the derivatives of
$f_{\theta}$, $g_{\theta}$ of the order up to $p-1$.
Then, Assumption \ref{aa2.1} implies that
$\partial_{\theta}^{\boldsymbol\alpha}h_{\theta}=
\partial_{\theta}^{\boldsymbol e}\left( \partial_{\theta}^{\boldsymbol\alpha-\boldsymbol e}h_{\theta} \right)$
exist and satisfies
\begin{align}\label{la2.5}
	\partial_{\theta}^{\boldsymbol\alpha}h_{\theta}
	=&
	\sum_{\stackrel{\scriptstyle \boldsymbol\beta \in \mathbb{N}_{0}^{d}}
	{\boldsymbol\beta\leq\boldsymbol\alpha-\boldsymbol e}}
	\frac{\partial_{\theta}^{\boldsymbol\beta+\boldsymbol e}f_{\theta} \:
	P_{\boldsymbol\alpha-\boldsymbol e,\boldsymbol\beta}(G_{\theta,\boldsymbol\alpha-\boldsymbol e} )
	+
	\partial_{\theta}^{\boldsymbol\beta}f_{\theta} \:
	\partial_{\theta}^{\boldsymbol e}P_{\boldsymbol\alpha-\boldsymbol e,\boldsymbol\beta}
	(G_{\theta,\boldsymbol\alpha-\boldsymbol e} )
	}
	{ g_{\theta}^{|\boldsymbol\alpha|} }
	-
	|\boldsymbol\alpha|
	\sum_{\stackrel{\scriptstyle \boldsymbol\beta \in \mathbb{N}_{0}^{d}}
	{\boldsymbol\beta\leq\boldsymbol\alpha-\boldsymbol e}}
	\frac{\partial_{\theta}^{\boldsymbol\beta}f_{\theta} \:
	\partial_{\theta}^{\boldsymbol e}g_{\theta} \:
	P_{\boldsymbol\alpha-\boldsymbol e,\boldsymbol\beta}(G_{\theta,\boldsymbol\alpha-\boldsymbol e} ) }
	{ g_{\theta}^{|\boldsymbol\alpha|+1} }
	\nonumber\\
	=&
	\sum_{\stackrel{\scriptstyle \boldsymbol\beta \in \mathbb{N}_{0}^{d}}
	{\boldsymbol\beta\leq\boldsymbol\alpha-\boldsymbol e}}
	\frac{\partial_{\theta}^{\boldsymbol\beta}f_{\theta}
	\left(
	g_{\theta} \:
	\partial_{\theta}^{\boldsymbol e}P_{\boldsymbol\alpha-\boldsymbol e,\boldsymbol\beta}
	(G_{\theta,\boldsymbol\alpha-\boldsymbol e} )
	-
	|\boldsymbol\alpha|
	\partial_{\theta}^{\boldsymbol e}g_{\theta} \:
	P_{\boldsymbol\alpha-\boldsymbol e,\boldsymbol\beta}
	(G_{\theta,\boldsymbol\alpha-\boldsymbol e} )
	\right)
	}
	{ g_{\theta}^{|\boldsymbol\alpha|+1} }
	+
	\sum_{\stackrel{\scriptstyle \boldsymbol\beta \in \mathbb{N}_{0}^{d}}
	{\boldsymbol e\leq\boldsymbol\beta\leq\boldsymbol\alpha}}
	\frac{\partial_{\theta}^{\boldsymbol\beta}f_{\theta} \:
	P_{\boldsymbol\alpha-\boldsymbol e,\boldsymbol\beta-\boldsymbol e}(G_{\theta,\boldsymbol\alpha-\boldsymbol e} )
	}
	{ g_{\theta}^{|\boldsymbol\alpha|} }
\end{align}
Moreover, terms
\begin{align*}
	\partial_{\theta}^{\boldsymbol e}g_{\theta} \:
	P_{\boldsymbol\alpha-\boldsymbol e,\boldsymbol\beta}
	(G_{\theta,\boldsymbol\alpha-\boldsymbol e} ),
	\;\;\;\;\;
	g_{\theta} \:
	\partial_{\theta}^{\boldsymbol e}P_{\boldsymbol\alpha-\boldsymbol e,\boldsymbol\beta}
	(G_{\theta,\boldsymbol\alpha-\boldsymbol e} ),
	\;\;\;\;\;
	g_{\theta}\:
	P_{\boldsymbol\alpha-\boldsymbol e,\boldsymbol\beta-\boldsymbol e}(G_{\theta,\boldsymbol\alpha-\boldsymbol e} )
\end{align*}
are polynomial in derivatives
$\left\{ \partial_{\theta}^{\boldsymbol\gamma}g_{\theta}:
\boldsymbol\gamma\in\mathbb{N}_{0}^{d},
\boldsymbol\gamma\leq\boldsymbol\alpha \right\}$.
Apparently, the order of these polynomials is up to $|\boldsymbol\alpha-\boldsymbol e|+1=|\boldsymbol\alpha|$,
while the corresponding coefficients are independent of $\theta$
and depend only on $\boldsymbol\alpha$, $\boldsymbol\beta$.
Therefore, the right-hand side of (\ref{la2.5}) admits representation (\ref{la2.1*}).
Hence, the same holds for $\partial_{\theta}^{\boldsymbol\alpha}h_{\theta}$.

(ii) Let $\tilde{C}_{\boldsymbol\alpha,\boldsymbol\beta}$
be the maximum absolute value of the coefficients of polynomial
$P_{\boldsymbol\alpha,\boldsymbol\beta}(\cdot)$,
where $\boldsymbol\alpha,\boldsymbol\beta\in\mathbb{N}_{0}^{d}$,
$\boldsymbol\alpha\leq\boldsymbol\beta$.
As the number of different power terms in $P_{\boldsymbol\alpha,\boldsymbol\beta}(\cdot)$
is at most $m_{\boldsymbol\alpha}$, Assumption \ref{aa2.2} and (i) yield
\begin{align*}
	\left|
	\frac{\partial_{\theta}^{\boldsymbol\beta}f_{\theta}\:
	P_{\boldsymbol\alpha,\boldsymbol\beta}(G_{\theta,\boldsymbol\alpha} ) }
	{g_{\theta}^{|\boldsymbol\alpha|+1} }
	\right|
	\leq
	\tilde{C}_{\boldsymbol\alpha,\boldsymbol\beta}\:m_{\boldsymbol\alpha}
	\phi_{\theta}^{|\boldsymbol\beta|} \left|\frac{f_{\theta}}{g_{\theta}}\right|
	\left( \frac{\psi_{\theta} }{|g_{\theta}| } \right)^{|\boldsymbol\alpha|}
	\leq
	\tilde{C}_{\boldsymbol\alpha,\boldsymbol\beta}\:m_{\boldsymbol\alpha}
	\left|\frac{f_{\theta}}{g_{\theta}}\right|
	\left( \frac{\phi_{\theta}\psi_{\theta} }{|g_{\theta}| } \right)^{|\boldsymbol\alpha|}
\end{align*}
for all $\theta\in\Theta$ and any
$\boldsymbol\alpha,\boldsymbol\beta\in\mathbb{N}_{0}^{d}\setminus\{\boldsymbol 0\}$,
$\boldsymbol\beta\leq\boldsymbol\alpha$.
Then, using (i) again, we conclude that there exists a real number $K\in[1,\infty)$
with the properties specified in the lemma's statement.
\end{proof}

\refstepcounter{appendixcounter}\label{appendix3}
\section*{Appendix \arabic{appendixcounter} }

In this section, we present auxiliary results which Proposition \ref{proposition5.1}
and Theorem \ref{theorem1.1} crucially rely on.
Let $\Theta$ and $d$ have the same meaning as in Section \ref{section1}.
Moreover, let ${\cal Z}$ be a Borel set in $\mathbb{R}^{d_{z} }$,
where $d_{z}\geq 1$ is an integer.
We consider here functions $F_{\theta}(z)$ and $g_{\theta}$ mapping
$\theta\in\Theta$, $z\in{\cal Z}$ to $\mathbb{R}$ and $\mathbb{R}\setminus\{0\}$
(respectively).
We also consider non-negative measure $\mu(dz)$ on ${\cal Z}$.
The analysis carried out in this section relies on the following assumptions.

\begin{assumptionappendix}\label{aa3.1}
$F_{\theta}(z)$ and $g_{\theta}$ are $p$-times differentiable in $\theta$
for each $\theta\in\Theta$, $z\in{\cal Z}$, where $p\geq 1$.
\end{assumptionappendix}

\begin{assumptionappendix}\label{aa3.2}
There exists a function $\phi:{\cal Z}\rightarrow[1,\infty )$ such that
\begin{align*}
	|\partial_{\theta}^{\boldsymbol\alpha} F_{\theta}(z) |\leq \phi(z),
	\;\;\;\;\;
	\int \phi(z')\mu(dz') < \infty
\end{align*}
for all $\theta\in\Theta$, $z\in{\cal Z}$ and any multi-index $\boldsymbol\alpha\in\mathbb{N}_{0}^{d}$,
$|\boldsymbol\alpha|\leq p$.
\end{assumptionappendix}

Throughout this section, we use the following notation.
$f_{\theta}$, $h_{\theta}$, $H_{\theta}(z)$ are the functions defined by
\begin{align*}
	f_{\theta} = \int F_{\theta}(z') \mu(dz'),
	\;\;\;\;\;
	h_{\theta}
	=
	\frac{f_{\theta} }{g_{\theta} },
	\;\;\;\;\;
	H_{\theta}(z)
	=
	\frac{F_{\theta}(z) }{g_{\theta} }
\end{align*}
for $\theta\in\Theta$, $z\in{\cal Z}$.
$\xi_{\theta}(dz)$ and $\zeta_{\theta}(dz)$ are the signed measures on ${\cal Z}$
defined by
\begin{align*}
	\xi_{\theta}(B)
	=
	\int_{B} F_{\theta}(z) \mu(dz),
	\;\;\;\;\;
	\zeta_{\theta}(B)
	=
	\int_{B} H_{\theta}(z) \mu(dz)
\end{align*}
for $B\in{\cal B}({\cal Z} )$.
$\xi_{\theta}^{\boldsymbol\alpha}(dz)$ and $\zeta_{\theta}^{\boldsymbol\alpha}(dz)$ are the signed measures on ${\cal Z}$
defined by
\begin{align*}
	\xi_{\theta}^{\boldsymbol\alpha}(B)
	=
	\int_{B} \partial_{\theta}^{\boldsymbol\alpha} F_{\theta}(z) \mu(dz),
	\;\;\;\;\;
	\zeta_{\theta}^{\boldsymbol\alpha}(B)
	=
	\int_{B} \partial_{\theta}^{\boldsymbol\alpha} H_{\theta}(z) \mu(dz)
\end{align*}
for $\boldsymbol\alpha\in\mathbb{N}_{0}^{d}$, $|\boldsymbol\alpha|\leq p$.

\begin{lemmaappendix}\label{lemmaa3}
Let Assumptions \ref{aa3.1} and \ref{aa3.2} hold.
Then, the following is true.

(i) $f_{\theta}$ and $g_{\theta}$ are well-defined for each $\theta\in\Theta$.
Moreover, $f_{\theta}$ and $g_{\theta}$ are $p$-times differentiable and satisfy
\begin{align}\label{la3.1*}
	\partial_{\theta}^{\boldsymbol\alpha} f_{\theta}
	=
	\int \partial_{\theta}^{\boldsymbol\alpha} F_{\theta}(z) \mu(dz),
	\;\;\;\;\;
	\partial_{\theta}^{\boldsymbol\alpha} h_{\theta}
	=
	\int \partial_{\theta}^{\boldsymbol\alpha} H_{\theta}(z) \mu(dz)
\end{align}
for all $\theta\in\Theta$ and any mutli-index $\boldsymbol\alpha\in\mathbb{N}_{0}^{d}$,
$|\boldsymbol\alpha|\leq p$.

(ii) $\xi_{\theta}(B)$, $\zeta_{\theta}(B)$,
$\xi_{\theta}^{\boldsymbol\alpha}(B)$ and $\zeta_{\theta}^{\boldsymbol\alpha}(B)$ are well-defined
for each $\theta\in\Theta$, $B\in{\cal B}({\cal Z} )$.
Moreover, $\xi_{\theta}(B)$ and $\zeta_{\theta}(B)$ are $p$-times differentiable (in $\theta$) and satisfy
\begin{align}\label{la3.3*}
	\partial_{\theta}^{\boldsymbol\alpha} \xi_{\theta}(B)
	=
	\xi_{\theta}^{\boldsymbol\alpha}(B),
	\;\;\;\;\;
	\partial_{\theta}^{\boldsymbol\alpha} \zeta_{\theta}(B)
	=
	\zeta_{\theta}^{\boldsymbol\alpha}(B)
\end{align}
for all $\theta\in\Theta$, $B\in{\cal B}({\cal Z} )$ and any mutli-index $\boldsymbol\alpha\in\mathbb{N}_{0}^{d}$,
$|\boldsymbol\alpha|\leq p$.
\end{lemmaappendix}

\begin{proof}
Let $\theta$ be any element of $\Theta$,
while $\boldsymbol\alpha$ is any multi-index in $\mathbb{N}_{0}^{d}$
satisfying $|\boldsymbol\alpha|\leq p$.
Moreover, let $z$ be any element of ${\cal Z}$,
while $B$ is any element of ${\cal B}({\cal Z})$.
Owing to Assumptions \ref{aa3.1}, \ref{aa3.2},
$f_{\theta}$, $\xi_{\theta}(B)$, $\xi_{\theta}^{\boldsymbol\alpha}(B)$
are well-defined.
Consequently, $h_{\theta}$, $\zeta_{\theta}(B)$ are also well-defined.
Moreover, due to the dominated convergence theorem and Assumptions \ref{aa3.1}, \ref{aa3.2},
$f_{\theta}$, $\xi_{\theta}(B)$ are $p$-times differentiable in $\theta$ on $\Theta$
and satisfy the first part of (\ref{la3.1*}), (\ref{la3.3*}).
Therefore, $h_{\theta}$, $\zeta_{\theta}(B)$ are also $p$-times differentiable in $\theta$ on $\Theta$.

Using Lemma \ref{lemmaa2}, we conclude that
$\partial_{\theta}^{\boldsymbol\alpha} h_{\theta}$,
$\partial_{\theta}^{\boldsymbol\alpha} H_{\theta}(z)$,
$\partial_{\theta}^{\boldsymbol\alpha}\zeta_{\theta}(B)$ admit
the following representation:
\begin{align}\label{la3.1}
	&
	\partial_{\theta}^{\boldsymbol\alpha} h_{\theta}
	=
	\sum_{\stackrel{\scriptstyle \boldsymbol\beta\in\mathbb{N}_{0}^{d} }
	{\boldsymbol\beta\leq\boldsymbol\alpha } }
	\frac{G_{\theta}^{\boldsymbol\alpha,\boldsymbol\beta}
	\partial_{\theta}^{\boldsymbol\beta} f_{\theta} }
	{(g_{\theta} )^{|\boldsymbol\alpha|+1} },
	\;\;\;\;\;
	\partial_{\theta}^{\boldsymbol\alpha} H_{\theta}(x)
	=
	\sum_{\stackrel{\scriptstyle \boldsymbol\beta\in\mathbb{N}_{0}^{d} }
	{\boldsymbol\beta\leq\boldsymbol\alpha } }
	\frac{G_{\theta}^{\boldsymbol\alpha,\boldsymbol\beta}
	\partial_{\theta}^{\boldsymbol\beta} F_{\theta}(z) }
	{(g_{\theta} )^{|\boldsymbol\alpha|+1} },
	\;\;\;\;\;
	\partial_{\theta}^{\boldsymbol\alpha}\zeta_{\theta}(B)
	=
	\sum_{\stackrel{\scriptstyle \boldsymbol\beta\in\mathbb{N}_{0}^{d} }
	{\boldsymbol\beta\leq\boldsymbol\alpha } }
	\frac{G_{\theta}^{\boldsymbol\alpha,\boldsymbol\beta}
	\partial_{\theta}^{\boldsymbol\beta} \xi_{\theta}(B) }
	{(g_{\theta} )^{|\boldsymbol\alpha|+1} },
\end{align}
where $G_{\theta}^{\boldsymbol\alpha,\boldsymbol\beta}$ is a polynomial function of
derivatives $\big\{\partial_{\theta}^{\boldsymbol\gamma}g_{\theta}: \boldsymbol\gamma\in\mathbb{N}_{0},
\boldsymbol\gamma\leq\boldsymbol\alpha \big\}$.
Owing to (\ref{la3.1}) and the first part of (\ref{la3.1*}), we have
\begin{align*}
	\partial_{\theta}^{\boldsymbol\alpha} h_{\theta}
	=
	\sum_{\stackrel{\scriptstyle \boldsymbol\beta\in\mathbb{N}_{0}^{d} }
	{\boldsymbol\beta\leq\boldsymbol\alpha } }
	\frac{G_{\theta}^{\boldsymbol\alpha,\boldsymbol\beta} }
	{(g_{\theta} )^{|\boldsymbol\alpha|+1} }
	\int \partial_{\theta}^{\boldsymbol\beta} F_{\theta}(z) \mu(dz)
	=
	\int \partial_{\theta}^{\boldsymbol\alpha} H_{\theta}(z) \mu(dz).
\end{align*}
Similarly, due to (\ref{la3.1}) and the first part of (\ref{la3.3*}), we have
\begin{align*}
	\partial_{\theta}^{\boldsymbol\alpha} \zeta_{\theta}(B)
	=
	\sum_{\stackrel{\scriptstyle \boldsymbol\beta\in\mathbb{N}_{0}^{d} }
	{\boldsymbol\beta\leq\boldsymbol\alpha } }
	\frac{G_{\theta}^{\boldsymbol\alpha,\boldsymbol\beta} }
	{(g_{\theta} )^{|\boldsymbol\alpha|+1} }
	\int_{B} \partial_{\theta}^{\boldsymbol\beta} F_{\theta}(z) \mu(dz)
	=
	\int_{B} \partial_{\theta}^{\boldsymbol\alpha} H_{\theta}(z) \mu(dz)
	=
	\zeta_{\theta}^{\boldsymbol\alpha}(B).
\end{align*}
Hence, $\partial_{\theta}^{\boldsymbol\alpha} h_{\theta}$, $\zeta_{\theta}^{\boldsymbol\alpha}(B)$
are well-defined and satisfy the second part of (\ref{la3.1*}), (\ref{la3.3*}).
\end{proof}

\end{document}